\documentclass{article}
\usepackage{lmodern}
\usepackage{amsmath}
\usepackage{amsfonts}
\usepackage{amssymb}
\usepackage{graphicx}%
\providecommand{\U}[1]{\protect\rule{.1in}{.1in}}
\newtheorem{theorem}{Theorem}

\newtheorem{claim}[theorem]{Claim}

\newtheorem{corollary}[theorem]{Corollary}

\newtheorem{definition}[theorem]{Definition}

\newtheorem{lemma}[theorem]{Lemma}

\newtheorem{problem}[theorem]{Problem}
\newtheorem{proposition}[theorem]{Proposition}

\newenvironment{proof}[1][Proof]{\noindent\textbf{#1.} }{\ \rule{0.5em}{0.5em}}
\begin{document}

\date{}
\title{Square-bracket operations clubs}
\author{Osvaldo Guzm\'{a}n \thanks{\textit{keywords: }Square-bracket operations, walks
on ordinals, Proper Forcing Axiom, Martin%
\'{}%
s axiom, Continuum Hypothesis. \newline\textit{AMS Classification:} 03E05,
3E50, 03E57, 03E35.}
\and Stevo Todorcevic \thanks{The first author was supported by the PAPIIT grants
IA 104124, IN105926and the SECIHTI grant CBF2023-2024-903.}}
\maketitle

\begin{abstract}
This paper continues the investigation of the three square-bracket operations
$[\cdot\cdot]$ from chapter 5 of \cite{Walks}. \ We say that a square-bracket
operation $[\cdot\cdot]$ has the \emph{Ramsey club property} if for every club
$C\subseteq\omega_{1}$, there is an uncountable subset $W$ $\subseteq
\omega_{1}$ such that $\left[  \alpha\beta\right]  \in C$ for every
$\alpha,\beta\in W.$ \ The second author proved that the Proper Forcing
Axiom\textsf{ }implies that all the square-bracket operations induced by
Aronszajn trees have this property. We extend this result to the other two
classes. We conclude that each of the statements \textquotedblleft all
square-bracket operations have the Ramsey club property\textquotedblright\ and
\textquotedblleft No square-bracket operation has the Ramsey club
property\textquotedblright\ are consistent with \textsf{ZFC. }In other words,
\textsf{ZFC }is unable to decide the status of the Ramsey club property for
any square-bracket operation. Furthermore, we analyze the status of the Ramsey
club property for square-bracket operations under Martin's Axiom and the
Continuum Hypothesis.

\end{abstract}

\section{Introduction}

The famous Ramsey's theorem states that every graph on $\omega$ contains
either a countably infinite complete subgraph or a countably infinite
independent subgraph\footnote{Recall that a graph is complete (independent) if
any two vertices are (not) connected.}. The straightforward generalization to
$\omega_{1}$ would assert that every graph on $\omega_{1}$ contains either an
uncountable complete subgraph or an uncountable independent subgraph. However,
as noted by Sierpi\'{n}ski, this generalization is false. In fact, it in a
spectacular way, as can be seen by the following theorem of the second author:

\begin{theorem}
[T. \cite{Walks}, \cite{PartitioningPairs}]There is $c:\left[  \omega
_{1}\right]  ^{2}\longrightarrow\omega_{1}$ such that for every\newline%
$A\in\left[  \omega_{1}\right]  ^{\omega_{1}}$ and $\gamma\in\omega_{1}$,
there are $\alpha,\beta\in A$ such that $c\left(  \alpha,\beta\right)
=\gamma.$ \label{Falla Ramsey}
\end{theorem}

\qquad\qquad\qquad\ \ \qquad\ \ \ \ \qquad\ \ \ \ \ \ \ \ \ 

In other words, the pairs of $\omega_{1}$ can be colored with uncountably many
colors in such a way that every color appears in every uncountable set (for
more results on this type, see \cite{Lspace} or
\cite{RectangularSquareBracket}). Given a function $F:\left[  \omega
_{1}\right]  ^{2}\longrightarrow\omega_{1}$ and a subset $X\subseteq\omega
_{1},$ define the graph $G\left(  F,X\right)  $ on $\omega_{1}$ by connecting
$\alpha$ and $\beta$ with an edge in case $F\left(  \alpha,\beta\right)  \in
X.$ This family of graphs can be very interesting when $F$ has strong
combinatorial properties. In chapter 5 of the book \cite{Walks}, the
\emph{square-bracket operations} $[\cdot\cdot]_{\mathcal{C}},$ $[\cdot
\cdot]_{T,a}$ and $[\cdot\cdot]_{\mathcal{R},e}$ are introduced, where
$\mathcal{C}$ is a $C$-sequence, $T$ is a special Aronszajn tree with an
specializing mapping $a$, $\mathcal{R}$ is an $\omega_{1}$-sequence of reals
and $e$ is a sequence of mappings. The definitions and main properties of
these operations will be reviewed in the Preliminaries Section. In this paper,
the term \textquotedblleft square-bracket operation\textquotedblright\ will
always refer to one of these three functions. The following is another result
of the second author:

\begin{theorem}
[T. \cite{Walks}]Let $F$ be a square-bracket operation and $S,Z\subseteq
\omega_{1}.$ If the symmetric difference $S\triangle Z$ is stationary, then
the graphs $G\left(  F,S\right)  $ and $G\left(  F,Z\right)  $ are orthogonal
to each other (they do not contain uncountable isomorphic subgraphs).
\end{theorem}

\qquad\qquad\ \qquad\ \ 

With this result, the second author was able to solve the \emph{basis problem
for uncountable graphs:}\footnote{This result should be contrasted with the
following well-known consequences of \textsf{PFA: }Any two $\omega_{1}$-dense
subsets of reals are isomorphic (see
\cite{Allw1DenseSetsofRealscanbeIsomorphic} or \cite{ApplicationsofPFA}),
there is a five element base for the uncountable linear orders (see
\cite{FiveElementBasis}) and there are exactly five cofinal types of directed
partial orders of size $\omega_{1}$ (see \cite{5CofinalTypes}).}\emph{
}$2^{\omega_{1}}$ is\emph{ }the minimum size of a family $\mathcal{G}$ of
uncountable graphs such that every uncountable graph contains an isomorphic
copy of some member of $\mathcal{G}.$ We believe this result justifies a
deeper investigation into the graphs induced by the square-bracket operations,
which is the central topic of the present paper. We begin with the following definition:

\begin{definition}
Let $F:\left[  \omega_{1}\right]  ^{2}\longrightarrow\omega_{1}$ and
$X\subseteq\omega_{1}.$ We say that $X$ is $F$\emph{-Ramsey }if $G\left(
F,X\right)  $ contains either an uncountable complete subgraph or an
uncountable independent subgraph. \label{Def F Ramsey}
\end{definition}

\qquad\ \ \qquad\ \ 

We will investigate which subsets of $\omega_{1}$ are Ramsey for the
square-bracket operations. The following is result is proved in \cite{Walks}.

\begin{theorem}
[T.]Let $[\cdot\cdot]$ be a square-bracket operation and $W\in\left[
\omega_{1}\right]  ^{\omega_{1}}.$ The set $\left\{  \left[  \alpha
\beta\right]  \mid\alpha,\beta\in W\right\}  $ contains a club.
\label{sq bracket general contiene club}
\end{theorem}

\qquad\ \ 

Note that this result implies Theorem \ref{Falla Ramsey}. Another easy
corollary is the following:

\begin{corollary}
Let $[\cdot\cdot]$ be a square-bracket operation and $X\subseteq\omega_{1}.$
\label{Corol 1}

\begin{enumerate}
\item If $G([\cdot\cdot],X)$ contains an uncountable complete subgraph, then
$X$ contains a club.

\item If $G([\cdot\cdot],X)$ contains an uncountable independent subgraph,
then $X$ is nonstationary.
\end{enumerate}
\end{corollary}

\qquad\qquad\qquad\ \ \ 

Note that Theorem \ref{sq bracket general contiene club} implies that if
$S\subseteq\omega_{1}$ is stationary (in particular, if it is a club), then
$G([\cdot\cdot],S)$ does not contain an uncountable independent subgraph. This
allows us to reformulate Definition \ref{Def F Ramsey} for clubs as follows:

\begin{definition}
Let $[\cdot\cdot]$ be a square-bracket operation and $C\subseteq\omega_{1}$ a
club. We say that $C$ is $[\cdot\cdot]$\emph{-Ramsey }if $G([\cdot\cdot],C)$
contains an uncountable complete subgraph.
\end{definition}

\qquad\qquad\qquad\ \ \ \ \qquad\ \ \ 

By the first point of Corollary \ref{Corol 1}, the problem of identifying the
subsets $X\subseteq\omega_{1}$ for which $G([\cdot\cdot],X)$ contains an
uncountable complete subgraph can be reduced to the case where $X$ is a club.
Moreover, since the question of which subsets yield a graph with an
uncountable independent subgraph is dual to the question for complete
subgraphs, we are led to the following problem:\newline

\begin{center}%
\begin{tabular}
[c]{l}%
\emph{Which clubs are}\textit{ }$[\cdot\cdot]$\emph{-Ramsey?}%
\end{tabular}

\end{center}

The first result in this direction can be found in the book \cite{Walks}
(Lemma 5.4.4):

\begin{theorem}
[T. \cite{Walks}]The \emph{Proper Forcing Axiom} (\textsf{PFA}) implies that
if $T$ is a special Aronszajn tree and $a:T\longrightarrow\omega$ is an
specialization mapping, then every club is $[\cdot\cdot]_{T,a}$-Ramsey.
\label{Teorema PFA Aronszajn}
\end{theorem}

\qquad\qquad\qquad\ \ \ \qquad\ 

This theorem answered a question raised by Woodin while he was working with
his forcing $\mathbb{P}_{\text{\textsf{max}}}$ (see \cite{PmaxBook}). The
following questions motivated our research on this topic:

\begin{enumerate}
\item Is the conclusion of Theorem \ref{Teorema PFA Aronszajn} independent of
\textsf{ZFC? }What is its status under \textsf{MA }or \textsf{CH}?

\item Does Theorem \ref{Teorema PFA Aronszajn} hold for the other
square-bracket operations?
\end{enumerate}

\qquad\ \ 

This work addresses these questions. We will show that Theorem
\ref{Teorema PFA Aronszajn} holds for the other square-bracket operations and
that its conclusion is independent from \textsf{ZFC }(for all square-bracket
operations).\textsf{ }In fact, for some of the square-bracket operations, the
conclusion is even independent from \textsf{MA }and its negation is consistent
with \textsf{CH}. The most important open question remaining is whether
\textsf{CH }is consistent with the statement that every club is $[\cdot\cdot]$-Ramsey.

\qquad\qquad\ \ \ 

The structure of the paper is as follows. Section \ref{Secion notacion} covers
the necessary notation and definitions that will be used throughout the paper.
The next five sections provide the required background on elementary
submodels, Aronszajn trees, walks on ordinals, forcing and Baumgartner's club
forcing. The reader may skip these preliminary sections and return to them as
needed. The definitions of the square-bracket operations from the book
\cite{Walks} are reviewed in Section \ref{The operations}. In Section
\ref{Finding Ramsey clubs} we prove that for every club $C\subseteq\omega_{1}%
$, there exists a square-bracket operation $[\cdot\cdot]$ (of any type) for
which $C$ is $[\cdot\cdot]$-Ramsey. On the other hand, in Section
\ref{Destroying square} we prove that the generic club added by Baumgartner's
club forcing is not $[\cdot\cdot]$-Ramsey for any square-bracket operation
$[\cdot\cdot]$ from the ground model. Consequently, \textsf{ZFC }can not prove
the conclusion of Theorem \ref{Teorema PFA Aronszajn}. Furthermore, we prove
that \textsf{MA }is consistent with the existence of a square-bracket
operation induced by a \textsf{C}-sequence for which not all clubs are Ramsey.
In Section \ref{Taking care of all} we extend these results for all
square-bracket operations and not just the ones from the ground model. In
Sections \ref{Seccion PFA reals} and \ref{Seccion PFA walks} we extend Theorem
\ref{Teorema PFA Aronszajn} for the remaining square-bracket operations. The
final section presents a collection of open questions.

\section{Notation \label{Secion notacion}}

We say that $\left(  T,<_{T}\right)  $ is a \emph{tree} if it is a partial
order with a least element and for every $s\in T,$ the set $s^{\downarrow
}=\left\{  t\in T\mid t<_{T}s\right\}  $ is a well-order. Given an ordinal
$\alpha,$ by $T_{\alpha}$ we denote the $\alpha$-level of $T,$ which is the
set of all $s\in T$ for which $s^{\downarrow}$ is isomorphic to $\alpha.$ The
\emph{height of }$T$ is the least $\alpha$ for which $T_{\alpha}$ is empty.
For $s\in T,$ the \emph{height of }$s$ is the (unique) ordinal $\alpha$ such
that $s\in T_{\alpha}.$ The height of $s$ will be denoted as $\left\vert
s\right\vert $ (although we are using the same symbol that is used to denote
cardinality, there should not be any room for confusion). If $\gamma
\leq\left\vert s\right\vert ,$ denote by $s\upharpoonright\gamma$ the only
element of $T_{\gamma}$ for which $s\upharpoonright\gamma\leq_{T}s$ (so
$s\upharpoonright\left\vert s\right\vert =s$). For $C\subseteq\omega_{1},$
denote $T\upharpoonright C=\left\{  s\in T\mid\left\vert s\right\vert \in
C\right\}  .$ We say that $T$ is a \emph{Hausdorff tree }if for every limit
ordinal $\alpha$ and $s,t\in T_{\alpha}$, if $s^{\downarrow}=t^{\downarrow},$
then $s=t$ (equivalently, the tree topology of $T$ is Hausdorff). We say that
$s,t\in T$ are \emph{incompatible }(denoted by $s\perp t$) if $s\nleqslant
_{T}t$ and $t\nleqslant_{T}s.$ Denote by $\triangle\left(  s,t\right)  $ the
largest $\alpha$ for which $s\upharpoonright\alpha=t\upharpoonright\alpha.$ A
branch through a tree is a maximal linearly ordered subset.

\qquad\ \qquad\ \qquad\ 

Although our main focus is on trees of height $\omega_{1},$ we will also
encounter subtrees of $2^{<\omega}.$ Given $s,t\in2^{<\omega},$ the
concatenation of $s$ and $t$ is denoted by $s^{\frown}t.$ We will occasionally
denote $s^{\frown}n$ instead of $s^{\frown}\left(  n\right)  $ (for some
$n\in\omega$). The \emph{cone of }$s$ is defined as $\left\langle
s\right\rangle =\left\{  f\in2^{\omega}\mid s\subseteq f\right\}  .$ Let
$S\subseteq2^{<\omega}$ be a tree (the tree ordering is $\subseteq$). We say
that $s\in S$ is a \emph{splitting node }if both $s^{\frown}0$ and $s^{\frown
}1$ belong to $S.$ If $s$ is a splitting node of $S$ and it extends exactly
$n$-many splitting nodes, we will say $s$ is an $n$\emph{-splitting node of}
$S.$ The set of all $n$-splitting nodes of $S$ is denoted by \textsf{split}%
$_{n}\left(  S\right)  .$ A \emph{Sacks tree (or perfect tree)} is a tree in
which any node can be extended to a splitting node. The \emph{set of branches
of }$S$ is defined as $\left[  S\right]  =\left\{  f\in2^{\omega}\mid\forall
n\in\omega\left(  f\upharpoonright n\in S\right)  \right\}  .$

\qquad\ \qquad\ \ \ 

By \textsf{LIM}$\left(  \omega_{1}\right)  $ we denote the set of all
countable limit ordinals. Given $S\subseteq\omega_{1},$ its set of
\emph{accumulation points }(denoted by $S^{\prime}$) is the set of all
$\alpha\in\omega_{1}$ such that $S\cap\alpha$ is unbounded in $\alpha$. By
$\left[  \alpha,\beta\right]  $ we denote the closed interval $\left\{
\xi\mid\alpha\leq\xi\leq\beta\right\}  .$ This should not be confused with the
square-bracket operations, which are always denoted by $\left[  \alpha
\beta\right]  $ (with no comma). The intervals $\left(  \alpha,\beta\right)  $
and $(\alpha,\beta]$ have the expected meaning. The expression $\left\langle
\delta_{n}\right\rangle \longrightarrow\delta$ means that the sequence
$\left\langle \delta_{n}\right\rangle $ converges to $\delta$ (with the usual
order topology of $\omega_{1}$). The class of ordinals is denoted by
\textsf{OR.}

\qquad\ \qquad\ \qquad\ \ \ 

The quantifier $\exists^{\omega_{1}}$ means \textquotedblleft there are
uncountably many\textquotedblright\ while $\forall^{\infty}$ means
\textquotedblleft for almost all\textquotedblright\ (for all except finitely
many). For $A$ a set of the ordinal numbers, we will denote by \textsf{OT}%
$\left(  A\right)  $ its order type. If $f$ is a function, by \textsf{dom}%
$\left(  f\right)  $ we denote its domain and \textsf{im}$\left(  f\right)  $
is its image.

\section{Preliminaries on countable elementary submodels}

Elementary submodels play a fundamental role while studying the combinatorics
of $\omega_{1}.$ We now review the basic facts about them that are needed for
our work and refer the reader to the survey \cite{AlanSubmodelos} to learn
more. For a cardinal $\kappa$, define \textsf{H}$\left(  \kappa\right)  $ as
the set consisting of all sets whose transitive closure has size less than
$\kappa.$ Arguments involving elementary submodels, often start by choosing a
large enough regular $\kappa$ that is needed for the argument in question. To
avoid constant repetition, we will omit the reference to $\kappa$ when there
is no need to explicitly specify it. In this way, we will simply write
\textquotedblleft let $M$ be a countable elementary submodel\textquotedblright%
\ instead of \textquotedblleft let $M$ be a countable elementary submodel of
some \textsf{H}$\left(  \kappa\right)  ,$ for a large enough regular $\kappa
$\textquotedblright. It is often the case that the reference to $\kappa$ is
not needed. However, this is not the case when working with models as side
conditions (see \cite{PartitionProblems} and \cite{NotesonForcingAxioms}), as
we will do in sections \ref{Seccion PFA reals} and \ref{Seccion PFA walks}. In
this two sections, we will be more carefully naming the cardinal that is needed.

\qquad\qquad\qquad\qquad

Let $M$ be a countable elementary submodel. The \emph{height of }$M$ is
defined as $\delta_{M}=M\cap\omega_{1}.$ It is well-known that $\delta_{M}$ is
a countable ordinal. We will need the following result regarding the height of
a model:

\begin{proposition}
Let $M$ be a countable elementary submodel and $C,S\in M$ be two subsets of
$\omega_{1}.$ \label{Prop submodelos club y est}

\begin{enumerate}
\item If $C$ is a club, then $\delta_{M}\in C.$

\item If $\delta_{M}\in S,$ then $S$ is stationary.
\end{enumerate}
\end{proposition}

\qquad\ \qquad\ \qquad\ 

Another very useful result is the following:

\begin{lemma}
Let $M$ be a countable elementary submodel and $A\in M.$ If $A\nsubseteq M,$
then $A$ is uncountable. \label{Lema modelo contension}
\end{lemma}

\qquad\ \qquad\ \ 

Recall the following notion:

\begin{definition}
Let $\delta$ be an ordinal. We say that $\delta$ is \emph{indecomposable }if
it is closed under (ordinal) addition. Equivalently, $\gamma+\delta=\delta$
for every $\gamma<\delta.$ By \textsf{IND}$\left(  \omega_{1}\right)  $ we
denote the set of all countable indecomposable ordinals.
\end{definition}

\qquad\qquad\qquad

The height of a model is an indecomposable ordinal (it is also closed under
ordinal multiplication and exponentiation). It is not hard to see that
\textsf{IND}$\left(  \omega_{1}\right)  $ is a club.

\begin{definition}
Let $\gamma\leq\omega_{1}.$ We say that $\left\{  M_{\alpha}\mid\alpha
<\gamma\right\}  $ is a \emph{continuos chain of elementary submodels} if the
following conditions hold:

\begin{enumerate}
\item There is an uncountable regular $\kappa$ such that each $M_{\alpha}$ is
an elementary submodel of \textsf{H}$\left(  \kappa\right)  .$

\item If $\alpha<\beta<\gamma,$ then $M_{\alpha}\in M_{\beta}.$

\item If $\alpha<\gamma$ is a limit ordinal, then $M_{\alpha}=\bigcup
\limits_{\xi<\alpha}M_{\xi}.$
\end{enumerate}
\end{definition}

\section{Preliminaries on Aronszajn trees}

An Aronszajn tree is a tree of height $\omega_{1}$ such that all of its levels
are countable and has no uncountable branches. A large part of the
combinatorics of $\omega_{1}$ deals with Aronszajn trees. The reader may
consult the survey \cite{StevoHandbook} to learn a lot on Aronszajn trees (see
also \cite{BaumgartnerTesis} and \cite{Kunen}). An Aronszajn tree $T$ \emph{is
special }if there is $a:T\longrightarrow\omega$ such that for every
$n\in\omega,$ the set $a^{-1}\left(  \left\{  n\right\}  \right)  $ is an
antichain. We call $a$ \emph{an specialization mapping for }$T.$

\qquad\qquad\qquad\ \ \ \ 

Let $T^{1},...,T^{n}$ be trees of height $\omega_{1}.$ The \emph{tree product
}$S=T^{0}\otimes...\otimes T^{n}$ consists of all sequences $\overline
{s}=\left(  s_{0},...,s_{n}\right)  \in T^{0}\times...\times T^{n}$ such that
$\left\vert s_{0}\right\vert =...=\left\vert s_{n}\right\vert .$ For
$\overline{s},\overline{t}\in S,$ define $\overline{s}\leq_{S}\overline{t}$ if
$s_{i}\leq_{T^{i}}t_{i}$ for every $i\leq n.$ It is easy to see verify that
$\left(  S,\leq_{S}\right)  $ is a tree of height $\omega_{1}$ and $S_{\delta
}$ is precisely the set of all $\overline{s}=\left(  s_{0},...,s_{n}\right)
\in S$ with $\left\vert s_{0}\right\vert =...=\left\vert s_{n}\right\vert
=\delta.$ If $\overline{s}=\left(  s_{0},...,s_{n}\right)  \in S$ and
$\delta<\left\vert \overline{s}\right\vert ,$ we have that $\overline
{s}\upharpoonright\delta=\left(  s_{0}\upharpoonright\delta,...,s_{n}%
\upharpoonright\delta\right)  \in S_{\delta}$ (we take this opportunity to
remind the reader that in this context, $\left\vert \overline{s}\right\vert $
denotes the height of $s$ in the tree $S,$ not its cardinality as a sequence).
A set $A=\{\overline{s}_{\alpha}\mid\alpha\in\omega_{1}\}\subseteq S$ is
\emph{height increasing }if $\left\vert \overline{s}_{\alpha}\right\vert
<\left\vert \overline{s}_{\beta}\right\vert $ for all $\alpha<\beta$.

\begin{definition}
Let $T^{1},...,T^{n}$ be trees of height $\omega_{1},$ $S=T^{0}\otimes
...\otimes T^{n},$ $A=\{\overline{s}_{\alpha}\mid\alpha\in\omega
_{1}\}\subseteq S,$ $k\in\omega$ and $\delta\in\omega_{1}.$ We say that $A$
\emph{is }$k$\emph{-distributed at level} $\delta$ if for every $\overline
{u}_{0},...,\overline{u}_{k}\in S_{\delta},$ there are uncountably many
$\alpha\in\omega_{1}$ such that $\overline{s}_{\alpha}\upharpoonright
\delta\cap\overline{u}_{i}=\emptyset$ for all $i\leq k.$
\end{definition}

\qquad\ \ \ \ \ \ \ \ 

The following lemma from \cite{StevoHandbook} (Lemma 5.9) is particularly
useful when working with Aronszajn trees.

\begin{proposition}
Let $T^{1},...,T^{n}$ be trees of height $\omega_{1},$ $S=T^{0}\otimes
...\otimes T^{n}$ and $A=\{\overline{s}_{\alpha}\mid\alpha\in\omega
_{1}\}\subseteq S$ that is height increasing. For every $k\in\omega$, there is
$\delta\in\omega_{1}$ such that $A\ $is $k$-distributed at level $\delta
.$\label{Prop Aronszajn distribuido}
\end{proposition}

\qquad\ \qquad\ \qquad\ \ 

A consequence of this result is the following proposition (an alternative
proof can be found in Claim 3.2 of \cite{RectangleRefiningProperty}).

\begin{proposition}
Let $T$ be an Aronszajn tree and $A$ an uncountable subset of $T.$ There are
two incomparable nodes $s,t\in T$ such that the sets $\left\{  u\in T\mid
s<_{T}u\right\}  $ and $\left\{  u\in T\mid t<_{T}u\right\}  $ are
uncountable. \label{Prop Aronszajn no numerable}
\end{proposition}

\qquad\qquad\ \ \qquad\ \ \ 

An \textquotedblleft elementary submodel\textquotedblright\ version of the
preceding proposition is also true.

\begin{proposition}
Let $T$ be an Aronszajn tree, $A\subseteq T,$ $M$ a countable elementary
submodel with $T,A\in M$ and $t\in A\setminus M.$ There are $\xi\in M$ and
$u\in T\upharpoonright\xi$ such that: \label{Prop Aronszajn y submodelo}

\begin{enumerate}
\item $t\perp u.$

\item There are uncountably many $s\in A$ such that $u\leq_{T}s.$
\end{enumerate}
\end{proposition}

\section{Preliminaries on forcing}

We assume the reader is already familiar with the method of forcing as
presented in \cite{oldKunen}. \ If $\mathbb{P}$ is a forcing and $M$ is a
countable elementary submodel with $\mathbb{P}\in M,$ we say that
$p\in\mathbb{P}$ is an $(M,\mathbb{P})$\emph{-generic condition }if for every
$D\subseteq\mathbb{P}$ open dense set with $D\in M,$ the set $D\cap M$ is
predense below $p.$ Moreover, $p$ is an $(M,\mathbb{P})$\emph{-totally generic
condition }if $p\in D$ for every $D\in M$ that is an open dense subset of
$\mathbb{P}.$ It is well-known that if $p$ is $(M,\mathbb{P})$-generic and
$G\subseteq\mathbb{P}$ is a generic filter with $p\in G,$ then $M\left[
G\right]  $ is a generic extension of $M$ and $M\left[  G\right]  \cap V=M$.
The following equivalence of generic conditions is very useful (see
\cite{NotesonForcingAxioms}).

\begin{lemma}
Let $\mathbb{P}$ be a forcing, $p\in\mathbb{P}$ and $M$ an elementary submodel
with $\mathbb{P}\in M.$ The following are equivalent:
\label{equiv condicion generica}

\begin{enumerate}
\item $p$ is an $(M,\mathbb{P})$-generic condition.

\item For every $A\subseteq\mathbb{P}$ with $A\in M$ and $q\leq p,$ if $q\in
A,$ then there is $r\in A\cap M$ that is compatible with $q.$
\end{enumerate}
\end{lemma}

\qquad\ \ \ \ 

A forcing $\mathbb{P}$ is \emph{(totally) proper }if for every countable
elementary submodel $M$ with $\mathbb{P}\in M$ and $p\in M\cap\mathbb{P},$
there is $q\leq p$ that is an $(M,\mathbb{P})$\emph{-(totally) generic
condition. }It can be proved a forcing is totally proper if and only if it is
proper and does not add new reals (see \cite{IteratedForcingandCH}).

\qquad\ \qquad\ \qquad\ \ \ 

We say that $\mathbb{P}$ has the $\kappa$\emph{-chain condition }($\mathbb{P}$
is $\kappa$\textsf{-c.c.}) if $\mathbb{P}$ does not have antichains of size
$\kappa.$ We write \textsf{c.c.c. }instead of $\omega_{1}$-\textsf{c.c. }The
following is well-known (see \cite{forcingwithoutcombinatorics} or
\cite{NotesonForcingAxioms}).

\begin{proposition}
Let $\mathbb{P}$ be \textsf{c.c.c. }and $M$ a countable elementary submodel
with $\mathbb{P}\in M.$ Every $p\in\mathbb{P}$ is $(M,\mathbb{P})$-generic.
\label{prop ccc gen}
\end{proposition}

\qquad\ \qquad\ \ 

In regard with two step iterations, we have the following result (see
\cite{AbrahamHandbook}).

\begin{proposition}
Let $\mathbb{P}$ be a forcing notion, $\mathbb{\dot{Q}}$ a $\mathbb{P}$-name
for a partial order and $M$ a countable elementary submodel with
$\mathbb{P},\mathbb{\dot{Q}}\in M.$ For $(p,\dot{q})\in\mathbb{P}%
\ast\mathbb{\dot{Q}},$ the following are equivalent:
\label{prop iteracion gen}

\begin{enumerate}
\item $(p,\dot{q})$ is $(M,\mathbb{P}\ast\mathbb{\dot{Q}})$-generic.

\item $p$ is $(M,\mathbb{P})$-generic and $p\Vdash$\textquotedblleft$\dot{q}$
is $(M[\dot{G}],\mathbb{\dot{Q}}[\dot{G}])$-generic\textquotedblright\ (where
$\dot{G}$ is the canonical $\mathbb{P}$-name for the generic filter).
\end{enumerate}
\end{proposition}

\qquad\ \qquad\ \ \ \qquad\ \ 

Recall the following notion:

\begin{definition}
Let $\mathbb{P}$ be a suborder of $\mathbb{Q}$. We say that $\mathbb{P}$
\emph{is a regular suborder of }$\mathbb{Q}$ if the following conditions hold:

\begin{enumerate}
\item If $p\perp_{\mathbb{P}}r,$ then $p\perp_{\mathbb{Q}}r.$

\item For every $q\in\mathbb{Q},$ there is $p\in\mathbb{P}$ such that for
every $r\in\mathbb{P},$ if $r\leq p,$ then $r$ is compatible with $q.$ In this
situation, we say that $p$ is a reduction of $q$ to $\mathbb{P}.$
\end{enumerate}
\end{definition}

\qquad\ \qquad\ \ \ \ 

If $\mathbb{P}$ is a regular suborder of $\mathbb{Q},$ then forcing with
$\mathbb{Q}$ adds a $(V,\mathbb{P})$-generic filter. We will say that
$\mathbb{P}$ and $\mathbb{Q}$ are \emph{forcing equivalent }if they give rise
to the same generic extensions.

\begin{definition}
Let $\mathbb{P}$ and $\mathbb{Q}$ be partial orders. We say that
$f:\mathbb{P\longrightarrow Q}$ \emph{is a dense embedding} if for every
$p,r\in\mathbb{P},$ the following conditions hold:

\begin{enumerate}
\item If $p\leq r,$ then $f\left(  p\right)  \leq f\left(  r\right)  .$

\item If $p\perp r,$ then $f\left(  p\right)  \perp f\left(  r\right)  .$

\item The image of $f$ is dense in $\mathbb{Q}.$
\end{enumerate}
\end{definition}

\qquad\qquad\ \qquad\ \ \ 

It is well-known that if there is a dense embedding between two partial
orders, then they are forcing equivalent. See \cite{Kunen} for more
information on regular suborders and dense embeddings.

\qquad\ \qquad\ \ \ \ \ \ \ \qquad\ \ \ 

The \emph{Martin number }$\mathfrak{m}$ is the smallest cardinal $\kappa$ for
which there exists a \textsf{c.c.c. }forcing $\mathbb{P}$ and a family
$\mathcal{D}$ of $\kappa$ dense subsets of $\mathbb{P}$ such that there is no
filter $G\subseteq\mathbb{P}$ that intersects every member of $\mathcal{D}.$
\textsf{Martin's Axiom} (\textsf{MA}) is the statement that $\mathfrak{m=c}$
(where $\mathfrak{c}$ denotes the size of the continuum). The following is a
very well-known result (see Theorem V.4.1 of \cite{Kunen}).

\begin{theorem}
Let $\lambda>\omega$ be a regular cardinal such that $2^{<\lambda}=\lambda.$
There is $\mathbb{P}$ a \textsf{c.c.c.} forcing of size $\lambda$ such that
$\mathbb{P}\Vdash$\textquotedblleft$\mathfrak{m=c=}$ $\lambda$%
\textquotedblright. \label{Teo forzar Martin}
\end{theorem}

\section{Preliminaries on walks on ordinals}

The method of \emph{walks on ordinals} (introduced by the second author) is a
powerful technique for constructing structures on $\omega_{1}.$ Virtually all
interesting combinatorial objects on $\omega_{1},$ including Aronszajn trees,
Hausdorff gaps, Countryman lines, homogenous Eberlein compacta, colorings on
$\omega_{1}$ with strong combinatorial properties, \textsf{L}-spaces,
Lindel\"{o}f groups with non-Lindel\"{o}f square, ultrafilters, Fr\'{e}chet
groups... and much more can be constructed using this method. This section
will review the necessary result on walks on ordinals that are needed for our
work. The reader may learn more of this method in the books \cite{Walks} and
\cite{TopicsinSetTheory}. Further applications can be found in
\cite{WalksandselectiveUltrafilters}, \cite{PIDandTukey},
\cite{KomjathInaccesible}, \cite{Lspace}, \cite{LSpacewithseparablesquare},
\cite{HrusakEstrellitas}, \cite{LindeloffGroup}, \cite{LVectorSpaces},
\cite{FirstOmegaAlephs}, \cite{GroupsfromWalks},
\cite{RectangularSquareBracket}, \cite{ComplicatedColoringsRevisited},
\cite{CombinatorialPropertyofpfunctions} or \cite{KnasterandFriendsI} among
many others.

\qquad\qquad\qquad\qquad

We say that $\mathcal{C}=\left\{  C_{\alpha}\mid\alpha\in\omega_{1}\right\}  $
is a \textsf{C}\emph{-sequence }if the following conditions hold for every
$\alpha\in\omega_{1}:$

\begin{enumerate}
\item $C_{\alpha}\subseteq\alpha.$

\item $C_{\alpha+1}=\left\{  \alpha\right\}  .$

\item If $\alpha$ is a limit ordinal, then $C_{\alpha}=\left\{  C_{\alpha
}\left(  n\right)  \mid n\in\omega\right\}  $ is a cofinal subset of $\alpha$
of order type $\omega.$
\end{enumerate}

\qquad\qquad\ \qquad\ \qquad\ \ \ 

Fix $\mathcal{C}$ a \textsf{C}-sequence. When we write $C_{\alpha}=\left\{
C_{\alpha}\left(  n\right)  \mid n\in\omega\right\}  ,$ we always implicitly
assume that $C_{\alpha}\left(  n\right)  <C_{\alpha}\left(  n+1\right)  $ for
every $n\in\omega.$ For $\alpha<\beta$ two countable ordinals, the \emph{step
from} $\beta$ \emph{to} $\alpha$ is defined as \textsf{stp}$\left(
\alpha,\beta\right)  =$ \textsf{min}$\left(  C_{\beta}\setminus\alpha\right)
.$ The \emph{upper trace of the walk from} $\beta$ \emph{to }$\alpha$ is
define as \textsf{Tr}$\left(  \alpha,\beta\right)  =\left\{  \beta
_{0},...,\beta_{n}\right\}  $ where $\beta_{0}=\beta,$ $\beta_{n}=\alpha$ and
$\beta_{i+1}=$ \textsf{stp}$\left(  \alpha,\beta_{i}\right)  $ (\textsf{Tr}%
$\left(  \alpha,\beta\right)  $ is a finite set since there are no infinite
decreasing sequences of ordinals). Note that The \emph{full code of the walk
}is defined as $\rho_{0}\left(  \alpha,\beta\right)  =\langle\left\vert
C_{\beta_{0}}\cap\alpha\right\vert ,\left\vert C_{\beta_{1}}\cap
\alpha\right\vert ,...,\left\vert C_{\beta_{n-1}}\cap\alpha\right\vert
\rangle,$ which is an element of $\omega^{<\omega}.$ Define $\lambda\left(
\alpha,\beta\right)  =$ \textsf{max}$\{\bigcup C_{\beta_{i}}\cap\alpha\mid
i<n\}$ and note that it is an ordinal smaller than $\alpha.$ Both the upper
trace and the full code can be defined recursively: \qquad\ \ \ 

\begin{center}%
\begin{tabular}
[c]{c}%
\textsf{Tr}$\left(  \alpha,\beta\right)  =\left\{  \beta\right\}  \cup$
\textsf{Tr(}$\alpha,$\textsf{min}$\left(  C_{\beta}\setminus\alpha\right)
$\textsf{)}\\
\ \ \ \ \\
$\rho_{0}\left(  \alpha,\beta\right)  =\langle\left\vert C_{\beta}\cap
\alpha\right\vert \rangle^{\frown}\rho_{0}(\alpha,$\textsf{min}$\left(
C_{\beta}\setminus\alpha\right)  )$%
\end{tabular}

\end{center}

\qquad\ \qquad\ \ \ \qquad\ \ \ 

A simple, yet crucial result is the following:

\begin{proposition}
Let $\alpha<\gamma<\beta$ be countable ordinals. The following conditions are
equivalent: \label{Prop juntar trazas}

\begin{enumerate}
\item $\gamma\in$ \textsf{Tr}$\left(  \alpha,\beta\right)  .$

\item \textsf{Tr}$\left(  \alpha,\beta\right)  =$\textsf{Tr}$\left(
\gamma,\beta\right)  \cup$\textsf{Tr}$\left(  \alpha,\gamma\right)  .$

\item $\lambda\left(  \gamma,\beta\right)  <\alpha.$

\item $\rho_{0}\left(  \alpha,\beta\right)  =\rho_{0}\left(  \gamma
,\beta\right)  ^{\frown}\rho_{0}\left(  \alpha,\gamma\right)  .$
\end{enumerate}
\end{proposition}

\qquad\ \ \ \ \ \ \ \ 

Therefore, if $\lambda\left(  \gamma,\beta\right)  <\alpha,$ then the walk
from $\beta$ to $\alpha$ decomposes as the walk from $\beta$ to $\gamma$
followed by the walk from $\gamma$ to $\alpha.$ For $\beta$ a countable
ordinal, define $\rho_{0\beta}:\beta\longrightarrow\omega^{<\omega}$ by
$\rho_{0\beta}\left(  \alpha\right)  =\rho_{0}\left(  \alpha,\beta\right)  .$
The \emph{Aronszajn tree induced by} $\rho_{0}$ is then defined as $T\left(
\rho_{0}\right)  =\left\{  \rho_{0\beta}\upharpoonright\alpha\mid\alpha
\leq\beta<\omega_{1}\right\}  .$ In \cite{Walks} it is proved that $T\left(
\rho_{0}\right)  $ is a Hausdorff, special Aronszajn tree that does not branch
at limit levels. For convenience, we write $\triangle_{\mathcal{C}}\left(
\alpha,\beta\right)  $ instead of $\triangle_{T\left(  \rho_{0}\right)
}\left(  \rho_{0\alpha},\rho_{0\beta}\right)  .$ The following lemma is
well-known, we prove it for completeness.

\begin{lemma}
The set $\{\rho_{0\alpha}\mid\alpha\in$ \textsf{LIM}$\left(  \omega
_{1}\right)  \}$ is an antichain of $T\left(  \rho_{0}\right)  .$
\label{Lema limite anticadena}
\end{lemma}

\begin{proof}
If $\alpha$ is a limit ordinal and $\xi<\alpha,$ then $\rho_{0\alpha}\left(
\xi\right)  $ is a sequence of length $1$ if and only if $\xi\in C_{\alpha}.$
The conclusion of the lemma easily follows.
\end{proof}

\qquad\ \qquad\ \ \ \ 

The following lemma is well-known (for example, it follows from Fact 2 in
\cite{Lspace}). We provide a proof here for the convenience of the reader.

\begin{lemma}
Let $\delta<\omega_{1}.$ If $\xi<\delta,$ then $\xi\leq\lambda\left(
\alpha,\delta\right)  $ for almost all $\alpha\in\left(  \xi,\delta\right)  .$
In particular, if $\left\langle \delta_{n}\right\rangle \longrightarrow
\delta,$ then $\left\langle \lambda\left(  \delta_{n},\delta\right)
\right\rangle \longrightarrow\delta.$ \label{Lemma lambas converjen}
\end{lemma}

\begin{proof}
Suppose there is an infinite set $A\subseteq\left(  \xi,\delta\right)  $ such
that $\lambda\left(  \alpha,\delta\right)  <\xi$ for all $\alpha\in A.$ By
Proposition \ref{Prop juntar trazas}, it follows that $A\subseteq$
\textsf{Tr}$(\xi,\delta).$ Hence, \textsf{Tr}$(\xi,\delta)$ is infinite, which
is a contradiction.
\end{proof}

\section{Preliminaries on Baumgartner's club forcing}

\emph{Baumgartner's club forcing} is the standard method for adding a club
subset to $\omega_{1}$ using finite conditions. It was introduced and studied
by Baumgartner in \cite{ApplicationsofPFA}. This is a very interesting forcing
notion with several applications and is a natural example of a proper forcing
that is not $\omega$-proper. The finiteness of its conditions facilitates
various amalgamation arguments, which in turn yield rich combinatorial
properties for the generic club. This section reviews the fundamentals of this
forcing notion. For further reading, the reader is referred to
\cite{MultipleForcing}, \cite{ApplicationsofPFA}, \cite{ForcingClubs},
\cite{CharacterizationofClubForcing},
\cite{CharacterizationofDefinableForcings}, \cite{AddingBaumgartnerClubs},
\cite{SeparatingClubPrinciples}. By \textsf{Int}$\left(  \omega_{1}\right)  $
we denote the set $\left\{  (\alpha,\beta]\mid\alpha<\beta<\omega_{1}\right\}
.$

\begin{definition}
\emph{Baumgartner's club forcing }(denoted by $\mathbb{BC}$) consists of all
pairs $p=\left(  a,I\right)  $ with the following properties:

\begin{enumerate}
\item $a\in\left[  \omega_{1}\right]  ^{<\omega}.$

\item $I\subseteq$ \textsf{Int}$\left(  \omega_{1}\right)  $ is finite.

\item $a\cap\bigcup I=\emptyset.$
\end{enumerate}
\end{definition}

\qquad\qquad\qquad\ \ \ 

Given $p=\left(  a_{p},I_{p}\right)  $ and $q=\left(  a_{q},I_{q}\right)  $ in
$\mathbb{BC},$ define $p\leq q$ if $a_{q}\subseteq a_{p}$ and $\bigcup
I_{q}\subseteq\bigcup I_{p}.$ This is the presentation of Baumgartner's club
forcing from \cite{CharacterizationofClubForcing} (there are other several
presentations, all yielding forcing equivalent notions). For convenience, we
will say that $\alpha$ \emph{is above }$p$ if $a_{p}\subseteq\alpha$ and
$\bigcup I_{p}\subseteq\alpha.$

\qquad\ \qquad\ \ \ \ 

If $G\subseteq\mathbb{BC}$ is a generic filter, the \emph{Baumgartner's
generic club }is defined as $C_{gen}=\bigcup\left\{  a_{p}\mid p\in G\right\}
.$ It is easy to verify that $C_{gen}$ is forced to be a club.

\qquad\qquad\qquad\qquad

Since $\mathbb{BC}$ is a partial order of size $\omega_{1},$ then it satisfies
the $\omega_{2}$-chain condition.\textsf{ }Furthermore, it is a proper
forcing. In fact, identifying the generic conditions over a given countable
elementary submodel is a very easy:

\begin{proposition}
[\cite{ApplicationsofPFA}]Let $M$ be a countable elementary submodel and
$p=\left(  a_{p},I_{p}\right)  \in\mathbb{BC}.$ The following conditions are
equivalent: \label{Prop BC proper}

\begin{enumerate}
\item $p$ is an $(M,\mathbb{BC})$-generic condition.

\item $\delta_{M}\in a_{p}$ (where $p=\left(  a_{p},I_{p}\right)  $).
\end{enumerate}
\end{proposition}

\qquad\ \ \ \ \ \ 

With this last result and Propositions \ref{prop ccc gen} and
\ref{prop iteracion gen}, we get the following:

\begin{corollary}
Let $\mathbb{\dot{Q}}$ be a $\mathbb{BC}$-name for a \textsf{c.c.c.} partial
order and $M$ a countable elementary submodel with $\mathbb{\dot{Q}}\in M.$
For $(p,\dot{x})\in\mathbb{BC}\ast\mathbb{\dot{Q}},$ the following conditions
are equivalent: \label{Cor BC 2 pasos}

\begin{enumerate}
\item $(p,\dot{x})$ is $(M,\mathbb{BC}\ast\mathbb{\dot{Q}})$-generic.

\item $\delta_{M}\in a_{p}.$\qquad
\end{enumerate}
\end{corollary}

\qquad\qquad\ \ \ 

The properness and the $\omega_{2}$-chain condition of $\mathbb{BC}$\ imply
that no cardinals are collapsed after performing a forcing extension. Let
$p,r\in\mathbb{BC}.$ Note that if $p$ and $r$ are compatible, then $p\wedge
r=\left(  a_{r}\cup a_{p},I_{p}\cup I_{q}\right)  $ is the largest common
extension of both $p$ and $r.$ For $p,q\in\mathbb{BC},$ define $q\sqsubseteq
p$ if there is $\delta\in a_{p}$ such that $a_{q}=a_{p}\cap\delta$ and
$I_{q}=\left\{  (\alpha,\beta]\in I_{p}\mid\beta<\delta\right\}  .$ Note that
if $q\sqsubseteq p,$ then $p\leq q.$ If $M$ is a countable elementary submodel
and $p$ is an $(M,\mathbb{BC})$-generic condition, define $p_{M}=\left(
a_{p}\cap M,I_{p}\cap M\right)  .$ We have the following:

\begin{proposition}
Let $M$ be a countable elementary submodel of some \textsf{H}$\left(
\kappa\right)  $ and $p$ an $(M,\mathbb{BC})$-generic condition. The following
holds:\label{Prop BC Strongly proper}

\begin{enumerate}
\item $p_{M}=\left(  a_{p}\cap M,I_{p}\cap M\right)  \sqsubseteq p.$ In
particular, $p\leq p_{M}.$

\item $p_{M}\in M.$

\item If $r\in M$ and $r\leq p_{M},$ then $r$ and $p$ are compatible.
\end{enumerate}
\end{proposition}

\qquad\ \qquad\ 

This last proposition implies that $\mathbb{BC}$ is \emph{strongly proper
}(see \cite{NotesonForcingAxioms}). How can we add many Baumgartner clubs?
Since $\mathbb{BC}$ is not \textsf{c.c.c.}, a finite support product (which in
this case, is equivalent to a finite support iteration) would collapse
$\omega_{1}.$ What about the countable support product? In general, if
$\mathbb{P}$ is a forcing notion and $X\ $is a set of ordinals, by
$\prod\limits_{\alpha\in X}^{CS}$ $\mathbb{P}$ we denote the \emph{countable
support product of }$\mathbb{P},$ which consists of all countable partial
functions from\ $X$ to $\mathbb{P}$. Although the countable support product is
often useful for some applications (for example, to add many Sacks or Silver
reals), it is of no use in this context. In \cite{AddingBaumgartnerClubs} it
is shown that the countable support product also collapses $\omega_{1}$ (this
is an instance of a more general phenomenon: Hru\v{s}\'{a}k proved in
\cite{Rendezvous} that any countable support product of infinitely many proper
forcing of size at most $\mathfrak{c}$ that add unbounded reals will
collapse$\mathbb{\ }\omega_{1}$). Of course, we could use a countable support
iteration. However, this approach not only restricts us to models where the
continuum is at most $\omega_{2},$ but also causes us to lose the
\textquotedblleft finite nature\textquotedblright\ that is such an appealing
feature of Baumgartner's original forcing. Fortunately for us, in
\cite{AddingBaumgartnerClubs} Asper\'{o} found a method for adding many
Baumgartner clubs without collapsing cardinals. This method behaves much like
a product (one can think of it as lying between the finite and countable
support products), while still preserving some of the \textquotedblleft finite
essence\textquotedblright\ of the Baumgartner club forcing. We will now review
this method, all the result are due to Asper\'{o} and can be found in
\cite{AddingBaumgartnerClubs}.

\begin{definition}
$\mathbb{BC}_{0}$ denotes the set of all $p=\left(  a,I\right)  \in
\mathbb{BC}$ such that $a\subseteq$ \textsf{IND}$\left(  \omega_{1}\right)  $
and $I=\emptyset.$ Elements of $\mathbb{BC}_{0}$ are often called \emph{pure
conditions.}
\end{definition}

\qquad\ \qquad\ \ 

It is easy to see that any finite set of conditions in $\mathbb{BC}_{0}$ has a
lower bound in $\mathbb{BC}_{0}$. For $X$ a set of ordinals, by $\mathbb{BC}%
^{X}$ we denote the countable support product $\prod\limits_{\alpha\in X}%
^{CS}\mathbb{BC}.$

\begin{definition}
Let $X\subseteq$ \textsf{OR}, $p\in\mathbb{BC}^{X}$ and $\gamma\in\omega_{1}.$

\begin{enumerate}
\item The \emph{support of} $p$ is \textsf{sup}$\left(  p\right)  =\left\{
\alpha\in X\mid p\left(  \alpha\right)  \neq1_{\mathbb{BC}}\right\}  .$

\item The \emph{pure support of} $p$ is \textsf{sup}$_{\text{\textsf{pr}}%
}\left(  p\right)  =\{\alpha\in$ \textsf{sup}$\left(  p\right)  \mid p\left(
\alpha\right)  \in\mathbb{BC}_{0}\}.$

\item The \emph{impure support of} $p$ is \textsf{sup}$_{\text{\textsf{imp}}%
}\left(  p\right)  =$ \textsf{sup}$\left(  p\right)  \setminus$\textsf{sup}%
$_{\text{\textsf{pr}}}\left(  p\right)  $.

\item The \emph{full image of }$p$ is \textsf{Fim}$\left(  p\right)
=\bigcup\{a_{p\left(  \alpha\right)  }\mid\alpha\in$ \textsf{sup}$\left(
p\right)  \}.$

\item The $\gamma$\emph{-support of} $p$ is \textsf{sup}$_{\gamma}\left(
p\right)  =\{\alpha\in$ \textsf{sup}$\left(  p\right)  \mid\gamma\in
a_{p\left(  \alpha\right)  }\}.$
\end{enumerate}
\end{definition}

\qquad\qquad\qquad\ \ \ \ 

If $\gamma$ is an indecomposable ordinal, we denote by $\gamma^{+}$ the first
indecomposable ordinal above $\gamma.$ Although this notation is also used for
successor cardinals, no confusion should arise. We can now define Asper\'{o}'s
forcing for adding many Baumgartner clubs.

\begin{definition}
Let $X\subseteq$ \textsf{OR. }The \textsf{Asper\'{o} product }$\mathbb{BC}%
^{\left(  X\right)  }$ is the suborder of $\mathbb{BC}^{X}$ consisting of all
$p\in\mathbb{BC}^{X}$ with the following properties:

\begin{enumerate}
\item Both \textsf{sup}$_{\text{\textsf{imp}}}\left(  p\right)  $ and
\textsf{Fim}$\left(  p\right)  $ are finite.

\item If $\delta$ is an indecomposable ordinal, then \textsf{OT}$\mathsf{(}%
$\textsf{sup}$_{\delta}\left(  p\right)  )<\delta^{+}.$
\end{enumerate}
\end{definition}

\qquad\ \ \ 

The order on $\mathbb{BC}^{\left(  X\right)  }$ is inherited from
$\mathbb{BC}^{X}.$ Although our definition of the Asper\'{o} product differs
slightly from the one in \cite{AddingBaumgartnerClubs}, these differences are inconsequential.

\qquad\ \ \ 

Let $p,q\in\mathbb{BC}^{\left(  X\right)  }.$ Note that $p$ and $q$ are
compatible in $\mathbb{BC}^{X}$ if and only if they are compatible in
$\mathbb{BC}^{\left(  X\right)  }.$ It is straightforward to check that if
$Y\subseteq X$ and $p\in\mathbb{BC}^{\left(  X\right)  },$ then the
restriction $p\upharpoonright Y$ is in $\mathbb{BC}^{\left(  Y\right)  }.$ We
also have the following:

\begin{lemma}
Let $Y\subseteq X$. The forcing $\mathbb{BC}^{\left(  Y\right)  }$ is a
regular suborder of $\mathbb{BC}^{\left(  X\right)  }.$ Moreover, if
$p\in\mathbb{BC}^{\left(  X\right)  },$ then $p\leq p\upharpoonright y$ and
$p\upharpoonright y$ is a reduction of $p$ to $\mathbb{BC}^{\left(  Y\right)
}.$
\end{lemma}

\qquad\ \qquad\ \ \ 

It is clear that $\mathbb{BC}^{\left(  \left\{  \gamma\right\}  \right)  }$ is
isomorphic to $\mathbb{BC},$ hence $\mathbb{BC}^{\left(  X\right)  }$ adds
many Baumgartner clubs.

\begin{lemma}
Let $X,Y\subseteq$ \textsf{OR.}

\begin{enumerate}
\item If $X\cap Y=\emptyset,$ then $\mathbb{BC}^{\left(  X\cup Y\right)  }$ is
isomorphic to $\mathbb{BC}^{\left(  X\right)  }\times\mathbb{BC}^{\left(
Y\right)  }.$ In particular, they are forcing equivalent.

\item If \textsf{OT}$\left(  X\right)  =$ \textsf{OT}$\left(  Y\right)  ,$
then $\mathbb{BC}^{\left(  X\right)  }$ and $\mathbb{BC}^{\left(  Y\right)  }$
are isomorphic.
\end{enumerate}
\end{lemma}

\qquad\ \ \ \qquad\ \ \ \qquad\ \qquad\ \ \ \ 

Regarding properness, the following result is the analogue of Proposition
\ref{Prop BC proper}:

\begin{theorem}
[Asper\'{o}, \cite{AddingBaumgartnerClubs}]Let $X\subseteq$ \textsf{OR}, $M$ a
countable elementary submodel with $X\in M$ and $p\in\mathbb{BC}^{\left(
X\right)  }.$ The following are equivalent: \label{Teo BCX proper}

\begin{enumerate}
\item $p$ is $(M,\mathbb{BC}^{\left(  X\right)  })$-generic.

\item $M\cap X\subseteq$ \textsf{sup}$_{\delta_{M}}\left(  p\right)  .$
\end{enumerate}
\end{theorem}

\qquad\qquad\ \ \ 

Although this last result is not stated explicitly in
\cite{AddingBaumgartnerClubs}, it follows directly from the proof of Lemma 3.7
of that paper.

\begin{corollary}
[Asper\'{o}, \cite{AddingBaumgartnerClubs}]Let $X\subseteq$ \textsf{OR. }The
Asper\'{o} product $\mathbb{BC}^{\left(  X\right)  }$ is proper and has the
$\omega_{2}$-\textsf{c.c. }In particular, it does not collapse cardinals.
\end{corollary}

\qquad\ \qquad\ \ \ \ \ 

Regarding the size of the continuum after forcing with $\mathbb{BC}^{\left(
X\right)  },$ we quote the following result:

\begin{proposition}
Let $\kappa\geq\omega_{2}$ be a cardinal such that $\kappa^{\omega_{1}}%
=\kappa.$ If $G\subseteq$ $\mathbb{BC}^{\left(  \kappa\right)  }$ is a generic
filter, then $V\left[  G\right]  $ is a model of $\kappa=2^{\omega}%
=2^{\omega_{1}}.$ \label{Prop BC tamao continuo}
\end{proposition}

\qquad\ \qquad\ \qquad\qquad\qquad\ \ \ \ \ \ \ \ \ 

For similar constructions to $\mathbb{BC}^{\left(  X\right)  },$ we refer the
reader to \cite{AsperoMota} and \cite{AGeneralizationofMA}.

\section{The square-bracket operations \label{The operations}}

This section reviews the definitions of the three square-bracket operations
from chapter 5 of the book \cite{Walks}, which are the protagonists of this
work. A key property of these operations is that every uncountable subset of
$\omega_{1}$ realizes club many values. We will study the Ramsey clubs
associated with each operation.

\qquad\qquad\qquad\qquad

The first square-bracket operation is defined using walks on ordinals. Let
$\mathcal{C}$ be a \textsf{C}-sequence. Define $[\cdot\cdot]_{\mathcal{C}%
}:\left[  \omega_{1}\right]  ^{2}\longrightarrow\omega_{1}$ such that for
every $\alpha<\beta$:

\begin{center}%
\begin{tabular}
[c]{|lll|}\hline
$\lbrack\alpha\beta]_{\mathcal{C}}$ & $=$ & \textsf{min}$\{$\textsf{Tr}%
$(\triangle_{\mathcal{C}}\left(  \alpha,\beta\right)  ,\beta)\setminus
\alpha\}$\\
& $=$ & \textsf{min}$\{$\textsf{Tr}$\left(  \alpha,\beta\right)  \cap$
\textsf{Tr}$(\triangle_{\mathcal{C}}\left(  \alpha,\beta\right)  ,\beta
)\}$\\\hline
\end{tabular}

\end{center}

\qquad\ \ \ \qquad\ \ \ 

In this way, when walking from $\beta$ to $\triangle_{\mathcal{C}}\left(
\alpha,\beta\right)  ,$ the smallest point encountered that is greater than or
equal to $\alpha$ is precisely $[\alpha\beta]_{\mathcal{C}}.$ Equivalently, it
is the minimal point common to both the walk from $\beta$ to $\alpha$ and the
walk from $\beta$ to $\triangle_{\mathcal{C}}\left(  \alpha,\beta\right)  $
(see \cite{Walks}). It is worth pointing out that the definition in
\cite{Walks} uses $\triangle_{\mathcal{C}}\left(  \alpha,\beta\right)  -1$
instead of $\triangle_{\mathcal{C}}\left(  \alpha,\beta\right)  .$ This
distinction is irrelevant to our arguments.

\qquad\qquad\qquad\ \ \ \ \ \ \qquad\qquad\ \ \ 

The second square-bracket operation is defined with a sequence of (distinct)
reals $\mathcal{R}=\left\{  r_{\alpha}\mid\alpha\in\omega_{1}\right\}
\subseteq2^{\omega}$ and a sequence $e=\left\langle e_{\beta}\mid\beta
\in\omega_{1}\right\rangle $ where $e_{\beta}:\beta+1\longrightarrow\omega$ is
an injective function such that $e_{\beta}\left(  \beta\right)  =0.$ For
convenience we will call a sequence with this properties a \emph{nice
sequence}. For $\alpha$ and $\beta$ countable ordinals, denote $\triangle
_{\mathcal{R}}\left(  \alpha,\beta\right)  =\triangle\left(  r_{\alpha
},r_{\beta}\right)  .$ Define $[\cdot\cdot]_{\mathcal{R},e}:\left[  \omega
_{1}\right]  ^{2}\longrightarrow\omega_{1}$ such that for every $\alpha<\beta$:

\begin{center}%
\begin{tabular}
[c]{|l|}\hline
$\lbrack\alpha\beta]_{\mathcal{R},e}=$ \textsf{min}$\{\mathsf{\xi}\in\left[
\alpha,\beta\right]  \mid e_{\beta}\left(  \xi\right)  \leq\triangle
_{\mathcal{R}}\left(  \alpha,\beta\right)  \}$\\\hline
\end{tabular}

\qquad\ \qquad\ \ \ 
\end{center}

The computation of $[\alpha\beta]_{\mathcal{R},e}$ proceeds by finding the
smallest $m$ for which $r_{\alpha}\left(  m\right)  \neq r_{\beta}\left(
m\right)  $ and then taking the least $\xi\in$ $\left[  \alpha,\beta\right]  $
such that $e_{\beta}\left(  \xi\right)  \leq m.$

\qquad\qquad\ \ \qquad\ \ \ \ \ 

The third square-bracket operation is defined with a Hausdorff, special
Aronszajn tree $T$ and an specialization mapping $a:T\longrightarrow\omega.$
This time, the operation is defined on $T$ instead of $\omega_{1}.$ Given
$s,t\in T,$ define $s\wedge t$ as the largest lower bound for both $s$ and
$t.$ Note that $s\wedge t$ is well-defined since $T$ is a Hausdorff tree. For
$s,t\in T,$ let $[s$ $,t]=\left\{  z\in T\mid s\leq_{T}z\leq_{T}t\right\}  .$

\qquad\ \qquad\ \qquad\ \ \ \ \ \ 

Define $[\cdot\cdot]_{T,a}:\left[  T\right]  ^{2}\longrightarrow T$ such that
for every $s,t\in T$ with $\left\vert s\right\vert \leq\left\vert t\right\vert
$:

\begin{center}%
\begin{tabular}
[c]{|l|}\hline
$\lbrack st]_{T,a}=$ \textsf{min}$\{z\in\lbrack t\upharpoonright\left\vert
s\right\vert $ $,t]\mid z=t\vee a\left(  z\right)  \leq a\left(  s\wedge
t\right)  $ $\}$\\\hline
\end{tabular}

\qquad\ \qquad\ \ 
\end{center}

To compute $[st]_{T,a},$ first set $m=a\left(  s\wedge t\right)  .$ Then, find
the first node $z\leq_{T}t$ with $\left\vert z\right\vert \geq\left\vert
s\right\vert $ and $a\left(  z\right)  \leq m.$ If no node such $z$ exists,
let $[st]_{T,a}=t.$ Since the operation is not defined on $\omega_{1},$ in
this context the vertex set of the graph $G([\cdot\cdot],X)$ (for
$X\subseteq\omega_{1}$) is the tree $T.$ Two nodes $s$ and $t$ are connected
by an edge if the height of $[st]_{T,a}$ is in $X.$

\qquad\qquad\qquad\qquad\qquad

The reader is encouraged to review the proofs that the square-bracket
operations defined above all have the property that every uncountable subset
of $\omega_{1}$ realizes club-many distinct values (Lemmas 5.1.4, 5.4.1, and
5.5.1 of \cite{Walks}), as ideas from these proofs will be used frequently.

\section{Finding Ramsey clubs \label{Finding Ramsey clubs}}

Let $[\cdot\cdot]$ be a square-bracket operation and $C\subseteq\omega_{1}$ a
club. Recall that $C$ is $[\cdot\cdot]$\emph{-Ramsey }if there exists an
uncountable $W\subseteq\omega_{1}$ such that $\left\{  \left[  \alpha
\beta\right]  \mid\alpha,\beta\in W\right\}  \subseteq C$ (in other words, $W$
induces a complete subgraph of $G([\cdot\cdot],C)$). In this section, we prove
that for every club $C,$ there exists a choice of parameters for the
square-bracket operations $[\cdot\cdot]$ such that $C$ is $[\cdot\cdot]$-Ramsey.

\begin{theorem}
Let $C\subseteq\omega_{1}$ be a club, $\mathcal{R}=\left\{  r_{\alpha}%
\mid\alpha\in\omega_{1}\right\}  \subseteq2^{\omega}$ a sequence of reals and
$T$ a Hausdorff special Aronszajn tree. There exist a \textsf{C}-sequence
$\mathcal{C}$, a nice sequence $e$ and a specializing mapping
$a:T\longrightarrow\omega$ such that $C$ is $[\cdot\cdot]$\emph{-}Ramsey,
where $[\cdot\cdot]$\emph{ }is\emph{ }any of the operations $[\cdot
\cdot]_{\mathcal{C}},$ $[\cdot\cdot]_{\mathcal{R},e}$ or $[\cdot\cdot]_{T,a}.$
\label{Teorema Todo club es Ramsey}
\end{theorem}

\begin{proof}
We may assume $C\subseteq$ \textsf{LIM}$\left(  \omega_{1}\right)  $ and
recall that $C^{\prime}$ denotes the accumulation points of $C.$ We first
construct the \textsf{C}-sequence $\mathcal{C}.$ Find a sequence
$\mathcal{C}=\left\{  C_{\alpha}\mid\alpha\in\omega_{1}\right\}  $ such that
for every $\alpha\in\omega_{1}:$

\begin{enumerate}
\item $C_{\alpha+1}=\left\{  \alpha\right\}  .$

\item If $\alpha\in C^{\prime},$ then $C_{\alpha}=\left\{  C_{\alpha}\left(
n\right)  \mid n\in\omega\right\}  \subseteq C$ is a cofinal subset of
$\alpha$ of order type $\omega$.

\item If $\alpha$ is a limit ordinal, but $\alpha\notin C^{\prime},$ then
$C_{\alpha}=\left\{  C_{\alpha}\left(  n\right)  \mid n\in\omega\right\}  $ is
a cofinal subset of $\alpha$ of order type $\omega$ and $C_{\alpha}\left(
0\right)  =\bigcup\left(  C\cap\alpha\right)  .$
\end{enumerate}

\qquad\ \ 

We emphasize that for $\alpha$ limit: In case $\alpha\in C^{\prime},$ then
$C_{\alpha}$ is contained in $C,$ while if $\alpha\notin C^{\prime},$ then at
most the first point of $C_{\alpha}$ is in $C$. Moreover, if $C\cap\alpha
\neq\emptyset,$ then $C_{\alpha}\left(  0\right)  $ is in $C.$ We claim that
if $\alpha,\beta\in C^{\prime}$ (with $\alpha<\beta$), then \textsf{Tr}%
$(\alpha,\beta)\subseteq C.$ We proceed by contradiction, let \textsf{Tr}%
$(\alpha,\beta)=\left\{  \beta_{0},...,\beta_{n+1}\right\}  $ (where
$\beta_{0}=\beta,$ $\beta_{n+1}=\alpha$ and $\beta_{i+1}=$ \textsf{stp}%
$\left(  \alpha,\beta_{i}\right)  $). Let $i$ be the first such that
$\beta_{i+1}\notin C$. Since $\beta_{i}\in C,$ it follows that $\beta_{i}$ is
a limit ordinal. Furthermore, since $\beta_{i+1}\in C_{\beta_{i}},$ it follows
that $C_{\beta_{i}}$ is not a subset of $C,$ hence $\beta_{i}\notin C^{\prime
}.$ We know that $C\cap\beta_{i}\neq\emptyset$ (because $\alpha$ is in the
intersection), so $C_{\beta_{i}}\left(  0\right)  =\bigcup\left(  C\cap
\beta_{i}\right)  $ and then $\alpha\leq C_{\beta_{i}}\left(  0\right)  .$ We
get that $\beta_{i+1}=C_{\beta_{i}}\left(  0\right)  ,$ which is in $C,$ so we
get a contradiction. It follows that $\{\left[  \alpha\beta\right]
_{\mathcal{C}}\mid\alpha,\beta\in C^{\prime}\}\subseteq C,$ so $C\ $is
$[\cdot\cdot]_{\mathcal{C}}$\emph{-}Ramsey.

\qquad\qquad\qquad

We now find the nice sequence $e.$ Find $e=\left\langle e_{\beta}\mid\beta
\in\omega_{1}\right\rangle $ as follows: If $\beta\notin C,$ then $e_{\beta
}:\beta+1\longrightarrow\omega$ can be any injective function such that
$e_{\beta}\left(  \beta\right)  =0.$ In case $\beta\in C,$ we choose an
injective function $e_{\beta}$ with the following property:

\qquad\ \qquad\ \ \ \ 

\hfill%
\begin{tabular}
[c]{|c|}\hline
For every $\gamma\in\beta,$ if $\gamma\notin C$ and $\delta=\bigcup\left(
C\cap\gamma\right)  ,$\\
then $e_{\beta}\left(  \delta\right)  <e_{\beta}\left(  \gamma\right)
$\\\hline
\end{tabular}
\hfill\ 

\qquad\ \ \ \qquad\ \ \ 

We claim that if $\alpha,\beta\in C$ (with $\alpha<\beta$), then $\left[
\alpha\beta\right]  _{\mathcal{R},e}\in C.$ Assume this is not the case. Let
$\gamma=\left[  \alpha\beta\right]  _{\mathcal{R},e}$ and suppose it is not in
$C.$ It follows that $\alpha<\gamma<\beta.$ Since $\alpha\in C,$ it follows
that $\delta=\bigcup\left(  C\cap\gamma\right)  $ is in $C\ $as well. Note
that $\alpha\leq\delta<\gamma$ and $e_{\beta}\left(  \delta\right)  <e_{\beta
}\left(  \gamma\right)  \leq\triangle_{\mathcal{R}}\left(  \alpha
,\beta\right)  $. In this way, $e_{\beta}\left(  \delta\right)  <\triangle
_{\mathcal{R}}\left(  \alpha,\beta\right)  ,$ but this is a contradiction,
since $\gamma$\textsf{ }is the smallest element of the set\textsf{
}$\{\mathsf{\xi}\in\left[  \alpha,\beta\right]  \mid e_{\beta}\left(
\xi\right)  \leq\triangle_{\mathcal{R}}\left(  \alpha,\beta\right)  \}.$ We
conclude that $\{\left[  \alpha\beta\right]  _{\mathcal{R},e}\mid\alpha
,\beta\in C\}\subseteq C,$ so $C\ $is $[\cdot\cdot]_{\mathcal{R},e}$\emph{-}Ramsey.

\qquad\qquad\qquad\qquad

It remains to find the specializing mapping. Let $f:T\longrightarrow P$ be an
specialization mapping, where $P$ is the set of prime numbers. Define
$a:T\longrightarrow\omega$ such that for every $s\in T,$ the following holds:

\begin{enumerate}
\item If $\left\vert s\right\vert \in C,$ then $a\left(  s\right)  =f\left(
s\right)  .$

\item If $\left\vert s\right\vert \notin C$ and $\delta=\bigcup\left(
C\cap\left\vert s\right\vert \right)  ,$ then $a\left(  s\right)  =f\left(
s\upharpoonright\delta\right)  ^{f\left(  s\right)  }.$
\end{enumerate}

\qquad\ \ 

Note that $a\left(  s\right)  $ is a prime number if and only if $\left\vert
s\right\vert \in C$. We claim that $a$ is a specialization mapping. Proceed by
contradiction, assume there are $s,t\in T$ such that $t<_{T}s$ and $a\left(
s\right)  =a\left(  t\right)  .$ If $a\left(  t\right)  $ is prime, it follows
that $a\left(  t\right)  =f\left(  t\right)  $ and $a\left(  s\right)
=f\left(  s\right)  ,$ which contradicts that $f$ is an specialization
mapping. We now assume $a\left(  t\right)  $ is not a prime number. Let
$\delta=\bigcup\left(  C\cap\left\vert s\right\vert \right)  .$ If $\delta
\leq\left\vert t\right\vert ,$ then $t\upharpoonright\delta=s\upharpoonright
\delta$ and $a\left(  t\right)  =f\left(  t\upharpoonright\delta\right)
^{f\left(  t\right)  },$ while $a\left(  s\right)  =f\left(  s\upharpoonright
\delta\right)  ^{f\left(  s\right)  }.$ It follows that $f\left(  t\right)
=f\left(  s\right)  ,$ which again contradicts that $f$ is a specialization
mapping. The only remaining case is when $\left\vert t\right\vert <\delta.$
Let $\gamma=\bigcup\left(  C\cap\left\vert s\right\vert \right)  ,$ which is
smaller that $\delta.$ It follows that $a\left(  s\right)  =f\left(
s\upharpoonright\delta\right)  ^{f\left(  s\right)  }$ and $a\left(  t\right)
=f\left(  t\upharpoonright\gamma\right)  ^{f\left(  t\right)  }.$ Since
$t\upharpoonright\gamma<_{T}s\upharpoonright\delta,$ it follows that $f\left(
s\upharpoonright\delta\right)  $ and $f\left(  t\upharpoonright\gamma\right)
$ are distinct prime numbers, so it is impossible that $a\left(  s\right)
=a\left(  t\right)  .$

\qquad\qquad\qquad

We now show that if $s,t\in T\upharpoonright C$ and $\left\vert s\right\vert
<\left\vert t\right\vert ,$ then $|\left[  st\right]  _{T,a}|$ $\in T.$ Let
$z=\left[  st\right]  _{T,a}$ and assume that $\left\vert z\right\vert \notin
C.$ It follows that $\left\vert s\right\vert <\left\vert z\right\vert .$ Let
$\delta=\bigcup\left(  C\cap\left\vert z\right\vert \right)  ,$ which is in
$C$ and $\left\vert s\right\vert \leq\delta.$ We know that $a\left(  z\right)
=f\left(  z\upharpoonright\delta\right)  ^{f\left(  z\right)  }.$ Since
$\delta\in C,$ we have that $a\left(  z\upharpoonright\delta\right)  =f\left(
z\upharpoonright\delta\right)  ,$ so $a\left(  z\upharpoonright\delta\right)
<a\left(  z\right)  <a\left(  s\wedge t\right)  .$ This contradicts the
minimality of $z.$ In this way, if $U$ is an uncountable subset of
$T\upharpoonright C$ such that any two different nodes have different height,
then $\{|\left[  st\right]  _{T,a}|\mid s,t\in U\}\subseteq C,$ so $C\ $is
$[\cdot\cdot]_{T,a}$\emph{-}Ramsey.
\end{proof}

\qquad\ \qquad\ \qquad\ \ \ \ \ \qquad\qquad\qquad\ 

Let $\mathcal{C}=\left\{  C_{\alpha}\mid\alpha\in\omega_{1}\right\}  $ be a
\textsf{C}-sequence. We say that $\mathcal{C}$ is a \emph{successor}
\textsf{C}\emph{-sequence }if for every $\alpha$ limit, every element of
$C_{\alpha}$ is a successor ordinal. Such sequences are sometimes useful (see
\cite{Walks}). The \textsf{C}-sequence found in Theorem
\ref{Teorema Todo club es Ramsey} is not a successor \textsf{C}-sequence.
However, one can be obtained if desired.

\begin{proposition}
Let $C\subseteq\omega_{1}$ be a club. There exists a successor \textsf{C}%
-sequence $\mathcal{D}$ such that $C$ is $[\cdot\cdot]_{\mathcal{D}}$\emph{-}Ramsey.
\end{proposition}

\begin{proof}
We may assume $C\subseteq$ \textsf{LIM}$\left(  \omega_{1}\right)  .$ Let
$\mathcal{C}$ be the \textsf{C}-sequence obtained in the proof of Theorem
\ref{Teorema Todo club es Ramsey}. Define a new \textsf{C}-sequence
$\mathcal{D}=\left\{  D_{\alpha}\mid\alpha\in\omega_{1}\right\}  $ where
$D_{\alpha}\left(  n\right)  =C_{\alpha}\left(  n\right)  +1$ for every
$\alpha$ limit and $n\in\omega.$ We claim that if $\alpha,\beta\in C^{\prime}$
(with $\alpha<\beta$), then $\left[  \alpha\beta\right]  =\left[  \alpha
\beta\right]  _{\mathcal{D}}\in C.$ If $\left[  \alpha\beta\right]  $ is
either $\alpha$ or $\beta,$ there is nothing to do, so assume that
$\alpha<\left[  \alpha\beta\right]  <\beta.$ In particular, $\alpha\notin$
\textsf{Tr}$(\triangle,\beta),$ where $\triangle=\triangle_{\mathcal{D}%
}\left(  \alpha,\beta\right)  .$ Let \textsf{Tr}$(\triangle,\beta)=\left\{
\beta_{0},...,\beta_{i},\beta_{i+1},...,\beta_{n+1}\right\}  $ where
$\beta_{0}=\beta,$ $\beta_{i+1}=\left[  \alpha\beta\right]  $ and $\beta
_{n+1}=\triangle.$ Note that $\beta_{i+2}<\alpha<\beta_{i+1}=$ $\left[
\alpha\beta\right]  .$ If $\left[  \alpha\beta\right]  $ was a successor
ordinal, we would have that $\left[  \alpha\beta\right]  -1<\alpha<\left[
\alpha\beta\right]  ,$ which is clearly impossible, so $\left[  \alpha
\beta\right]  $ is a limit ordinal. Hence $\beta_{i}=\left[  \alpha
\beta\right]  +1.$ Let $j<i$ be the largest such that $\beta_{j}$ is a limit ordinal.

\begin{claim}
$\beta_{j+2}\in C.$
\end{claim}

\qquad\qquad\ \ 

First consider the case where $\beta_{j}\in C^{\prime}.$ In this way, there is
$\gamma\in C$ such that $\beta_{j+1}=\gamma+1,$ so $\beta_{j+2}=\gamma.$ In
case $\beta_{j}\notin C^{\prime},$ we know that $D_{\beta_{j}}\left(
0\right)  =\bigcup\left(  C\cap\beta_{j}\right)  +1,$ which is in $C$ and
larger than $\alpha.$ It follows that $\beta_{j+1}=\bigcup\left(  C\cap
\beta_{j}\right)  +1$ and $\beta_{j+2}=\bigcup\left(  C\cap\beta_{j}\right)
.$ This finishes the proof of the claim.

\qquad\qquad\qquad

We now claim that $j=i-1.$ If this was not the case, we would have that
$j<i-1$ and then:

\hfill%
\begin{tabular}
[c]{lll}%
$j<i-1$ & $\Longrightarrow$ & $j+2<i+1$\\
& $\Longrightarrow$ & $j+2\leq i$\\
& $\Longrightarrow$ & $j+2<i$%
\end{tabular}
\hfill\ 

\qquad\qquad\qquad\ \ 

Where the last implication holds because $\beta_{i}$ is a successor and
$\beta_{j+2}$ is limit. But this is a contradicts the maximality of $j.$ In
this way, we have that $\left[  \alpha\beta\right]  =\beta_{i+1}=\beta
_{j+2}\in C.$ It follows that $\{\left[  \alpha\beta\right]  \mid\alpha
,\beta\in C^{\prime}\}\subseteq C,$ so $C\ $is $[\cdot\cdot]_{\mathcal{D}}%
$\emph{-}Ramsey.
\end{proof}

\section{Destroying square-bracket operations \label{Destroying square}}

We will say that a square-bracket operation $[\cdot\cdot]$ has the
\emph{Ramsey club property }if every club subset of $\omega_{1}$ is
$[\cdot\cdot]$-Ramsey. In this section we will prove that a Baumgartner club
can not be $[\cdot\cdot]$-Ramsey for any square-bracket operation from the
ground model. It follows that it is consistent some square-bracket operations
do not have the Ramsey club property.

\begin{theorem}
$\mathbb{BC}$ forces that the generic club is not $[\cdot\cdot]$-Ramsey, for
any square-bracket operation $[\cdot\cdot]$ in $V.$ \label{Teo BC destruve V}
\end{theorem}

\begin{proof}
Let $p_{0}\in\mathbb{BC}$ and $\dot{B}$ such that $p_{0}\Vdash$%
\textquotedblleft$\dot{B}\in\left[  \omega_{1}\right]  ^{\omega_{1}}%
$\textquotedblright. We want to extend $p_{0}$ to a condition that forces that
there are $\alpha,\beta\in\dot{B}$ such that $\left[  \alpha\beta\right]  $ is
not in $\dot{C}_{gen}$. Let $\left\langle M_{\alpha}\mid\alpha\leq
\omega\right\rangle $ be a continuous chain of elementary submodels such that
$p_{0},\dot{B}\in M_{0},$ as well as the square-bracket operation we want to
take care of. For convenience, we introduce the following notation:

\begin{enumerate}
\item $M=M_{\omega}.$

\item $\delta=\delta_{M}.$

\item $\delta_{n}=\delta_{M_{n}}$ for $n\in\omega.$

\item $p_{0}=\left(  a_{0},I_{0}\right)  .$
\end{enumerate}

\qquad\ \ \qquad\ \ 

Note that $\delta=\bigcup\limits_{n\in\omega}\delta_{n}.$ Let $p_{1}=\left(
a_{0}\cup\left\{  \delta\right\}  ,I_{0}\right)  .$ We know that $p_{1}\leq
p_{0}$ and is an $(M,\mathbb{BC})$-generic condition (see Proposition
\ref{Prop BC proper}). Since $\dot{B}$ is forced to be uncountable, we can now
find $p\leq p_{1}$ and $\beta>\delta$ such that $p\Vdash$\textquotedblleft%
$\beta\in\dot{B}$\textquotedblright. Write $p=\left(  a_{p},I_{p}\right)  .$
By Proposition \ref{Prop BC Strongly proper}, we know that $p_{M}=\left(
a_{p}\cap M,I_{p}\cap M\right)  $ is in $M$ and is weaker than $p.$ Denote
$p_{M}=\left(  a_{M},I_{M}\right)  .$

\qquad\qquad\qquad\qquad\ \ \ \ 

The proof now splits into cases depending on the type of square-bracket
operation under consideration. We begin with the case of the operation induced
by a \textsf{C}-sequence, say $\mathcal{C}.$ For notational convenience in
this part of the proof, we will write $\left[  \gamma\xi\right]  $ for
$\left[  \gamma\xi\right]  _{\mathcal{C}}$ and $\triangle\left(  \gamma
,\xi\right)  $ for $\triangle_{\mathcal{C}}\left(  \gamma,\xi\right)  $ (for
any $\gamma,\xi\in\omega_{1}$). We assume that $\mathcal{C}\in M_{0}.$

\qquad\ \ \ \qquad\ \ 

Let $g:T\left(  \rho_{0}\right)  \longrightarrow\omega$ be a specialization
mapping with $g\in M_{0}.$ Find $n\in\omega$ such that $\lambda\left(
\delta,\beta\right)  $ and $p_{M}$ are in $M_{n},$ then choose $\xi<\delta
_{n}$ such that $\lambda\left(  \delta,\beta\right)  ,$ $\lambda\left(
\delta_{n},\delta\right)  <\xi$ and $\xi$ is above $p_{M}.$ Denote
$s=\rho_{0\beta}\upharpoonright\xi$ and let $i\in\omega$ such that $g\left(
\rho_{0\beta}\upharpoonright\delta_{n}\right)  =i.$ Define $A$ as the set of
all $t\in T\left(  \rho_{0}\right)  $ that satisfy the following properties:

\begin{enumerate}
\item $s\subseteq t.$

\item $g\left(  t\right)  =i.$

\item For every $\eta<\omega_{1},$ there is $q\in\mathbb{BC}$ for which:

\begin{enumerate}
\item $p_{M}\sqsubseteq q.$

\item There is $\alpha>\eta$ such that $q\Vdash$\textquotedblleft$\alpha
\in\dot{B}$\textquotedblright\ and $t\subseteq\rho_{0\alpha}.$
\end{enumerate}
\end{enumerate}

\qquad\ \ 

It follows that $A\in M_{n},$ since all the parameters used to define in $A$
belong to $M_{n}.$ Moreover, note that $\rho_{0\beta}\upharpoonright\delta
_{n}\in A,$ which implies that $A\neq\emptyset.$ By elementarity., we can find
$t\in A\cap M_{n}.$ It follows that $s\subseteq t,\rho_{0\beta}\upharpoonright
\delta_{n}$ and that $t$ and $\rho_{0\beta}\upharpoonright\delta_{n}$ are
incompatible in $T\left(  \rho_{0}\right)  ,$ since $g\left(  t\right)
=g\left(  \rho_{0\beta}\upharpoonright\delta_{n}\right)  .$ Let $\triangle
=\triangle\left(  \rho_{0\beta},t\right)  .$ We look at \textsf{Tr}%
$(\triangle,\delta_{n}).$ Since $t\in A,$ we can find $q=\left(  a_{q}%
,I_{q}\right)  \in\mathbb{BC}$ and $\alpha$ such that:

\begin{enumerate}
\item $\alpha,q\in M_{n}.$

\item \textsf{Tr}$(\triangle,\delta_{n})\setminus\left\{  \delta_{n}\right\}
\subseteq\alpha.$

\item $p_{M}\sqsubseteq q.$

\item $q\Vdash$\textquotedblleft$\alpha\in\dot{B}$\textquotedblright\ and
$t\subseteq\rho_{0\alpha}.$
\end{enumerate}

\qquad\ \ \qquad\ \ \ \qquad\ \ 

It follows that $\triangle=\triangle\left(  \alpha,\beta\right)  .$ Pick
$\varepsilon<\delta_{n}$ that is above $q.$ Define $r=\left(  a_{p}\cup
a_{q},I_{p}\cup I_{q}\cup\left\{  (\varepsilon,\delta_{n}]\right\}  \right)
.$ It follows that $r$ extends $p$ and $q,$ it is $(M,\mathbb{BC})$-generic
but not $(M_{n},\mathbb{BC})$-generic. Note that $r\Vdash$\textquotedblleft%
$\alpha,\beta\in\dot{B}$\textquotedblright\ and $r\Vdash$\textquotedblleft%
$\delta_{n}\notin\dot{C}_{gen}$\textquotedblright. We will now compute
$\left[  \alpha\beta\right]  .$ Recall that $\lambda\left(  \delta
,\beta\right)  ,\lambda\left(  \delta_{n},\delta\right)  <\triangle.$ By
Proposition \ref{Prop juntar trazas} we get that:

\qquad\qquad\qquad\ \ \ \ \ \ \qquad\ \ \ 

\hfill%
\begin{tabular}
[c]{ccc}%
\textsf{Tr}$(\triangle,\beta)$ & $=$ & \textsf{Tr}$(\delta,\beta)\cup
$\textsf{Tr}$(\delta_{n},\delta)\cup$\textsf{Tr}$(\triangle,\delta_{n})$%
\end{tabular}
\hfill\ 

\ \ \ \qquad\ \ \ \ \qquad\ \ \ \ \qquad\qquad\ \ \ \ 

Since $\alpha<\delta_{n}$ and \textsf{Tr}$(\triangle,\delta_{n})\setminus
\left\{  \delta_{n}\right\}  \subseteq\alpha,$ it follows that $\left[
\alpha\beta\right]  =\delta_{n},$ which is forced by $r$ to be outside
$\dot{C}_{gen}.$

\qquad\ \ \qquad\ \ \ 

We now consider the case of the operation induced by a sequence of reals
$\mathcal{R}=\left\{  r_{\alpha}\mid\alpha\in\omega_{1}\right\}
\subseteq2^{\omega}$ and $e=\left\langle e_{\beta}\mid\beta\in\omega
_{1}\right\rangle $ a nice sequence. For notational convenience in this part
of the proof, we will write $\left[  \gamma\xi\right]  $ for $\left[
\gamma\xi\right]  _{\mathcal{R},e}$ and $\triangle\left(  \gamma,\xi\right)  $
for $\triangle_{\mathcal{R}}\left(  \gamma,\xi\right)  $ (for any $\gamma
,\xi\in\omega_{1}$). We are assuming that $\mathcal{R},e\in M_{0}.$

\qquad\qquad\qquad\qquad

Find $n\in\omega$ such that $p_{M}\in M_{n}$. Define $m=e_{\beta}\left(
\delta_{n}\right)  ,$ $s=r_{\beta}\upharpoonright m$ and $F=\left\{  \xi
<\beta\mid e_{\beta}\left(  \xi\right)  \leq m\right\}  \cap M_{n}.$ Since $F$
is a finite set, we have that $F\in M_{n}.$ Choose $\eta<\delta_{n}$ such that
$F\subseteq\eta$ and $\eta$ is above $p_{M}.$ Let $A$ be the set of all
$t\in2^{<\omega}$ such that there is $q\in\mathbb{BC}$ with the following properties:

\begin{enumerate}
\item $s\subseteq t.$

\item $p_{M}\sqsubseteq q.$

\item $q\Vdash$\textquotedblleft$\exists^{\omega_{1}}\alpha\in\dot{B}\left(
r_{\alpha}\in\left\langle t\right\rangle \right)  $\textquotedblright.
\end{enumerate}

\qquad\ \ 

Note that if $k>m,$ then $r_{\beta}\upharpoonright k\in A,$ as witnessed by
$p.$ Since $A\in M_{n}$ while $r_{\beta}$ is not, we can find $t\in A$ for
which $r_{\beta}\notin\left\langle t\right\rangle .$ Denote $\triangle
=\triangle\left(  t,r_{\beta}\right)  \geq m.$ By elementarity., we can find
$q_{0}\in\mathbb{BC}\cap M_{n}$ such that $p_{M}\sqsubseteq q_{0}$ and
$q_{0}\Vdash$\textquotedblleft$\exists^{\omega_{1}}\alpha\in\dot{B}\left(
r_{\alpha}\in\left\langle t\right\rangle \right)  $\textquotedblright. \ Let
$H=$ $\left\{  \xi<\beta\mid e_{\beta}\left(  \xi\right)  \leq\triangle
\right\}  \cap M_{n}.$ We can now find $q=\left(  a_{q},I_{q}\right)  \in
M_{n}$ extending $q_{0}$ and $\alpha\in M_{n}$ such that $q\Vdash
$\textquotedblleft$\alpha\in\dot{B}$\textquotedblright, $r_{\alpha}%
\in\left\langle t\right\rangle $ and $H\subseteq\alpha.$ Note that
$\triangle=\triangle\left(  \alpha,\beta\right)  .$ Choose $\varepsilon\in
M_{n}$ that is above $q$ and define $r=\left(  a_{p}\cup a_{q},I_{p}\cup
I_{q}\cup\left\{  (\varepsilon,\delta_{n}]\right\}  \right)  .$ It follows
that $r$ extends $p$ and $q,$ it is $(M,\mathbb{BC})$-generic but not
$(M_{n},\mathbb{BC})$-generic. It follows that $r\Vdash$\textquotedblleft%
$\alpha,\beta\in\dot{B}$\textquotedblright\ and $r\Vdash$\textquotedblleft%
$\delta_{n}\notin\dot{C}_{gen}$\textquotedblright. Moreover, it is easy to see
that $\left[  \alpha\beta\right]  =\delta_{n},$ which is forced by $r$ to be
outside $\dot{C}_{gen}.$

\qquad\qquad\qquad\ \ \ 

The last case is when the operation is induced by a Hausdorff, special
Aronszajn tree $T$ and $a:T\longrightarrow\omega$ a specialization mapping.
For notational convenience, we now write $\left[  st\right]  $ for $\left[
st\right]  _{T,r}$. We can assume that $T,a\in M_{0}.$ One more change is
needed, now $\dot{B}$ is forced to be an uncountable subset of $T$ and
$\beta>\delta$ is the height of some $t\in T$ such that $p\Vdash
$\textquotedblleft$t\in\dot{B}$\textquotedblright.

\qquad\ \qquad\ \ \ 

Choose $n\in\omega$ such that $p_{M}\in M_{n}$ . We can now define $m=a\left(
t\upharpoonright\delta_{n}\right)  $ and $F=\left\{  \eta<\beta\mid a\left(
t\upharpoonright\eta\right)  \leq m\right\}  \cap M_{n}.$ Since $F$ is a
finite set, we have that $F\in M_{n}.$ Choose $\xi<\delta_{n}$ such that
$F\subseteq\xi$ and $\xi$ is above $p_{M}.$ Let $A$ be the set of all $s\in T$
such that there is $q\in\mathbb{BC}$ with the following properties:

\begin{enumerate}
\item $t\upharpoonright\xi\leq_{T}s.$

\item $p_{M}\sqsubseteq q.$

\item $q\Vdash$\textquotedblleft$s\in\dot{B}$\textquotedblright.
\end{enumerate}

\qquad\ \qquad\ \ \ 

It follows that $A\in M_{n}$ and $t\in A.$ Apply Proposition
\ref{Prop Aronszajn y submodelo} and find $\gamma\in M$ and $w\in T_{\gamma}$
such that $w\perp t$ and there are uncountably many $s\in A$ such that
$w\leq_{T}s.$ We look at $w\wedge t.$ Since $\xi<\left\vert w\wedge
t\right\vert ,$ it follows that $a\left(  w\wedge t\right)  >m=a\left(
t\upharpoonright\delta_{n}\right)  .$ Define $G=\left\{  \eta<\beta\mid
a\left(  t\upharpoonright\eta\right)  \leq a\left(  w\wedge t\right)
\right\}  \cap M_{n}.$ Once again, $G\in M_{n}$ since it is a finite subset of
$M_{n}.$ We can now find $s\in A\cap M_{n}$ and $q\in\mathbb{BC}\cap M_{n}$
with the following properties:

\begin{enumerate}
\item $w\leq_{T}s.$

\item $\left\vert s\right\vert >\eta$ for every $\eta\in G.$

\item $p_{M}\sqsubseteq q.$

\item $q\Vdash$\textquotedblleft$s\in\dot{B}$\textquotedblright.
\end{enumerate}

\qquad\ \ \ \qquad\ \ \ 

Note that $s\wedge t=w\wedge t.$ Choose $\varepsilon\in M_{n}$ that is above
$q$ and define $r=\left(  a_{p}\cup a_{q},I_{p}\cup I_{q}\cup\left\{
(\varepsilon,\delta_{n}]\right\}  \right)  .$ It follows that $r$ extends $p$
and $q,$ it is $(M,\mathbb{BC})$-generic but not $(M_{n},\mathbb{BC}%
)$-generic. Note that $r\Vdash$\textquotedblleft$s,t\in\dot{B}$%
\textquotedblright\ and $r\Vdash$\textquotedblleft$\delta_{n}\notin\dot
{C}_{gen}$\textquotedblright. Moreover, we have that $\left[  st\right]
=t\upharpoonright\delta_{n},$ whose height is forced by $r$ to be outside
$\dot{C}_{gen}.$
\end{proof}

\qquad\qquad\qquad\ \ \qquad\ \ \ 

We now want to prove the consistency of Martin Axiom (\textsf{MA}) with the
existence of square-bracket operations that fail to have the Ramsey club
property. To establish this result, we aim to show that the Baumgartner club
causes this property to fail so badly (for square-bracket operations from the
ground model) that it can not be fixed by a \textsf{c.c.c.} forcing.
Unfortunately, we were only able to prove this result for the square-bracket
operations induced by \textsf{C}-sequences.

\begin{proposition}
Let $\mathbb{\dot{Q}}$ be a $\mathbb{BC}$-name for a \textsf{c.c.c.} forcing
and $\mathcal{C}$ a \textsf{C}-sequence. $\mathbb{BC\ast\dot{Q}}$ forces that
the generic club is not $[\cdot\cdot]_{\mathcal{C}}$-Ramsey$.$
\label{Prop ccc no arregla}
\end{proposition}

\begin{proof}
We follow an approach similar to that of Theorem \ref{Teo BC destruve V}.
However, this proof involves further complications because of the extra
difficulties that arise when working with an iteration. For notational
convenience, we will write $\left[  \gamma\xi\right]  $ for $\left[  \gamma
\xi\right]  _{\mathcal{C}}$ and $\triangle\left(  \gamma,\xi\right)  $ for
$\triangle_{\mathcal{C}}\left(  \gamma,\xi\right)  $ (for any $\gamma,\xi
\in\omega_{1}$).

\qquad\ \ \ \qquad\ \ \ 

Let $\left(  p_{0},\dot{x}\right)  \in\mathbb{BC\ast\dot{Q}}$ and $\dot{B}$
such that $\left(  p_{0},\dot{x}\right)  \Vdash$\textquotedblleft$\dot{B}%
\in\left[  \omega_{1}\right]  ^{\omega_{1}}$\textquotedblright. We want to
extend $\left(  p_{0},\dot{x}\right)  $ to a condition forcing that there are
$\alpha,\beta\in\dot{B}$ such that $\left[  \alpha\beta\right]  $ is not in
$\dot{C}_{gen}$. Let $\left\langle M_{\alpha}\mid\alpha\leq\omega
^{2}\right\rangle $ be a continuous chain of elementary submodels such that
$\mathbb{\dot{Q}},\left(  p_{0},\dot{x}\right)  ,\dot{B},\mathcal{C}\in M_{0}%
$. For convenience, we introduce the following notation:

\begin{enumerate}
\item $M=M_{\omega^{2}}.$

\item $\delta=\delta_{M}.$

\item $\delta_{\alpha}=\delta_{M_{\alpha}}$ for $\alpha<\omega^{2}.$

\item $p_{0}=\left(  a_{0},I_{0}\right)  .$
\end{enumerate}

\qquad\ \ \qquad\ \ 

Let $p_{1}=\left(  a_{0}\cup\left\{  \delta\right\}  ,I_{0}\right)  .$ By
Corollary \ref{Cor BC 2 pasos}, we know that $\left(  p_{1},\dot{x}\right)
\leq(p_{0},\dot{x})$ is an $(M,\mathbb{BC\ast\dot{Q}})$-generic condition.
Since $\dot{B}$ is forced to be uncountable, we can now find $(p,\dot{y}%
)\leq\left(  p_{1},\dot{x}\right)  $ and $\beta>\delta$ such that $(p,\dot
{y})\Vdash$\textquotedblleft$\beta\in\dot{B}$\textquotedblright. Write
$p=\left(  a_{p},I_{p}\right)  .$ By Proposition \ref{Prop BC Strongly proper}%
, we know that $p_{M}=\left(  a_{p}\cap M,I_{p}\cap M\right)  $ is in $M$ and
is weaker than $p.$ Denote $p_{M}=\left(  a_{M},I_{M}\right)  .$

\qquad\qquad\qquad\qquad\ \ \ \ \qquad\ \qquad\ \ \ \ \ \qquad\ \ \ \ \ 

Let $g:T\left(  \rho_{0}\right)  \longrightarrow\omega$ be a specialization
mapping with $g\in M_{0}.$ Find $\gamma<\omega^{2}$ a limit ordinal such that
$\lambda\left(  \delta,\beta\right)  ,p_{M}\in M_{\gamma}$. Choose
$\eta<\delta_{\gamma}$ that is above $p_{M}$ and $\lambda\left(  \delta
,\beta\right)  ,\lambda\left(  \delta_{\gamma},\delta\right)  <\eta.$ Since
$\gamma$ is limit, we can find $\varepsilon<\gamma$ such that $\eta
<\delta_{\varepsilon}.$ Note that $\delta_{\gamma}\in$ \textsf{Tr}%
$(\delta_{\varepsilon},\delta).$ Once again, since $\gamma$ is limit, we can
find $\alpha\in\left(  \varepsilon,\gamma\right)  $ such that \textsf{Tr}%
$(\delta_{\varepsilon},\delta_{\gamma})\setminus\left\{  \delta_{\gamma
}\right\}  \subseteq\delta_{\alpha}.$ Define $q=(a_{p}\cup\left\{
\delta_{\varepsilon},\delta_{\alpha}\right\}  ,I_{p}\cup\left\{
(\delta_{\alpha},\delta_{\gamma}]\right\}  ).$ It follows that $q$ is a
$\mathbb{BC}$-generic condition for $M,M_{\varepsilon}$ and $M_{\alpha},$ but
it is not for $M_{\gamma}.$ Let $G\subseteq\mathbb{BC}$ be a generic filter
with $q\in G.$

\qquad\qquad\qquad

We now go to $V\left[  G\right]  .$ Note that $M\left[  G\right]
,M_{\varepsilon}\left[  G\right]  $ and $M_{\alpha}\left[  G\right]  $ are
forcing extensions of $M,M_{\varepsilon}$ and $M_{\alpha}$ respectively.
Denote $y=\dot{y}\left[  G\right]  ,$ $\mathbb{Q=\dot{Q}}\left[  G\right]  $
and $i\in\omega\ $such\ that $g\left(  \rho_{0\beta}\upharpoonright
\delta_{\varepsilon}\right)  =i.$ Define $A$ as the set of $w\in\mathbb{Q}$
for which there is $t\in T\left(  \rho_{0}\right)  $ such that:

\begin{enumerate}
\item $\rho_{0\beta}\upharpoonright\lambda\left(  \delta_{\varepsilon}%
,\delta_{\gamma}\right)  ,$ $\rho_{0\beta}\upharpoonright\eta\subseteq t.$

\item $g\left(  t\right)  =i.$

\item $w\Vdash$\textquotedblleft$\exists^{\omega_{1}}\mu\in\dot{B}\left(
t\subseteq\rho_{0\mu}\right)  $\textquotedblright.
\end{enumerate}

\qquad\ \qquad\ \ \ \ 

If $w\in A$ and $t$ is as above, we say that $t$ is a witness for $w\in A.$
Note that $A\in M_{\varepsilon}\left[  G\right]  .$ We claim that $y\in A$ and
$\rho_{0\beta}\upharpoonright\delta_{\varepsilon}$ is a witness of it. We just
need to check that $y\Vdash$\textquotedblleft$\exists^{\omega_{1}}\mu\in
\dot{B}\left(  t\subseteq\rho_{0\mu}\right)  $\textquotedblright. In order to
check this, let $H\subseteq\mathbb{Q}$ be a $\left(  V\left[  G\right]
,\mathbb{Q}\right)  $ generic filter with $y\in H.$ We have that $M\left[
G\right]  \left[  H\right]  $ is a generic extension of $M\left[  G\right]  ,$
which is a generic extension of $M,$ so $\beta\notin M\left[  G\right]
\left[  H\right]  .$ It follows by Lemma \ref{Lema modelo contension} that
$\{\mu\in\dot{B}\left[  G\times H\right]  \mid t\subseteq\rho_{0\mu}\}$ is
uncountable. Since $H$ was any generic filter that has $y$ in it, it follows
that $y\Vdash$\textquotedblleft$\exists^{\omega_{1}}\mu\in\dot{B}\left(
t\subseteq\rho_{0\mu}\right)  $\textquotedblright, just as we wanted.

\qquad\qquad\qquad

We now know that $y\in A.$ Since $\mathbb{Q}$ is \textsf{c.c.c. }and $A\in
M_{\varepsilon}\left[  G\right]  ,$ it follows by Proposition
\ref{prop ccc gen} and Lemma \ref{equiv condicion generica} that there is
$w\in A\cap M_{\varepsilon}\left[  G\right]  $ that is compatible with $y.$
Pick $t\in T\left(  \rho_{0}\right)  \cap M_{\varepsilon}\left[  G\right]  $ a
witness for $w\in A.$ Since $g\left(  t\right)  $ and $g\left(  \rho_{0\beta
}\upharpoonright\delta_{\varepsilon}\right)  $ are the same, it follows that
$t$ and $\rho_{0\beta}$ are incomparable in $T\left(  \rho_{0}\right)  .$
Choose $u\in\mathbb{Q}$ a common extension of $w$ and $y.$

\qquad\qquad\qquad\qquad

With the recently new knowledge acquired, we return to $V.$ Let $\dot{u}$ and
$\dot{w}$ be $\mathbb{BC}$-names for $u$ and $w.$ We can now find $r\leq q$
such that:

\begin{enumerate}
\item $r\Vdash$\textquotedblleft$\dot{w}\in\dot{A}$\textquotedblright\ and
knows a $t\in M_{\varepsilon}$ that is a witness.

\item $r\Vdash$\textquotedblleft$\dot{u}\leq\dot{w},\dot{y}$\textquotedblright.
\end{enumerate}

\qquad\ \ \ 

It follows that $(r,\dot{u})\leq(p,\dot{y})$ (so $(r,\dot{u})\Vdash
$\textquotedblleft$\beta\in\dot{B}$\textquotedblright$)$ and that
\newline$(r,\dot{u})\Vdash$\textquotedblleft$\exists^{\omega_{1}}\mu\in\dot
{B}\left(  t\subseteq\rho_{0\mu}\right)  $\textquotedblright. Denote
$\triangle=\triangle\left(  \rho_{0\beta},t\right)  $ (recall that
$\rho_{0\beta}\perp t$). Since $\lambda\left(  \delta,\beta\right)
,\lambda\left(  \delta_{\gamma},\delta\right)  ,\lambda\left(  \delta
_{\varepsilon},\delta_{\gamma}\right)  <\triangle,$ we have that:

\qquad\ \ \ \ \ 

\hfill%
\begin{tabular}
[c]{ccc}%
\textsf{Tr}$(\triangle,\beta)$ & $=$ & \textsf{Tr}$(\delta,\beta)\cup
$\textsf{Tr}$(\delta_{\gamma},\delta)\cup$\textsf{Tr}$(\delta_{\varepsilon
},\delta_{\gamma})\cup$\textsf{Tr}$(\triangle,\delta_{\varepsilon})$%
\end{tabular}
\hfill\ 

\qquad\ \qquad\ \ \ \ 

It is also worth recalling that \textsf{Tr}$(\delta_{\varepsilon}%
,\delta_{\gamma})\setminus\left\{  \delta_{\gamma}\right\}  \subseteq
\delta_{\alpha},$ so \textsf{Tr}$(\triangle,\beta)\cap\delta_{\gamma}%
\subseteq\delta_{\alpha}.$ In this way, $\delta_{\gamma}=$ \textsf{min}%
$\{$\textsf{Tr}$(\triangle,\beta)\setminus\delta_{\alpha}\}.$ By Corollary
\ref{Cor BC 2 pasos}, we know that $(r,\dot{u})$ is an $(M_{\alpha
},\mathbb{BC\ast\dot{Q}})$-generic condition. Since $(r,\dot{u})\Vdash
$\textquotedblleft$\exists^{\omega_{1}}\mu\in\dot{B}\left(  t\subseteq
\rho_{0\mu}\right)  ,$ we can find $(r_{0},\dot{u}_{0})$ and $\mu$ such that:

\begin{enumerate}
\item $(r_{0},\dot{u}_{0})\leq(r,\dot{u}).$

\item \textsf{Tr}$(\triangle,\beta)\cap\delta_{\gamma}\subseteq\mu$ and
$\mu<\delta_{\alpha}.$

\item $t\subseteq\rho_{0\mu}.$

\item $(r_{0},\dot{u}_{0})\Vdash$\textquotedblleft$\mu\in\dot{B}%
$\textquotedblright.
\end{enumerate}

\qquad\ \ 

It follows that $\triangle=\triangle\left(  \mu,\beta\right)  $ and $\left[
\mu\beta\right]  =\delta_{\gamma,}$ which $r_{0}$ forces that is not in
$\dot{C}_{gen}.$
\end{proof}

\qquad\ \qquad\ \qquad\ \ \ 

The difficulty in extending this result to the other square-bracket operations
lies in the choice of $\alpha.$ While a natural choice exists for operations
induced by \textsf{C}-sequences, it is unclear what the corresponding choice
should be in the other cases.

\qquad\ \qquad\ \ \ \ \ \qquad\ \ \ 

By first adding a Baumgartner club and then forcing Martin's Axiom
(\textsf{MA) }with a \textsf{c.c.c. }forcing, we obtain the following
consistency result:

\begin{theorem}
The statement \textquotedblleft There exists a \textsf{C}-sequence whose
square-bracket operation fails to have the Ramsey club
property\textquotedblright\ is consistent with \textsf{MA }and an arbitrarily
large continuum.
\end{theorem}

It can be proved that the failure of the $[\cdot\cdot]_{\mathcal{C}}$-Ramsey
club property for the Baumgartner generic club can not be fixed by a forcing
of the form $\mathbb{COL}_{\in}\left(  \kappa\right)  \ast\mathbb{\dot{P}}$,
where $\mathbb{COL}_{\in}\left(  \kappa\right)  $ denotes the $\in$-collapse
of $\kappa$ (see \cite{NotesonForcingAxioms}) and $\mathbb{\dot{P}}$ is a name
for an $\omega$-proper forcing. This result will appear in a forth-comming paper.

\section{Taking care of all square-bracket operations
\label{Taking care of all}}

In the previous section we proved the consistency of the existence of
square-bracket operations that fail to have the Ramsey club property. In fact,
the Baumgartner club witnesses the failure of this property for all operations
from the ground model. We now aim to construct models in which no
square-bracket operation has the Ramsey club property. Note that by Theorem
\ref{Teorema Todo club es Ramsey}, a single club cannot witness this failure
for all such operations simultaneously. To build a model where every
square-bracket operation fails to have the Ramsey club property, we will add
many Baumgartner clubs using Asper\'{o}'s product.

\begin{theorem}
The statement \textquotedblleft No square-bracket operation has the Ramsey
club property\textquotedblright\ is consistent with an arbitrarily large
continuum. \label{Teo todas fallan}
\end{theorem}

\begin{proof}
Let $\kappa\geq\omega_{2}$ be a cardinal such that $\kappa^{\omega_{1}}%
=\kappa.$ We claim that forcing with $\mathbb{BC}^{\left(  \kappa\right)  }$
produces a model as desired. Pick $G\subseteq$ $\mathbb{BC}^{\left(  X\right)
}$ a generic filter. We already know that $V\left[  G\right]  \models
\mathfrak{c}=$ $\kappa$ (see Proposition \ref{Prop BC tamao continuo}). Let
$[\cdot\cdot]$ be a square-bracket operation in$\ V\left[  G\right]  .$ Since
$[\cdot\cdot]$ is an object of size $\omega_{1}$ and $\mathbb{BC}^{\left(
\kappa\right)  }$ has the $\omega_{2}$\textsf{-c.c.}, we can find $\alpha
\in\kappa$ such that $[\cdot\cdot]$ only depends on $\mathbb{BC}^{\left(
X\right)  },$ where $X=\kappa\setminus\left\{  \alpha\right\}  .$ Since
$\mathbb{BC}^{\left(  \kappa\right)  }$ is forcing equivalent to
$\mathbb{BC}^{\left(  X\right)  }\times\mathbb{BC},$ the result follows by
Theorem \ref{Teo BC destruve V}.
\end{proof}

\qquad\ \qquad\ \ \ \ \ \qquad\ \ \ \qquad\ \ \qquad\ \ \ 

We would like to emphasize following result, which follows by the previous
theorem, Lemma 5.4.4. of \cite{Walks} and the results established in the
subsequent sections:

\begin{corollary}
Both of the following statements are consistent (but not at the same time)
with \textsf{ZFC:}

\begin{enumerate}
\item All square-bracket operations have the Ramsey club property.

\item No square-bracket operation has the Ramsey club property.\textsf{ }
\end{enumerate}
\end{corollary}

\qquad\ \qquad\ \ \ \ 

In other words, \textsf{ZFC }is unable to decide the status of the Ramsey club
property for any square-bracket operation.

\qquad\ \qquad\ \qquad\ \qquad\ \ \ \ \qquad\ \ 

Our aim is now to produce a model of Martin's Axiom\textsf{ }in which no
square-bracket operation (induced by a \textsf{C}-sequence) has the Ramsey
club property. The desired model is obtained by first adding many Baumgartner
clubs and then forcing \textsf{MA} in the usual way. To achieve this, we will
require several preliminary technical lemmas.

\qquad\qquad\qquad

Let $X$ be a set of ordinals, $\beta\in X$ and $\dot{B}$ a $\mathbb{BC}%
^{\left(  X\right)  }$-name$.$ We say that $\dot{B}$ \emph{does not depend on
}$\beta$ if $\dot{B}$ is (equivalent) to a $\mathbb{BC}^{\left(
X\setminus\left\{  \beta\right\}  \right)  }$-name. We have the following:

\begin{proposition}
Let $\kappa\geq\omega_{1}$ be a regular cardinal, $\mathbb{\dot{P}}$ a
$\mathbb{BC}^{\left(  \kappa^{+}\right)  }$-name for a partial order of size
$\kappa$ and $\dot{f}$ a $\mathbb{BC}^{\left(  \kappa^{+}\right)  }$-name for
an operation on $\omega_{1}.$ There is $\beta\in\kappa^{+}$ such that:
\label{Prop no depende}

\begin{enumerate}
\item $\mathbb{\dot{P}}$ and $\dot{f}$ do not depend on $\beta.$

\item $\mathbb{BC}^{\left(  \kappa^{+}\right)  }\ast\mathbb{\dot{P}}$ and
$(\mathbb{BC}^{(\kappa^{+}\setminus\left\{  \beta\right\}  )}\ast
\mathbb{\dot{P}})\times\mathbb{BC}$ are forcing equivalent.
\end{enumerate}
\end{proposition}

\begin{proof}
We may assume there is a set $P$ in the ground model such that $\mathbb{\dot
{P}=(}P,\leq_{\mathbb{\dot{P}}}\mathbb{)}.$ For every $x,y\in P$ and
$\alpha,\beta\in\omega_{1}$ define the sets $D\left(  x,y,\alpha,\beta\right)
$ be the set of all $p\in\mathbb{BC}^{\left(  \kappa^{+}\right)  }$ such that
either $p\Vdash$\textquotedblleft$x\leq_{\mathbb{\dot{P}}}y$\textquotedblright%
\ or $p\Vdash$\textquotedblleft$x\nleq_{\mathbb{\dot{P}}}y$\textquotedblright%
\ and there is $\gamma\in\omega_{1}$ such that $p\Vdash$\textquotedblleft%
$\dot{f}\left(  \alpha,\beta\right)  =\gamma$\textquotedblright. It is clear
that each $D\left(  x,y,\alpha,\beta\right)  $ is an open dense subset of
$\mathbb{BC}^{\left(  \kappa^{+}\right)  }.$ Pick $A\left(  x,y,\alpha
,\beta\right)  \subseteq$ $D\left(  x,y,\alpha,\beta\right)  $ a maximal
antichain and let $A$ be the union of all these antichains. Define
$B=\bigcup\{$\textsf{sup}$\left(  p\right)  \mid p\in A\}.$ Since
$\mathbb{BC}^{\left(  \kappa^{+}\right)  }$ has the $\omega_{2}$-\textsf{c.c.,
}it follows that $\left\vert B\right\vert \leq\kappa,$ so we can find
$\beta\in\kappa^{+}\setminus B.$ We claim $\beta$ is as desired. It is clear
that $\mathbb{\dot{P}}$ and $\dot{f}$ do not depend on $\beta.$

\qquad\qquad\qquad\qquad\ \ 

For ease of writing, denote $X=\kappa^{+}\setminus\left\{  \beta\right\}  .$
Define the function $F:(\mathbb{BC}^{(X)}\ast\mathbb{\dot{P}})\times
\mathbb{BC\longrightarrow BC}^{\left(  \kappa^{+}\right)  }\ast\mathbb{\dot
{P}}$ given by $F(p,\dot{x},r)=(p^{\frown}r,\dot{x}),$ where $p^{\frown}%
r\in\mathbb{BC}^{\left(  \kappa^{+}\right)  }$ is such that \textsf{sup}%
$\left(  p^{\frown}r\right)  =$ \textsf{sup}$\left(  p\right)  \cup\left\{
\beta\right\}  ,$ $p\subseteq p^{\frown}r$ and $p^{\frown}r\left(
\beta\right)  =r.$ It is easy to see that $F$ is a dense embedding.
\end{proof}

\qquad\qquad\ \ \ \ \ \ \ \ 

We will need the following result (see Theorem 2.1 of \cite{IteratedForcing}
and the remark below it).

\begin{proposition}
Let $\kappa$ be a regular cardinal, $\mathbb{P}$ a partial order and
$\mathbb{\dot{Q}}$ a $\mathbb{P}$-name for a forcing.
\label{Prop iteracion kcc}

\begin{enumerate}
\item If $\mathbb{P}$ is $\kappa$-\textsf{c.c.} and it forces that
$\mathbb{\dot{Q}}$ is also $\kappa$-\textsf{c.c., }then $\mathbb{P\ast\dot{Q}%
}$ is $\kappa$-\textsf{c.c. }

\item If $\mathbb{P\ast\dot{Q}}$ is $\kappa$-\textsf{c.c., }then $\mathbb{P}$
is $\kappa$-\textsf{c.c.} and forces that $\mathbb{\dot{Q}}$ is $\kappa
$-\textsf{c.c.}
\end{enumerate}
\end{proposition}

\qquad\qquad\ \ \ 

In particular, we get that if $\mathbb{\dot{P}}$ is a $\mathbb{BC}^{\left(
X\right)  }$-name for a \textsf{c.c.c.} partial order, then $\mathbb{BC}%
^{\left(  X\right)  }\ast\mathbb{\dot{P}}$ is $\omega_{2}$\textsf{-c.c. }The
following lemma follows by Proposition \ref{Prop BC tamao continuo} and
standard arguments.

\begin{lemma}
[\textsf{GCH}]Let $\kappa>\omega_{1}$ be a regular cardinal and $X\subseteq$
\textsf{OR }of size\textsf{ }$\kappa.$ The following holds:

\begin{enumerate}
\item $\mathbb{BC}^{\left(  X\right)  }$ has size $\kappa.$

\item $\mathbb{BC}^{\left(  X\right)  }$ forces that \textquotedblleft%
$\mathfrak{c}=2^{<\kappa}=\kappa\wedge2^{\kappa}=\kappa^{+}$\textquotedblright.
\end{enumerate}
\end{lemma}

\qquad\ \qquad\ \ \ \ \qquad\ \ \ \ 

We can now prove the main result of the section:

\begin{theorem}
The statement \textquotedblleft No square-bracket operation induced by a
\textsf{C}-sequence has the Ramsey club property\textquotedblright\ is
consistent with \textsf{MA }and an arbitrarily large continuum. \label{MA}
\end{theorem}

\begin{proof}
Assume \textsf{GCH }in the ground model, pick $\kappa\geq\omega_{1}$ a regular
cardinal and let $G\subseteq$ $\mathbb{BC}^{(\kappa^{+})}$ be a generic
filter. We go to $V\left[  G\right]  .$ In here, we know that
$\mathfrak{c=\kappa}^{+}$ and $2^{<\kappa^{+}}=2^{\kappa}=\kappa^{+}.$ By the
standard proof of Theorem \ref{Teo forzar Martin}, there is a a finite support
iteration $\langle\mathbb{P}_{\alpha},\mathbb{\dot{Q}}_{\alpha}\mid\alpha
\leq\kappa^{+}\rangle$ of \textsf{c.c.c. }forcings such that $\mathbb{P}%
_{\kappa^{+}}$ forces that \textquotedblleft$\mathfrak{m=c=}$ $\kappa^{+}%
$\textquotedblright\ and if $\alpha<\kappa^{+},$ then $\left\vert
\mathbb{P}_{\alpha}\right\vert \leq\kappa.$ For convenience, denote
$\mathbb{P=P}_{\kappa^{+}}.$

\qquad\qquad\qquad\ \ 

We return to $V$ and claim that $\mathbb{Q=BC}^{(\kappa^{+})}\ast
\mathbb{\dot{P}}$ is the partial order we are looking for. We already know
that $\mathbb{Q}$ forces that \textquotedblleft$\mathfrak{m=c=}$ $\kappa^{+}%
$\textquotedblright. It remains to prove that it also forces that no
square-bracket operation has the Ramsey club property . Let $\dot{f}$ be a
$\mathbb{Q}$-name for a square-bracket operation. Since $\mathbb{Q}$ is
$\omega_{2}$-\textsf{.c.c. }and $\kappa^{+}>\omega_{1}$ is regular, we can
find $\alpha<\kappa^{+}$ such that $\dot{f}$ is equivalent to a $\mathbb{BC}%
^{(\kappa^{+})}\ast\mathbb{\dot{P}}_{\alpha}$-name. Let $\mathbb{\dot{R}}$ be
a $\mathbb{BC}^{(\kappa^{+})}\ast\mathbb{\dot{P}}_{\alpha}$-name for the
quotient $\mathbb{P}_{\kappa^{+}}/\mathbb{P}_{\alpha}.$ In this way,
$\mathbb{Q\equiv(BC}^{(\kappa^{+})}\ast\mathbb{\dot{P}}_{\alpha}%
)\ast\mathbb{\dot{R}}$ (where $\mathbb{\equiv}$ denotes being forcing
equivalent). Note that $\mathbb{\dot{R}}$ is forced to be \textsf{c.c.c.} by
Proposition \ref{Prop iteracion kcc}.\ Since $\mathbb{P}_{\alpha}$ is forced
to have size at most $\kappa,$ by Proposition \ref{Prop no depende}, we can
find $\beta\in\kappa^{+}$ such that the following conditions hold:

\begin{enumerate}
\item $\mathbb{\dot{P}}_{\alpha}$ and $\dot{f}$ do not depend on $\beta.$

\item $\mathbb{BC}^{\left(  \kappa^{+}\right)  }\ast\mathbb{\dot{P}}_{\alpha
}\equiv(\mathbb{BC}^{(\kappa^{+}\setminus\left\{  \beta\right\}  )}%
\ast\mathbb{\dot{P}}_{\alpha})\times\mathbb{BC}$.
\end{enumerate}

In this way, we get the following:\medskip

\hfill%
\begin{tabular}
[c]{lll}%
$\mathbb{Q}$ & $=$ & $\mathbb{BC}^{(\kappa^{+})}\ast\mathbb{\dot{P}}$\\
& $\mathbb{\equiv}$ & $\mathbb{(BC}^{(\kappa^{+})}\ast\mathbb{\dot{P}}%
_{\alpha})\ast\mathbb{\dot{R}}$\\
& $\mathbb{\equiv}$ & $((\mathbb{BC}^{(\kappa^{+}\setminus\left\{
\beta\right\}  )}\ast\mathbb{\dot{P}}_{\alpha})\times\mathbb{BC}%
)\ast\mathbb{\dot{R}}$%
\end{tabular}
\hfill\ 

\qquad\ \qquad\ \ \ \qquad\ \ \ \ \ \qquad\ \medskip

It follows that $\dot{f}$ will live in the extension by $\mathbb{BC}%
^{(\kappa^{+}\setminus\left\{  \beta\right\}  )}\ast\mathbb{\dot{P}}_{\alpha
},$ which is then followed by $\mathbb{BC}$ and then a \textsf{c.c.c.
}forcing. We can apply Proposition \ref{Prop ccc no arregla} and conclude that
$\dot{f}$ will not have the Ramsey club property in the final extension.
\end{proof}

\ \ \ \ \ \qquad\qquad\qquad\ \ 

Note that if Proposition \ref{Prop ccc no arregla} could be proved for the
other square-bracket operations, then the argument of Theorem \ref{MA} would
also apply for them.

\section{Square-bracket operations and CH}

Having settled the (non) relationship between Martin's Axiom and the Ramsey
club property for the square-bracket operations, we now turn to the Continuum
Hypothesis. It is easy to see that if \textsf{CH }holds in the ground model,
then adding a single Baumgartner club preserves it. Thus, Theorem
\ref{Teo BC destruve V} implies that \textsf{CH }is consistent with the
existence of square-bracket operations that do not have the Ramsey club
property. The goal of this section is to prove a stronger result: \textsf{CH
}(and even $\Diamond$) is consistent with the failure of the Ramsey club
property for every square-bracket operation. We cannot proceed as in previous
sections, since adding many Baumgartner clubs would force the failure of
\textsf{CH}. Instead, we will use a forcing notion for adding a club with
countable conditions that was introduced by Baumgartner, Harrington, and
Kleinberg in \cite{AddingaClub} (see also \cite{TasteofGoldstern}).

\begin{definition}
Let $\mathbb{SC}$ be the the set of all closed and countable subsets of
$\omega_{1}.$ For $p,q\in\mathbb{SC},$ define $q\leq p$ if $p$ is an initial
segment of $q.$
\end{definition}

\qquad\ \ \ \ 

$\mathbb{SC}$ is the usual forcing for shooting a club through $\omega_{1}$
with countable conditions. Clearly $\mathbb{SC}$ is a $\sigma$-closed tree. If
$G\subseteq\mathbb{SC}$ is a generic filter, the \emph{generic club }is
defined as $C_{gen}=\bigcup G.$ It is not hard to verify that $C_{gen}$ is
forced to be a club on $\omega_{1}.$ The following lemma is easy and well-known.

\begin{lemma}
Let $M$ be a countable elementary submodel. \label{Lemma SC basic generic}

\begin{enumerate}
\item If $p\in\mathbb{SC}$ is $(M,\mathbb{SC})$-generic, then $\delta_{M}\in
p.$

\item Let $p\in\mathbb{SC}\cap M.$ There is $q\leq p$ an $(M,\mathbb{SC}%
)$-generic condition such that $\delta_{M}=$ \textsf{max}$\left(  q\right)  .$
\end{enumerate}
\end{lemma}

\qquad\ \ \qquad\ \ \ 

The paper \cite{IteratedForcingandCH} provides an example of a totally proper
forcing $\mathbb{P}$ for which there exist conditions that are $(M,\mathbb{P}%
)$-generic but not $(M,\mathbb{P})$-totally generic. Fortunately, this
phenomenon does not occur with our forcing $\mathbb{SC}.$

\begin{lemma}
Let $M$ be a countable elementary submodel and $p\in\mathbb{SC}.$ If $p$ is
$(M,\mathbb{SC})$-generic, then it is $(M,\mathbb{SC})$-totally generic.
\end{lemma}

\begin{proof}
Let $D\in M$ be an open dense subset of $\mathbb{SC}.$ By the genericity of
$p,$ there is $q\in D\cap M$ that is compatible with $p.$ Since $\mathbb{SC}$
is a tree, it must be the case that either $p\leq q$ or $q\leq p.$ Since
\textsf{max}$\left(  q\right)  <\delta_{M}\leq$ \textsf{max}$\left(  p\right)
,$ then $q\leq p.$ It follows that $p\in D$ because $D$ is an open subset of
$\mathbb{SC}.$
\end{proof}

\qquad\qquad\ \ \ \ \ \qquad\ \ \ \ \ \qquad\qquad\qquad\ \ \ 

We also have the following:

\begin{lemma}
Let $\left\{  M_{n}\mid n\in\omega\right\}  $ be a continuous increasing chain
of countable elementary submodels and $M=\bigcup\limits_{n\in\omega}M_{n}.$ If
$p\in\mathbb{SC}$ is $(M_{n},\mathbb{SC})$-generic for every $n\in\omega,$
then $p$ is $(M,\mathbb{SC})$-generic. \label{Lema union genericidad}
\end{lemma}

\begin{proof}
We will prove that $p$ is $(M,\mathbb{SC})$-totally generic. Let $D\in M$ be
an open dense subset of $\mathbb{SC}.$ Choose $n\in\omega$ such that $D\in
M_{n}.$ Since $p$ is $(M_{n},\mathbb{SC})$-totally generic, it follows that
$p\in D.$
\end{proof}

\qquad\ \qquad\ \ \ \qquad\ \ \ \ \ \ 

To establish the main results of this section, we must first introduce new
definitions and prove several auxiliary lemmas. We start with some preliminary
results concerning uncountable subsets of the reals.

\begin{definition}
Let $X$ be a subset of $2^{\omega}.$ The \emph{kernel tree of} $X$ (denoted by
$S\left(  X\right)  $) consists of all $s\in2^{<\omega}$ for which the set
$\left\langle s\right\rangle \cap X$ is uncountable.
\end{definition}

\qquad\ \qquad\ \ \ 

We have the following simple observation.

\begin{lemma}
If $X\subseteq2^{\omega}$ is uncountable, then $S\left(  X\right)  $ is a
Sacks tree. \label{Lema es Sacks tree}
\end{lemma}

\qquad\qquad\ \ \ \ 

The following result is also easy, so we leave it to the reader.

\begin{lemma}
Let $X\subseteq2^{\omega}$ be uncountable, $M$ a countable elementary submodel
such that $X\in M.$ If $r\in X\setminus M,$ then $r\in\left[  S\left(
X\right)  \right]  .$ \label{Lemma Kernel modelo}
\end{lemma}

\qquad\ \qquad\ \ \ \qquad\ \ \ \ 

We now introduce the definitions and results concerning Aronszajn trees that
are necessary for the main theorem of this section.

\begin{definition}
Let $\left(  T,\leq_{T}\right)  $ be an Aronszajn tree, $s\in T,$
$p\in\mathbb{SC}$ and $\dot{B}$ an $\mathbb{SC}$-name for an uncountable
subset of $T.$ We say that:

\begin{enumerate}
\item $p$ $\dot{B}$\emph{-kills }$s$ if there is $t<_{T}s$ such that $p\Vdash
$\textquotedblleft$\forall z\in\dot{B}\left(  t\nleq_{T}z\right)
$\textquotedblright.

\item $p$ $\dot{B}$\emph{-freezes }$s$ if for every $t<_{T}s,$ we have that
$p\Vdash$\textquotedblleft$\exists^{\omega_{1}}z\in\dot{B}\left(  t\leq
_{T}z\right)  $\textquotedblright.

\item $p$ $\dot{B}$\emph{-decides }$s$ if $p$ either $\dot{B}$-kills $s$ or
$\dot{B}$-freezes $s.$
\end{enumerate}
\end{definition}

\qquad\ \ \ \qquad\ \ \ \ \ 

It is crucial to note that in both the killing and freezing definitions, the
node $t$ is required to be strictly below $s.$ We have the following lemma.

\begin{lemma}
Let $\left(  T,\leq_{T}\right)  $ be an Aronszajn tree, $\dot{B}$ an
$\mathbb{SC}$-name for an uncountable subset of $T,$ $M$ a countable
elementary submodel with $\dot{B},T\in M,$ $p\in\mathbb{SC}$ an
$(M,\mathbb{SC})$-generic condition and $\delta=\delta_{M}.$ The following
holds: \label{Lema SC decide y congela}

\begin{enumerate}
\item $p$ $\dot{B}$-decides every $s\in T_{\delta}.$

\item There is $s\in T_{\delta}$ such that $p$ $\dot{B}$-freezes $s.$
\end{enumerate}
\end{lemma}

\begin{proof}
We start with the first point. Let $s\in T_{\delta}.$ Since $p$ is
$(M,\mathbb{SC})$-totally generic, for every $t<_{T}s,$ we have that one of
the following statements holds:

\begin{enumerate}
\item $p\Vdash$\textquotedblleft$\exists^{\omega_{1}}z\in\dot{B}\left(
t\leq_{T}z\right)  $\textquotedblright.

\item $p$ forces that the set $\{z\in\dot{B}\mid t\leq_{T}z\}$ is at most countable.
\end{enumerate}

\qquad\ \ \ 

If the first point holds for every $t<_{T}s,$ then we have that $p$ $\dot{B}%
$-freezes $s$ and we are done. Assume there is $t<_{T}s$ such that forces that
the set $\{z\in\dot{B}\mid t\leq_{T}z\}$ is at most countable. Appealing once
again to total genericity, we know that there is $\alpha<\delta$ such that
$p\Vdash$\textquotedblleft$\{\left\vert z\right\vert \mid z\in\dot{B}\wedge
t\leq_{T}z\}\subseteq\alpha$\textquotedblright. It follows that $p\Vdash
$\textquotedblleft$\forall z\in\dot{B}\left(  s\upharpoonright\alpha\nleq
_{T}z\right)  ,$ hence $p$ $\dot{B}$-kills $s.$

\qquad\qquad\qquad\qquad\qquad

We will now prove that $p$ $\dot{B}$-freezes a node at the level $\delta.$
Since $\dot{B}$ is forced to be uncountable, we can find $z\in T$ with
$\left\vert z\right\vert >\delta$ and $q\leq p$ such that $q\Vdash
$\textquotedblleft$z\in\dot{B}$\textquotedblright. Let $s=z\upharpoonright
\delta.$ It follows that $p$ can not $\dot{B}$ kill $s,$ so $p$ $B$-freezes it.
\end{proof}

\qquad\ \ \qquad\ \ 

Let $e=\left\langle e_{\beta}\mid\beta\in\omega_{1}\right\rangle $ be a nice
sequence. Define the tree $T\left(  e\right)  =\{e_{\beta}\upharpoonright
\alpha\mid\alpha\leq\beta<\omega_{1}\}.$ We say that a square-bracket
operation $[\cdot\cdot]$ is an \emph{Aronszajn square-bracket operation} if it
is of one of the following types:

\begin{enumerate}
\item $[\cdot\cdot]_{\mathcal{C}}$ where $\mathcal{C}$ is a \textsf{C}-sequence.

\item $[\cdot\cdot]_{T,a}$ where $T$ is a Hausdorff, special Aronszajn tree
and $a:T\longrightarrow\omega$ is a specialization mapping.

\item $[\cdot\cdot]_{\mathcal{R},e}$ where $\mathcal{R}=\left\{  r_{\alpha
}\mid\alpha\in\omega_{1}\right\}  \subseteq2^{\omega}$ and $e=\left\langle
e_{\beta}\mid\beta\in\omega_{1}\right\rangle $ is a nice sequence for which
$T\left(  e\right)  $ is an Aronszajn tree.
\end{enumerate}

\qquad\ \ \qquad\ \ 

We can now prove the following:

\begin{proposition}
$\mathbb{SC}$ forces that the generic club is not $[\cdot\cdot]$-Ramsey for
any Aronszajn square-bracket operation $[\cdot\cdot]$ in $V.$
\label{Prop SC destruye a V}
\end{proposition}

\begin{proof}
Let $p\in\mathbb{SC}$ and $\dot{B}$ such that $p\Vdash$\textquotedblleft%
$\dot{B}\in\left[  \omega_{1}\right]  ^{\omega_{1}}$\textquotedblright. We
want to extend $p$ to a condition forcing that there are $\alpha,\beta\in
\dot{B}$ such that $\left[  \alpha\beta\right]  $ is not in $\dot{C}_{gen}$.
Let $\left\langle M_{\alpha}\mid\alpha\leq\omega^{2}\right\rangle $ be a
continuous chain of countable elementary submodels such that $\dot{B}\in
M_{0},$ as well as the square-bracket operation we want to take care of. For
convenience, we introduce the following notation:

\begin{enumerate}
\item $M=M_{\omega^{2}}.$

\item $\delta=\delta_{M}.$

\item $\delta_{\alpha}=\delta_{M_{\alpha}}$ for $\alpha<\omega^{2}.$
\end{enumerate}

\qquad\ \ \qquad\ \ \ \ \ \ \qquad\qquad\qquad\qquad\ \ \ \ 

The proof now splits into cases depending on the type of square-bracket
operation under consideration. We begin with the case of the operation induced
by a \textsf{C}-sequence, say $\mathcal{C}.$ We assume that $\mathcal{C}\in
M_{0}.$ For notational convenience in this part of the proof, we will write
$T$ for $T(\rho_{0}),$ $\left[  \gamma\xi\right]  $ for $\left[  \gamma
\xi\right]  _{\mathcal{C}}$ and $\triangle\left(  \gamma,\xi\right)  $ for
$\triangle_{\mathcal{C}}\left(  \gamma,\xi\right)  $ (for any $\gamma,\xi
\in\omega_{1}$). To simplify notation, a subset $W\subseteq\omega_{1}$ will
occasionally be identified with the corresponding set $\{\rho_{0\alpha}%
\mid\alpha\in W\}.$

\qquad\ \qquad\ \ \ \ \ \qquad\ \ 

Fix an enumeration $T_{\delta}=\left\{  s_{n}\mid n\in\omega\right\}  $ such
that each element of $T_{\delta}$ is listed infinitely many times. We will now
recursively construct a sequence $\langle\varepsilon_{n},\gamma_{n},\mu
_{n},q_{n},z_{n},t_{n},\triangle_{n}\rangle_{n\in\omega}$ satisfying the
following properties\footnote{The reader may wish to skip this lengthy list
initially, as the motivation for these properties will become clear in the
next paragraphs.}:

\begin{enumerate}
\item $\omega n<\varepsilon_{n}<\mu_{n}<\gamma_{n}<\omega^{2}.$

\item $\gamma_{n}<\varepsilon_{n+1}.$

\item $q_{n+1}\leq q_{n}\leq p$ and $q_{n}\in M.$

\item $q_{n}$ is both $(M_{\varepsilon_{n}},\mathbb{SC})$-generic and
$(M_{\mu_{n}},\mathbb{SC})$-generic.

\item \textsf{max}$\left(  q_{n}\right)  =\delta_{\gamma_{n}}+1$ and
$\delta_{\gamma_{n}}\notin q_{n}$ (in particular, $q_{n}$ is not
$(M_{\gamma_{n}},\mathbb{SC})$-generic).

\item $z_{n}\in T_{\delta_{\varepsilon_{n}}}$ and $q_{n}$ $\dot{B}$-freezes
it. Moreover, if $q_{n}$ $\dot{B}$-freezes $s_{n}\upharpoonright
\delta_{\varepsilon_{n}},$ then $z_{n}=s_{n}\upharpoonright\delta
_{\varepsilon_{n}}.$

\item $t_{n}\in M_{\varepsilon_{n}}.$

\item $t_{n}$ and $z_{n}$ are incompatible nodes of $T.$

\item $\triangle_{n}=\triangle\left(  z_{n},t_{n}\right)  $.

\item $\lambda\left(  \delta_{\gamma_{n},}\delta\right)  <\triangle_{n}$

\item \textsf{Tr}$(\triangle_{n},\delta)=$ \textsf{Tr}$(\delta_{\gamma_{n}%
},\delta)\cup$ \textsf{Tr}$(\triangle_{n},\delta_{\gamma_{n}}).$

\item \textsf{Tr}$(\triangle_{n},\delta_{\gamma_{n}})\setminus\left\{
\delta_{\gamma_{n}}\right\}  \subseteq\delta_{\mu_{n}}$

\item $q_{n}\Vdash$\textquotedblleft$\exists^{\omega_{1}}\alpha\in\dot
{B}\left(  t\subseteq\rho_{0\alpha}\right)  $\textquotedblright.
\end{enumerate}

\qquad\ \ \ \ \qquad\ \qquad\ \ \ \ \ \ 

The construction proceeds by recursion. For notational convenience, let
$q_{-1}=p,$ $\varepsilon_{-1}=\mu_{-1}=\gamma_{-1}=0.$ Suppose we have
constructed the sequence up to stage $n-1,$ we now define the items for stage
$n.$ Since $\omega^{2}$ is a limit of limit ordinals, we can find $\gamma
_{n}<\omega^{2}$ a limit ordinal such that $\gamma_{n-1},$ $\omega
n<\gamma_{n}$ and $q_{n-1}\in M_{\gamma_{n}}.$ Now, since $\gamma_{n}$ is
limit, it is possible to find $\varepsilon_{n}<\gamma_{n}$ such that
$\gamma_{n-1},\omega n<\varepsilon_{n}$ and $\lambda\left(  \delta_{\gamma
_{n},}\delta\right)  ,$ $q_{n-1}\in M_{\varepsilon_{n}}.$ We can now apply
Lemma \ref{Lemma SC basic generic} and find $p_{n}\in M_{\gamma_{n}}$ an
$(M_{\varepsilon_{n}},\mathbb{SC})$-generic condition extending $q_{n-1}$ such
that \textsf{max}$\left(  p_{n}\right)  =\delta_{\varepsilon_{n}}.$

\qquad\qquad\qquad\ \ \ \ 

Since $p_{n}$ is $(M_{\varepsilon_{n}},\mathbb{SC})$-generic, Lemma
\ref{Lema SC decide y congela} implies that $p_{n}$ $\dot{B}$-decides
$s_{n}\upharpoonright\delta_{\varepsilon_{n}}.$ If $p_{n}$ $\dot{B}$-freezes
$s_{n}\upharpoonright\delta_{\varepsilon_{n}},$ we set $z_{n}=s_{n}%
\upharpoonright\delta_{\varepsilon_{n}}.$ If instead $p_{n}$ $\dot{B}$-kills
$s_{n}\upharpoonright\delta_{\varepsilon_{n}},$ we let $z_{n}$ be any node in
$T_{\delta_{\varepsilon_{n}}}$ that is $\dot{B}$-frozen by $p_{n}$ (the
existence of such a node is guaranteed by Lemma \ref{Lema SC decide y congela}%
). Find $\eta<\delta_{\varepsilon_{n}}$ such that $\lambda\left(
\delta_{\gamma_{n},}\delta\right)  ,$ $\lambda\left(  \delta_{\varepsilon_{n}%
},\delta_{\gamma_{n},}\right)  <\eta.$ Since $p_{n}$ $\dot{B}$-freezes
$z_{n},$ it follows that $p_{n}\Vdash$\textquotedblleft$\exists^{\omega_{1}%
}\alpha\in\dot{B}\left(  z_{n}\upharpoonright\eta\subseteq\rho_{0\alpha
}\right)  $\textquotedblright. By Proposition \ref{Prop Aronszajn y submodelo}
and since $p_{n}$ is $(M_{\varepsilon_{n}},\mathbb{SC})$-totally generic, it
is possible to find $t_{n}\in T$ with the following properties:

\begin{enumerate}
\item $t_{n}\in M_{\varepsilon_{n}}.$

\item $z_{n}$ and $t_{n}$ are incompatible nodes of $T.$

\item $z_{n}\upharpoonright\eta\subseteq t_{n}.$

\item $p_{n}\Vdash$\textquotedblleft$\exists^{\omega_{1}}\alpha\in\dot
{B}\left(  t\subseteq\rho_{0\alpha}\right)  $\textquotedblright.
\end{enumerate}

\qquad\ \ \ 

Let $\triangle_{n}=\triangle\left(  z_{n},t_{n}\right)  .$ Clearly $\eta
\leq\triangle_{n},$ so in particular we have that $\lambda\left(
\delta_{\gamma_{n},}\delta\right)  ,$ $\lambda\left(  \delta_{\varepsilon_{n}%
},\delta_{\gamma_{n},}\right)  <\triangle_{n}.$ We can now apply Proposition
\ref{Prop juntar trazas} and conclude that \newline\textsf{Tr}$(\triangle
_{n},\delta)=$ \textsf{Tr}$(\delta_{\gamma_{n}},\delta)\cup$\textsf{Tr}%
$(\triangle_{n},\delta_{\gamma_{n}}).$ Since $\gamma_{n}$ is a limit ordinal,
we can find $\mu_{n}<\gamma_{n}$ such that \textsf{Tr}$(\triangle_{n}%
,\delta_{\gamma_{n}})\setminus\left\{  \delta_{\gamma_{n}}\right\}
\subseteq\delta_{\mu_{n}},$ $\varepsilon_{n}<\mu_{n}$ and $p_{n}\in
M_{\gamma_{n}}.$ By Lemma \ref{Lemma SC basic generic}, there is $\overline
{p}_{n}\in M_{\gamma_{n}}$ an $(M_{\mu_{n}},\mathbb{SC})$-generic condition
extending $p_{n}.$ Finally, let $q_{n}=\overline{p}_{n}\cup\left\{
\delta_{\gamma_{n}}+1\right\}  .$ This finishes the construction.

\qquad\qquad\qquad\ \ \ 

Define $q=\bigcup\limits_{n\in\omega}q_{n}\cup\left\{  \delta\right\}  .$ It
is straightforward to verify that $q\in\mathbb{SC}$ and extends $p.$ Moreover,
$q$ is $(M,\mathbb{SC})$-generic because it is $(M_{\varepsilon_{n}%
},\mathbb{SC})$-generic for every $n\in\omega$ (see Lemma
\ref{Lema union genericidad}). Since $\dot{B}$ is forced to be uncountable, we
can find $r\leq q$ and $\beta>\delta$ such that $r\Vdash$\textquotedblleft%
$\beta\in\dot{B}$\textquotedblright.

\qquad\qquad\ \qquad\ \ \ 

Lemma \ref{Lemma lambas converjen} implies that the sequence $\langle
\lambda(\delta_{\gamma_{m}},\delta)\rangle_{m\in\omega}$ converges to
$\delta.$ In particular, $\lambda(\delta_{\gamma_{m}},\delta)>\lambda\left(
\delta,\beta\right)  $ holds for almost all $m\in\omega.$ Moreover, since in
our enumeration $T_{\delta}=\left\{  s_{n}\mid n\in\omega\right\}  $ every
element appears infinitely many often, it is possible to find an $n\in\omega$
such that $\lambda\left(  \delta,\beta\right)  <\lambda(\delta_{\gamma_{n}%
},\delta)$ and $s_{n}=\rho_{0\beta}\upharpoonright\delta.$ Given that $q$
$\dot{B}$-freezes $s_{n},$ we conclude that $q_{n}$ $\dot{B}$-freezes
$s_{n}\upharpoonright\delta_{\varepsilon_{n}},$ and therefore $z_{n}%
=s_{n}\upharpoonright\delta_{\varepsilon_{n}}.$ Since $r$ is $(M_{\mu_{n}%
},\mathbb{SC})$-totally generic, there is $\alpha<\delta_{\mu_{n}}$ such that
$r\Vdash$\textquotedblleft$\alpha\in\dot{B}$\textquotedblright, $t_{n}%
\subseteq\rho_{0\alpha}$ and \textsf{Tr}$(\triangle_{n},\delta_{\gamma_{n}%
})\setminus\left\{  \delta_{\gamma_{n}}\right\}  =$ \textsf{Tr}$(\triangle
_{n},\delta_{\gamma_{n}})\cap M_{\mu_{n}}\subseteq\alpha.$ We will now compute
$\left[  \alpha\beta\right]  .$ First note that:\bigskip

\hfill%
\begin{tabular}
[c]{lll}%
$\triangle\left(  \alpha,\beta\right)  $ & $=$ & $\triangle\left(  t_{n}%
,s_{n}\right)  $\\
& $=$ & $\triangle\left(  t_{n},z_{n}\right)  $\\
& $=$ & $\triangle_{n}$%
\end{tabular}
\hfill\ 

\qquad\ \ \ \qquad\qquad\ \ \ \qquad\medskip

Recall that $\lambda\left(  \delta,\beta\right)  <\lambda(\delta_{\gamma_{n}%
},\delta)<\triangle_{n}.$ By Proposition \ref{Prop juntar trazas}:\bigskip

\hfill%
\begin{tabular}
[c]{lll}%
\textsf{Tr}$(\triangle_{n},\beta)$ & $=$ & \textsf{Tr}$(\delta,\beta)\cup$
\textsf{Tr}$(\triangle_{n},\delta)$\\
& $=$ & \textsf{Tr}$(\delta,\beta)\cup$ \textsf{Tr}$(\delta_{\gamma_{n}%
},\delta)\cup$ \textsf{Tr}$(\triangle_{n},\delta_{\gamma_{n}})$%
\end{tabular}
\hfill\ 

\ \ \ \ \ \medskip

Since \textsf{Tr}$(\triangle_{n},\delta_{\gamma_{n}})\setminus\left\{
\delta_{\gamma_{n}}\right\}  \subseteq\alpha<\delta_{\gamma_{n}},$ it follows
that \newline$\left[  \alpha\beta\right]  =$ \textsf{min}$\{$\textsf{Tr}%
$(\triangle_{n},\beta)\setminus\alpha\}=\delta_{\gamma_{n}}.$ Finally, recall
that $r\Vdash$\textquotedblleft$\alpha,\beta\in\dot{B}$\textquotedblright\ and
$\delta_{\gamma_{n}}\notin r.$ This concludes the proof for the square-bracket
operation induced by a \textsf{C}-sequence.

\qquad\qquad\ \qquad\ \ \ \ 

We now consider the case of the operation induced by a sequence of reals
$\mathcal{R}=\left\{  r_{\alpha}\mid\alpha\in\omega_{1}\right\}
\subseteq2^{\omega}$ and a nice sequence $e=\left\langle e_{\beta}\mid\beta
\in\omega_{1}\right\rangle $ for which $T=T\left(  e\right)  $ is an Aronszajn
tree. We are assuming that $\mathcal{R},e\in M_{0}.$ For notational
convenience in this part of the proof, we will write $\left[  \gamma
\xi\right]  $ for $\left[  \gamma\xi\right]  _{\mathcal{R},e}$ and
$\triangle\left(  \gamma,\xi\right)  $ for $\triangle_{\mathcal{R}}\left(
\gamma,\xi\right)  $ (for any $\gamma,\xi\in\omega_{1}$). Fix an enumeration
$T_{\delta}=\left\{  s_{n}\mid n\in\omega\right\}  .$

\qquad\ \qquad\ \ \ \ \ \qquad\ \ \ \ 

Let $\dot{X}$ be the $\mathbb{SC}$-name for the set $\{r_{\alpha}\mid\alpha
\in\dot{B}\}.$ By Lemma \ref{Lema es Sacks tree}, its kernel tree is forced to
be a Sacks tree. Furthermore, since $\mathbb{SC}$ is $\sigma$-closed, (by
extending $p$ if needed) we may assume that there is a Sacks tree $S$ such
that $p\Vdash$\textquotedblleft$S=S(\dot{X})$\textquotedblright. Since both
$p$ and $\dot{B}$ are in $M_{0},$ it follows that $S\in M_{0}.$ We will now
recursively construct a sequence $\langle\gamma_{n},\mu_{n},q_{n},m_{n}%
\rangle_{n\in\omega}$ satisfying the following properties:

\begin{enumerate}
\item $\omega n<\mu_{n}<\gamma_{n}<\omega^{2}.$

\item $\gamma_{n}<\mu_{n+1}.$

\item $q_{n+1}\leq q_{n}\leq p$ and $q_{n}\in M.$

\item $q_{n}$ is $(M_{\mu_{n}},\mathbb{SC})$-generic.

\item \textsf{max}$\left(  q_{n}\right)  =\delta_{\gamma_{n}}+1$ and
$\delta_{\gamma_{n}}\notin q_{n}$ (in particular, $q_{n}$ is not
$(M_{\gamma_{n}},\mathbb{SC})$-generic).

\item $s_{n}\left(  \delta_{\gamma_{n}}\right)  <m_{n}$ and\ \textsf{split}%
$_{s_{n}\left(  \delta_{\gamma_{n}}\right)  }\left(  S\right)  \subseteq
2^{<m_{n}}.$

\item $\{\xi\in M_{\gamma_{n}}\mid s_{n}\left(  \xi\right)  \leq
m_{n}\}\subseteq\delta_{\mu_{n}}.$
\end{enumerate}

\qquad\ \ \ \qquad\ \ \ 

The construction proceeds by recursion. For notational convenience, let
$q_{-1}=p,\mu_{-1}=\gamma_{-1}=0.$ Suppose we have constructed the sequence up
to stage $n-1,$ we now define the items for stage $n.$ Since $\omega^{2}$ is a
limit of limit ordinals, we can find $\gamma_{n}<\omega^{2}$ a limit ordinal
such that $\gamma_{n-1},$ $\omega n<\gamma_{n}$ and $q_{n-1}\in M_{\gamma_{n}%
}.$ Choose $m_{n}$ such that \textsf{split}$_{s_{n}\left(  \delta_{\gamma_{n}%
}\right)  }\left(  S\right)  \subseteq2^{<m_{n}}$. Note that $s_{n}\left(
\delta_{\gamma_{n}}\right)  <m_{n}.$ Since $\gamma_{n}$ is limit, we can now
find $\mu_{n}<\gamma_{n}$ such that $\gamma_{n-1},\omega n<\mu_{n},$
$q_{n-1}\in M_{\mu_{n}}$ and $\{\xi\in M_{\gamma_{n}}\mid s_{n}\left(
\xi\right)  \leq m_{n}\}\subseteq\delta_{\mu_{n}}.$ This last requirement is
possible because $s_{n}$ is an injective function. We can now apply Lemma
\ref{Lemma SC basic generic} and find $p_{n}\in M_{\gamma_{n}}$ an
$(M_{\mu_{n}},\mathbb{SC})$-generic condition extending $q_{n-1}$. Define
$q_{n}=p_{n}\cup\{\delta_{\gamma_{n}}+1\}.$ This finishes the construction.

\qquad\ \qquad\ \qquad\ \qquad\ 

Define $q=\bigcup\limits_{n\in\omega}q_{n}\cup\left\{  \delta\right\}  .$
Since $q$ is $(M_{\mu_{n}},\mathbb{SC})$-generic for every $n\in\omega,$ Lemma
\ref{Lema union genericidad} implies that $q$ is $(M,\mathbb{SC})$-generic.
Choose $r\leq q$ and $\beta>\delta$ such that $r\Vdash$\textquotedblleft%
$\beta\in\dot{B}$\textquotedblright. It follows by Lemma
\ref{Lemma Kernel modelo} that $r_{\beta}\in\left[  S\right]  .$

\qquad\ \qquad\ \qquad\ \qquad\ \ \ 

Let $n\in\omega$ be such that $s_{n}=e_{\beta}\upharpoonright\delta$ and $t\in
S$ be the only node satisfying $t\subseteq r_{\beta}$ and $t\in$
\textsf{split}$_{s_{n}\left(  \delta_{\gamma_{n}}\right)  }\left(  S\right)
.$ Define $\triangle=\left\vert t\right\vert $ and note that $s_{n}\left(
\delta_{\gamma_{n}}\right)  \leq\triangle<m_{n}$ . Since $t$ is a splitting
node of $S,$ the node $z=t^{\frown}\left(  1-r_{\beta}\left(  \triangle
\right)  \right)  $ also belongs to $S.$ Consequently, we have that $r\Vdash
$\textquotedblleft$\exists^{\omega_{1}}\alpha\in\dot{B}\left(  z\subseteq
r_{\alpha}\right)  $\textquotedblright. We now use the fact that $r$ is
$(M_{\mu_{n}},\mathbb{SC})$-totally generic to find $\alpha\in M_{\mu_{n}}$
such that $z\subseteq r_{\alpha},$ $r\Vdash$\textquotedblleft$\alpha\in\dot
{B}$\textquotedblright\ and $\{\xi<\delta_{\mu_{n}}\mid e_{\beta}\left(
\xi\right)  \leq m_{n}\}\subseteq\alpha.$

\qquad\qquad\qquad\qquad\qquad\ \ \ 

We claim that $\left[  \alpha\beta\right]  =\delta_{\mu_{n}}.$ Recall that
$\left[  \alpha\beta\right]  $ is the least $\xi\in\left[  \alpha
,\beta\right]  $ for which $e_{\beta}\left(  \xi\right)  \leq\triangle\left(
\alpha,\beta\right)  =\triangle.$ On one hand, we already knew that $e_{\beta
}(\delta_{\gamma_{n}})=s_{n}\left(  \delta_{\gamma_{n}}\right)  \leq\triangle
$. On the other hand, for $\xi<\delta_{\gamma_{n}}$ with $e_{\beta}\left(
\xi\right)  \leq\triangle,$ we have the following:\medskip

\hfill%
\begin{tabular}
[c]{lll}%
$e_{\beta}\left(  \xi\right)  \leq\triangle$ & $\longrightarrow$ & $e_{\beta
}\left(  \xi\right)  <m_{n}$\\
& $\longrightarrow$ & $s_{n}\left(  \xi\right)  <m_{n}$\\
& $\longrightarrow$ & $\xi<\delta_{\mu_{n}}$%
\end{tabular}
\hfill\ 

\qquad\ \qquad\ \ 

Since $\xi<\delta_{\mu_{n}}$ and $e_{\beta}\left(  \xi\right)  \leq m_{n},$ it
follows that $\xi<\alpha.$ We conclude that $\left[  \alpha\beta\right]
=\delta_{\mu_{n}}.$ Finally, recall that $r\Vdash$\textquotedblleft%
$\alpha,\beta\in\dot{B}$\textquotedblright\ and $\delta_{\gamma_{n}}\notin r.$
This concludes the proof for this square-bracket operation.

\qquad\ \qquad\ \qquad\ \ \ \ \qquad\ 

The last case is when the operation is induced by a Hausdorff, special
Aronszajn tree $T$ and $a:T\longrightarrow\omega$ a specialization mapping.
For notational convenience, we now write $\left[  st\right]  $ for $\left[
st\right]  _{T,r}$. We can assume that $T,a\in M_{0}.$ One more change is
needed, now $\dot{B}$ is forced to be an uncountable subset of $T.$ Fix an
enumeration $T_{\delta}=\left\{  s_{n}\mid n\in\omega\right\}  $. We will now
recursively construct a sequence $\langle\varepsilon_{n},\gamma_{n},\mu
_{n},q_{n},z_{n},w_{n},m_{n},\triangle_{n}\rangle_{n\in\omega}$ satisfying the
following properties:

\begin{enumerate}
\item $\omega n<\varepsilon_{n}<\mu_{n}<\gamma_{n}<\omega^{2}.$

\item $\gamma_{n}<\varepsilon_{n+1}.$

\item $q_{n+1}\leq q_{n}\leq p$ and $q_{n}\in M.$

\item $q_{n}$ is both $(M_{\varepsilon_{n}},\mathbb{SC})$-generic and
$(M_{\mu_{n}},\mathbb{SC})$-generic.

\item \textsf{max}$\left(  q_{n}\right)  =\delta_{\gamma_{n}}+1$ and
$\delta_{\gamma_{n}}\notin q_{n}$ (in particular, $q_{n}$ is not
$(M_{\gamma_{n}},\mathbb{SC})$-generic).

\item $m_{n}=a(s_{n}\upharpoonright\delta_{\gamma_{n}}).$

\item $z_{n}\in T_{\delta_{\varepsilon_{n}}}$ and $q_{n}$ $\dot{B}$-freezes
it. Moreover, if $q_{n}$ $\dot{B}$-freezes $s_{n}\upharpoonright
\delta_{\varepsilon_{n}},$ then $z_{n}=s_{n}\upharpoonright\delta
_{\varepsilon_{n}}.$

\item $w_{n}\in T\cap M_{\varepsilon_{n}}$ and is incompatible with $z_{n}.$

\item $w_{n}\Vdash$\textquotedblleft$\exists^{\omega_{1}}t\in\dot{B}\left(
w_{n}\subseteq t\right)  $\textquotedblright.

\item $\triangle_{n}=\triangle\left(  z_{n},w_{n}\right)  .$

\item $\left\{  \xi<\delta_{\gamma_{n}}\mid a\left(  s_{n}\upharpoonright
\xi\right)  \leq m_{n}\right\}  \subseteq\triangle_{n}.$

\item $\left\{  \xi<\delta_{\gamma_{n}}\mid a\left(  s_{n}\upharpoonright
\xi\right)  \leq a\left(  z_{n}\wedge w_{n}\right)  \right\}  \subseteq
\delta_{\mu_{n}}.$
\end{enumerate}

\qquad\ \qquad\ \qquad\ \qquad\ \qquad\ 

The construction proceeds by recursion. For notational convenience, let
$q_{-1}=p,\varepsilon_{-1}=\mu_{-1}=\gamma_{-1}=0.$ Suppose we have
constructed the sequence up to stage $n-1,$ we now define the items for stage
$n.$ Since $\omega^{2}$ is a limit of limit ordinals, we can find $\gamma
_{n}<\omega^{2}$ a limit ordinal such that $\gamma_{n-1},$ $\omega
n<\gamma_{n}$ and $q_{n-1}\in M_{\gamma_{n}}.$ Define $m_{n}$ $=a(s_{n}%
\upharpoonright\delta_{\gamma_{n}})$ and find $\eta<\delta_{\gamma_{n}}$ such
that $\left\{  \xi<\delta_{\gamma_{n}}\mid a\left(  s_{n}\upharpoonright
\xi\right)  \leq m_{n}\right\}  \subseteq\eta.$ Since $\gamma_{n}$ is a limit
ordinal, we can find $\varepsilon_{n}<\gamma_{n}$ such that $\gamma_{n-1},$
$\omega n<\varepsilon_{n}$ and $\eta,q_{n-1}\in M_{\varepsilon_{n}}.$ We can
now apply Lemma \ref{Lemma SC basic generic} and find $p_{n}\in M_{\gamma_{n}%
}$ an $(M_{\varepsilon_{n}},\mathbb{SC})$-generic condition extending
$q_{n-1}$ such that \textsf{max}$\left(  p_{n}\right)  =\delta_{\varepsilon
_{n}}.$ Since $p_{n}$ is $(M_{\varepsilon_{n}},\mathbb{SC})$-generic, Lemma
\ref{Lema SC decide y congela} implies that $p_{n}$ $\dot{B}$-decides
$s_{n}\upharpoonright\delta_{\varepsilon_{n}}.$ If $p_{n}$ $\dot{B}$-freezes
$s_{n}\upharpoonright\delta_{\varepsilon_{n}},$ we set $z_{n}=s_{n}%
\upharpoonright\delta_{\varepsilon_{n}}.$ If instead $p_{n}$ $\dot{B}$-kills
$s_{n}\upharpoonright\delta_{\varepsilon_{n}},$ we let $z_{n}$ be any node in
$T_{\delta_{\varepsilon_{n}}}$ that is $\dot{B}$-frozen by $p_{n}$ (the
existence of such a node is guaranteed by Lemma \ref{Lema SC decide y congela}).

\qquad\qquad\qquad

Recall that $p_{n}\Vdash$\textquotedblleft$\exists^{\omega_{1}}t\in\dot
{B}\left(  z_{n}\upharpoonright\eta\subseteq t\right)  $\textquotedblright%
\ and $p_{n}$ is an $(M_{\varepsilon_{n}},\mathbb{SC})$-totally generic
condition. We can use Proposition \ref{Prop Aronszajn no numerable} to find
two incompatible nodes $w^{0},w^{1}\in T\cap M_{\varepsilon_{n}}$ such that
$z_{n}\upharpoonright\eta\subseteq w^{0},w^{1}$ and $p_{n}\Vdash
$\textquotedblleft$\exists^{\omega_{1}}t\in\dot{B}\left(  w^{i}\subseteq
t\right)  $\textquotedblright\ for each $i\in2.$ Pick any $w_{n}\in\left\{
w^{0},w^{1}\right\}  $ that is incompatible with $z_{n}.$ Define
$\triangle_{n}=\triangle\left(  z_{n},w_{n}\right)  $ and note that
$\eta<\triangle_{n},$ so $\left\{  \xi<\delta_{\gamma_{n}}\mid a\left(
s_{n}\upharpoonright\xi\right)  \leq m_{n}\right\}  \subseteq\triangle_{n}.$
Since $\gamma_{n}$ is a limit ordinal, we can find $\mu_{n}<\gamma_{n}$ such
that $\varepsilon_{n}<\mu_{n},$ $p_{n}\in M_{\mu_{n}}$ and $\left\{
\xi<\delta_{\gamma_{n}}\mid a\left(  s_{n}\upharpoonright\xi\right)  \leq
a\left(  z_{n}\wedge w_{n}\right)  \right\}  \subseteq\delta_{\mu_{n}}.$ By
Lemma \ref{Lemma SC basic generic}, there is $\overline{p}_{n}\in
M_{\gamma_{n}}$ an $(M_{\mu_{n}},\mathbb{SC})$-generic condition extending
$p_{n}.$ We now define $q_{n}=\overline{p}_{n}\cup\left\{  \delta_{\gamma_{n}%
}+1\right\}  $ and finish the recursive construction.

\qquad\ \qquad\ \qquad\ 

Define $q=\bigcup\limits_{n\in\omega}q_{n}\cup\left\{  \delta\right\}  ,$
which is an $(M,\mathbb{SC})$-generic condition by Lemma
\ref{Lema union genericidad}. Choose $r\leq q$ and $s\in T$ with $\left\vert
s\right\vert >\delta$ such that $r\Vdash$\textquotedblleft$s\in\dot{B}%
$\textquotedblright. Let $n\in\omega$ such that $s_{n}=s\upharpoonright
\delta.$ Given that $r$ $\dot{B}$-freezes $s_{n},$ we conclude that $q_{n}$
$\dot{B}$-freezes $s_{n}\upharpoonright\delta_{\varepsilon_{n}}.$ Hence
$z_{n}=s_{n}\upharpoonright\delta_{\varepsilon_{n}}.$ Since $r$ is
$(M_{\mu_{n}},\mathbb{SC})$-totally generic, it follows that there is $t\in
T\cap M_{\mu_{n}}$ such that $w_{n}<_{T}t,$ $r\Vdash$\textquotedblleft%
$t\in\dot{B}$\textquotedblright\ and $\left\{  \xi<\delta_{\gamma_{n}}\mid
a\left(  s_{n}\upharpoonright\xi\right)  \leq a\left(  z_{n}\wedge
w_{n}\right)  \right\}  \subseteq\left\vert t\right\vert .$

\qquad\ \qquad\ \ \ \qquad\ \ 

We claim now that $\left[  ts\right]  =s\upharpoonright\delta_{\gamma_{n}}.$
Recall that $m_{n}=a\left(  s\upharpoonright\delta_{\gamma_{n}}\right)  $ and
note that $\triangle\left(  t,s\right)  =\triangle\left(  z_{n},w_{n}\right)
=\triangle_{n},$ so $t\wedge s$ is equal to $z_{n}\wedge w_{n}.$ Furthermore,
$\left\{  \xi<\delta_{\gamma_{n}}\mid a\left(  s_{n}\upharpoonright\xi\right)
\leq m_{n}\right\}  \subseteq\triangle_{n}$ and $m_{n}\leq a\left(  t\wedge
s\right)  .$ It then follows that $a\left(  s\upharpoonright\delta_{\gamma
_{n}}\right)  \leq a\left(  t\wedge s\right)  $. Moreover, if $\xi
<\delta_{\gamma_{n}}$ and $a\left(  s_{n}\upharpoonright\xi\right)
\leq\triangle_{n},$ then $\xi<\left\vert t\right\vert .$ We conclude that
$\left[  ts\right]  =s\upharpoonright\delta_{\gamma_{n}}.$ Finally, recall
that $r\Vdash$\textquotedblleft$\alpha,\beta\in\dot{B}$\textquotedblright\ and
$\delta_{\gamma_{n}}\notin r.$ This concludes the proof.
\end{proof}

\qquad\qquad\ \ \ \ \ \ \qquad\ \ \ \ \ \qquad\ \ \ \ \ \ \qquad\ \ \ \ \ 

We will now prove that a $\sigma$-closed forcing can not fix the failure of
the Ramsey club property.

\begin{lemma}
Let $[\cdot\cdot]$ be a square-bracket operation and $C\subseteq\omega_{1}$ a
club that is not $[\cdot\cdot]$-Ramsey. If $\mathbb{P}$ is a $\sigma$-closed
forcing, then $\mathbb{P}$ forces that $C$ is not $[\cdot\cdot]$-Ramsey.
\label{Lemma sigma closed no corrige}
\end{lemma}

\begin{proof}
Assume that there are $p\in\mathbb{P}$ and $\dot{B}$ a $\mathbb{P}$-name for
an uncountable subset of $\omega_{1}$ for which $p\Vdash$\textquotedblleft%
$\{\left[  \alpha\beta\right]  \mid\alpha,\beta\in\dot{B}\}\subseteq
C$\textquotedblright. Recursively, find $\left\{  \left(  p_{\alpha}%
,\gamma_{\alpha}\right)  \mid\alpha\in\omega_{1}\right\}  $ such that for
every $\alpha<\beta<\omega_{1},$ the following conditions hold:

\begin{enumerate}
\item $p_{\alpha}\leq p.$

\item $p_{\beta}\leq$ $p_{\alpha}$ and $\gamma_{\alpha}<\gamma_{\beta}.$

\item $p_{\alpha}\Vdash$\textquotedblleft$\gamma_{\alpha}\in\dot{B}%
$\textquotedblright.
\end{enumerate}

\qquad\ \ \ \qquad\ \ \ 

Let $E=\left\{  \gamma_{\alpha}\mid\alpha\in\omega_{1}\right\}  .$ Since $C$
is not $[\cdot\cdot]$-Ramsey, there are $\alpha<\beta$ such that $\left[
\gamma_{\alpha}\gamma_{\beta}\right]  \notin C.$ It follows that $p_{\beta}$
forces that $\gamma_{\alpha},\gamma_{\beta}$ are in $\dot{B},$ but $\left[
\gamma_{\alpha}\gamma_{\beta}\right]  $ is not in $C,$ which is a contradiction.
\end{proof}

\qquad\ \qquad\ \ \ \ \ 

The main result of the section now easily follows:

\begin{theorem}
The statement \textquotedblleft No Aronszajn square-bracket operation has the
Ramsey club property\textquotedblright\ is consistent with $\mathsf{\Diamond
.}$
\end{theorem}

\begin{proof}
We start with a model of $\Diamond$ and \textsf{GCH. }Let $\mathbb{SC}%
_{\omega_{2}}$ be the countable support iteration of length $\omega_{2}$ of
the forcing $\mathbb{SC}.$ This iteration is $\sigma$-closed, preserves
$\Diamond$ (see Lemma V.5.8 and Exercise IV.7.46 of \cite{Kunen}) and is
$\omega_{2}$\textsf{-c.c.} (see Theorem 2.10 of \cite{AbrahamHandbook}). It
follows that every Aronszajn square-bracket operation will appear in an
intermediate extension. By Proposition \ref{Prop SC destruye a V}, any such
operation will fail to have the Ramsey club property in an intermediate model.
Finally, by Lemma \ref{Lemma sigma closed no corrige}, this failure is
preserved throughout the remainder of the iteration.
\end{proof}

\section{PFA and square-bracket operations induced by a sequence of reals
\label{Seccion PFA reals}}

So far we have studied the failure of the Ramsey club property. We now turn
our attention to proving that it can consistently hold. As mentioned in the
introduction, the second author (answering a question of Woodin) proved that
the \emph{Proper Forcing Axiom (\textsf{PFA}) } implies the Ramsey club
property holds for the square-bracket operations induced by special Aronszajn
trees (see Lemma 5.4.4 of \cite{Walks}). We will show that the same is true
for the other two types of square-bracket operations. This section focuses on
the Ramsey club property for the square-bracket operations defined from a
sequence of reals and a nice sequence. This proof is the simplest of the
three. For the remainder of this section, fix a club $C\subseteq$
\textsf{LIM}$\left(  \omega_{1}\right)  $, a sequence $\mathcal{R=}$ $\left\{
r_{\alpha}\mid\alpha\in\omega_{1}\right\}  \subseteq2^{\omega}$ and a nice
sequence $e=\left\langle e_{\beta}\mid\beta\in\omega_{1}\right\rangle .$ For
notational convenience, in this section we will write $\left[  \gamma
\xi\right]  $ for $\left[  \gamma\xi\right]  _{\mathcal{R},e}$ and
$\triangle\left(  \gamma,\xi\right)  $ for $\triangle_{\mathcal{R}}\left(
\gamma,\xi\right)  $ (for any $\gamma,\xi\in\omega_{1}$). Moreover, for
$\beta\in\omega_{1}$ and $n\in\omega,$ we define the set $F_{n}\left(
\beta\right)  =\left\{  \xi\leq\beta\mid e_{\beta}\left(  \xi\right)  \leq
n\right\}  $ (note that $\beta\in F_{n}\left(  \beta\right)  $). In this way,
it is clear that if $\alpha<\beta,$ then $\left[  \alpha\beta\right]  =$
\textsf{min}$\left(  F_{\triangle\left(  \alpha,\beta\right)  }\left(
\beta\right)  \setminus\alpha\right)  .$

\qquad\qquad\qquad\qquad\qquad\ \qquad\ \ \ \qquad\ \ \ 

We introduce more notation. We say that a set $E\subseteq\left[  \omega
_{1}\right]  ^{m}$ \emph{is unbounded }if for every $\xi<\omega_{1},$ there is
$b\in E$ with $\xi<$ \textsf{min}$\left(  b\right)  .$ For $\overline
{s}=\left\langle s_{i}\right\rangle _{i<m}\in\left(  2^{<\omega}\right)
^{m},$ define $\left\langle \overline{s}\right\rangle $ as the set of all
$\left(  f_{i}\right)  _{i<m}\in\left(  2^{\omega}\right)  ^{m}$ for which
$f_{i}\in\left\langle s_{i}\right\rangle $ for every $i<m.$ Given $Y\in\left[
\omega_{1}\right]  ^{<\omega},$ denote $\mathcal{R}_{Y}=\left\langle
r_{\alpha}\right\rangle _{\alpha\in Y}.$ The following lemma is probably
well-known, we provide a proof for the convenience of the reader.

\begin{lemma}
Let $m\in\omega,$ $E\subseteq\left[  \omega_{1}\right]  ^{m}$ and $M$ a
countable elementary submodel such that $E,\mathcal{R}\in M.$ Assume there is
$X=\left\{  \alpha_{i}\mid i<m\right\}  \in E$ with $X\cap M=\emptyset.$ For
every $n\in\omega,$ there exists $\overline{s}=\left\langle s_{i}\right\rangle
_{i<m}\in\left(  2^{<\omega}\right)  ^{m}$ such that for all $i,j<m,$ the
following conditions hold: \label{Lema PFA reales}

\begin{enumerate}
\item $\left\vert s_{i}\right\vert =\left\vert s_{j}\right\vert .$

\item $s_{i}\upharpoonright n=r_{\alpha_{i}}\upharpoonright n.$

\item $r_{\alpha}\notin\left\langle s_{i}\right\rangle $ for every $\alpha\in
X$.

\item The set $\{Y\in E\mid\mathcal{R}_{Y}\in\left\langle \overline
{s}\right\rangle \}$ is unbounded.
\end{enumerate}
\end{lemma}

\begin{proof}
For every $i<m,$ denote $t_{i}=r_{\alpha_{i}}\upharpoonright n.$ Let $D$ be
the set of all $Y=\left\{  \beta_{i}\mid i<m\right\}  \in E$ such that
$r_{\beta_{i}}\in\left\langle t_{i}\right\rangle $ for every $i<m.$ Note that
$D\in M$ and $X\in D.$ Since $X\cap M=\emptyset,$ it follows by elementarity.
that $D$ is unbounded. In this way, we can find $Y=\left\{  \beta_{i}\mid
i<m\right\}  \in D$ such that $Y\cap M$ and $Y\cap X$ are both empty. Find
$l\in\omega$ such that $\triangle(r_{\alpha_{i}},r_{\beta_{i}}),$
$\triangle\left(  r_{\beta_{i}},r_{\beta_{j}}\right)  ,$ $\triangle\left(
r_{\alpha_{i}},r_{\alpha_{j}}\right)  <l$ for all $i,j<m$ and define
$s_{i}=r_{\beta_{i}}\mid l.$ It is easy to see that $\overline{s}=\left\{
s_{i}\mid i<m\right\}  $ is as desired.
\end{proof}

\qquad\ \ \ \ \ \qquad\ \qquad\ \ 

The following forcing notion was inspired by the one from Lemma 5.4.4 of
\cite{Walks}.

\begin{definition}
Define $\mathbb{P}\left(  C\right)  $ as the set of all $p=(\mathcal{M}%
,\mathcal{N})$ with the following properties:

\begin{enumerate}
\item $\mathcal{N}$ is a finite $\in$-chain of countable elementary submodels
of \textsf{H}$\left(  \omega_{2}\right)  $ such that $\mathcal{R},e$ and $C$
belong to all of them.

\item $\mathcal{M\subseteq N}.$

\item \textbf{Completeness condition. }If $M,L\in\mathcal{M}$ are distinct,
then $\left[  \delta_{M}\delta_{L}\right]  \in C.$

\item \textbf{Crossing models condition. }For every $M,L\in\mathcal{M}$
distinct and $N\in\mathcal{N},$ if $N\in M,$ then \textsf{min}$(F_{\triangle
\left(  \delta_{M},\delta_{L}\right)  }\left(  \delta_{M}\right)  \setminus
N)\in C.$
\end{enumerate}
\end{definition}

\qquad\ \qquad\ \qquad\ \ 

We highlight several key points regarding this notion. Let $p=(\mathcal{M}%
,\mathcal{N})$ be an element of $\mathbb{P(}C\mathbb{)}.$

\begin{enumerate}
\item If $N\in\mathcal{N},$ then $\delta_{N}\in C$ (see Proposition
\ref{Prop submodelos club y est}).

\item In the crossing models condition, the case $N=L$ is a consequence of the
completeness condition. Therefore, when verifying this condition for a
potential element of $\mathbb{P(}C\mathbb{)},$ we may always assume $N\neq L.$

\item In the crossing models condition, no relationship is assumed between
$M,N$ and $L$ other than $N\in M$ and $M\neq L.$ In particular, it is possible
that $N\in L\in M,$ $L\in N\in M$ or $N\in M\in L.$
\end{enumerate}

\qquad\ \qquad\ \ \ \qquad\ \ \ \ \ 

Let $p=(\mathcal{M}_{p},\mathcal{N}_{p})$ and $q=(\mathcal{M}_{q}%
,\mathcal{N}_{q})$ in $\mathbb{P}\left(  C\right)  .$ Define $p\leq q$ if
$\mathcal{M}_{q}\subseteq\mathcal{M}_{p},$ $\mathcal{N}_{q}\subseteq
\mathcal{N}_{p}$ and $\mathcal{M}_{q}=\mathcal{M}_{p}\cap\mathcal{N}_{q}.$ For
$p,q\in\mathbb{P}\left(  C\right)  ,$ define $q\sqsubseteq p$ if there is
$N\in\mathcal{N}_{p}$ such that $\mathcal{M}_{q}=\mathcal{M}_{p}\cap N$ and
$\mathcal{N}_{q}=\mathcal{N}_{p}\cap N.$ Clearly if $q\sqsubseteq p,$ then
$p\leq q.$

\qquad\ \qquad\ 

We now introduce several \textquotedblleft invariants\textquotedblright\ for
conditions in $\mathbb{P(}C\mathbb{)}$ which will be useful for analyzing the
combinatorial properties of the forcing.

\begin{definition}
Let $p=(\mathcal{M},\mathcal{N})\in\mathbb{P}\left(  C\right)  .$ Define the
following items:

\begin{enumerate}
\item \textsf{ht}$\left(  p\right)  =\left\{  \delta_{N}\mid N\in
\mathcal{N}\right\}  $ and \textsf{ht}$_{\mathcal{M}}\left(  p\right)
=\left\{  \delta_{M}\mid M\in\mathcal{M}\right\}  .$

\item $d_{p}$ is the smallest natural number that is larger than all of the following:

\begin{enumerate}
\item $\triangle(\delta_{M},\delta_{N})$ for $M,N\in\mathcal{N}$ distinct.

\item $e_{\delta_{M}}(\delta_{N})$ for $M,N\in\mathcal{N}$ with $N\in M.$
\end{enumerate}

\item $T_{p}=\left\langle r_{\delta_{N}}\upharpoonright d_{p}\right\rangle
_{N\in\mathcal{N}_{p}}.$

\item $S\left(  p\right)  =\bigcup\{F_{d_{p}}\left(  \delta_{N}\right)  \mid
N\in\mathcal{N}\}.$

\item $H\left(  p\right)  =(S\left(  p\right)  ,<,$\textsf{ht}$\left(
p\right)  ,$\textsf{ht}$_{\mathcal{M}}\left(  p\right)  ,d_{p},e,T_{p}%
,\triangle).$
\end{enumerate}
\end{definition}

\qquad\ \qquad\ \ \ 

It follows that $S\left(  p\right)  $ is a finite subset of $\omega_{1}$ that
contains \textsf{ht}$\left(  p\right)  .$ Let $p,q\in\mathbb{P}\left(
C\right)  $ and $f:S\left(  p\right)  \longrightarrow S\left(  q\right)  $ a
bijection. We say that $f$ \emph{is an isomorphism from }$H\left(  p\right)
$\emph{ to} $H\left(  q\right)  $ if for every $\alpha,\beta\in S\left(
p\right)  ,$ the following holds:

\begin{enumerate}
\item $\alpha<\beta$ if and only if $f\left(  \alpha\right)  <f\left(
\beta\right)  .$

\item $\alpha\in$ \textsf{ht}$\left(  p\right)  $ if and only if $f\left(
\alpha\right)  \in$ \textsf{ht}$\left(  q\right)  .$

\item $\alpha\in$ \textsf{ht}$_{\mathcal{M}}\left(  p\right)  $ if and only if
$f\left(  \alpha\right)  \in$ \textsf{ht}$_{\mathcal{M}}\left(  q\right)  .$

\item $d_{p}=d_{q}.$

\item $e_{\beta}\left(  \alpha\right)  =e_{f\left(  \beta\right)  }\left(
f\left(  \alpha\right)  \right)  .$

\item If $\alpha\in$ \textsf{ht}$\left(  p\right)  ,$ then $r_{\alpha
}\upharpoonright d_{p}=r_{f\left(  \alpha\right)  }\upharpoonright d_{q}.$

\item $\triangle\left(  \alpha,\beta\right)  =\triangle\left(  f\left(
\alpha\right)  ,f\left(  \beta\right)  \right)  .$
\end{enumerate}

\qquad\qquad\qquad\qquad\qquad\ \ \ \ 

We say that $H\left(  p\right)  $ and $H\left(  q\right)  $ are isomorphic if
there is an isomorphism between them. Since such an isomorphism is an
isomorphism between the finite linear orders $\left(  S\left(  p\right)
,<\right)  $ and $\left(  S\left(  q\right)  ,<\right)  $, it is necessarily
unique. It is easy to see that there are only countably many isomorphism types.

\begin{definition}
We say that $M$ is a \emph{big model} if the following conditions hold:

\begin{enumerate}
\item There is $\kappa>\omega_{2}$ a large enough regular cardinal such that
$M$ $\ $is a countable elementary submodel of \textsf{H}$\left(
\kappa\right)  .$

\item $C,\mathcal{R},e,\mathbb{P}\left(  C\right)  $ $\mathbb{\in}$ $M.$
\end{enumerate}
\end{definition}

\qquad\ \qquad\ \ \qquad\ \ \ 

For $M$ a big model, denote $\overline{M}=M\cap$ \textsf{H}$\left(  \omega
_{2}\right)  ,$ which is an elementary submodel of \textsf{H}$\left(
\omega_{2}\right)  .$ We now prove the following key proposition, which will
be used to establish the main result of this section.

\begin{proposition}
Let $M$ be a big model and $p=(\mathcal{M}_{p},\mathcal{N}_{p})\in
\mathbb{P}\left(  C\right)  .$ If $\overline{M}\in\mathcal{N}_{p},$ then $p$
is $(M,\mathbb{P}\left(  C\right)  )$-generic. \label{Prop Can Proper reales}
\end{proposition}

\begin{proof}
Let $D\in M$ be an open dense set of $\mathbb{P}\left(  C\right)  $. By
extending $p$ if necessary, we may assume that $p\in D.$ We need to prove that
$p$ is compatible with an element of $D\cap M.$ Denote $\delta=\delta_{M},$
which is an element of $C.$ Define $X=$ \textsf{ht}$\left(  p\right)
\setminus M$ and enumerate $X=\left\{  \delta_{i}\mid i<m\right\}  $ in
increasing order (hence $\delta_{0}=\delta$).

\qquad\ \qquad\ \qquad\ \qquad\ \ 

Let $p_{M}=(\mathcal{M}_{p}\cap M,\mathcal{N}_{p}\cap M).$ It is easy to see
that $p_{M}\in\mathbb{P}\left(  C\right)  \cap M$ and $p_{M}\sqsubseteq p.$
Find $\eta_{0}<\delta$ such that $S\left(  p\right)  \cap M\subseteq\eta_{0}.$
We now define $E$ as the set of all $Y\in\left[  \omega_{1}\right]  ^{m}$ for
which there is $r\in D$ with the following properties:

\begin{enumerate}
\item $p_{M}\sqsubseteq r.$

\item \textsf{ht}$\left(  r\right)  =$ \textsf{ht}$\left(  p_{M}\right)  \cup
Y.$

\item $S\left(  r\right)  \cap\eta_{0}=S\left(  p\right)  \cap\eta_{0}$ (which
is the same as $S\left(  p\right)  \cap M$).

\item $H\left(  p\right)  $ and $H\left(  r\right)  $ are isomorphic and the
isomorphism is the identity on $S\left(  p\right)  \cap M.$
\end{enumerate}

\qquad\ \ \ 

We claim that $E\in M.$ Although the third and fourth clause involve $M$ and
$p,$ they only depend on the isomorphism type of $p$ (which is in $M$) and the
finite set $S\left(  p\right)  \cap M\in M.$ Consequently, both clauses can be
reformulated without reference to objects outside $M.$ For $Y\in E$ and $r$ as
above, we will say that $r$ is a witness for $Y\in E.$

\qquad\qquad\qquad\qquad\qquad\qquad

Apply Lemma \ref{Lema PFA reales} and find $\overline{s}=\left\langle
s_{i}\right\rangle _{i<m}$ with the following properties:

\begin{enumerate}
\item $\left\vert s_{i}\right\vert =\left\vert s_{j}\right\vert $ for all
$i,j<m.$

\item $r_{\delta_{i}}\upharpoonright d_{p}=s_{i}\upharpoonright d_{p}$ for
$i<m.$

\item $r_{\alpha}\notin\left\langle s_{i}\right\rangle $ for all $\alpha\in X$
and $i<m.$

\item The set $\{Y\in E\mid\mathcal{R}_{Y}\in\left\langle \overline
{s}\right\rangle \}$ is unbounded.
\end{enumerate}

\qquad\ \ \ 

Denote $l=\left\vert s_{0}\right\vert $ (which is the same as $\left\vert
s_{i}\right\vert $ for every $i<m$). Find $\eta\in\left(  \eta_{0}%
,\delta\right)  $ such that $F_{l}\left(  \beta\right)  \cap M\subseteq\eta$
for every $\beta\in X.$ We can now find $Y\in E\cap M$ with $\eta<$
\textsf{min}$\left(  Y\right)  $ and $\mathcal{R}_{Y}\in\left\langle
\overline{s}\right\rangle .$ By elementarity., there exists $r\in D\cap M$
that that is a witness of $Y\in E.$ Let $h:S\left(  p\right)  \longrightarrow
S\left(  r\right)  $ be the unique isomorphism from $H\left(  p\right)  $ to
$H\left(  r\right)  .$ Recall that $h$ is the identity on $M.$

\qquad\qquad\qquad\ \ \ 

We claim that $p$ and $r$ are compatible. It suffices to show that
$q=(\mathcal{M}_{p}\cup\mathcal{M}_{r},\mathcal{N}_{p}\cup\mathcal{N}_{r})$
belongs to $\mathbb{P(}C\mathbb{)}.$ Clearly $\mathcal{N}_{p}\cup
\mathcal{N}_{r}$ is a finite $\in$-chain and contains $\mathcal{M}_{p}%
\cup\mathcal{M}_{r}.$ It remains to verify that $q$ satisfies both the
completeness condition and the crossing models condition.

\qquad\ \qquad\ \ \ 

To verify that $q$ satisfies the completeness condition, it suffices to show
that for any $\alpha,\beta\in$ \textsf{ht}$_{\mathcal{M}}\left(  p\right)
\setminus M,$ we have that $\left[  h\left(  \alpha\right)  \beta\right]  \in
C.$ Denote $\triangle=\triangle\left(  h\left(  \alpha\right)  ,\beta\right)
.$ Recall that $h\left(  \alpha\right)  \in M$ while $\beta\notin M$, so
$h\left(  \alpha\right)  <\beta.$ We proceed by cases, depending on whether
$\alpha$ and $\beta$ are equal.$\medskip$

$%
\begin{tabular}
[c]{|l|}\hline
\textbf{Case:} $\alpha\neq\beta$\\\hline
\end{tabular}
\medskip$

It follows that $\triangle=\triangle\left(  h\left(  \alpha\right)
,\beta\right)  =\triangle\left(  \alpha,\beta\right)  $ and is smaller than
$d_{p},$ so $F_{\triangle}\left(  \beta\right)  \cap M\subseteq\eta
_{0}\subseteq h\left(  \alpha\right)  .$ We have the following:\newline

\hfill%
\begin{tabular}
[c]{lll}%
$\left[  h\left(  \alpha\right)  \beta\right]  $ & $=$ & \textsf{min}$\left(
F_{\triangle}\left(  \beta\right)  \setminus h\left(  \alpha\right)  \right)
$\\
& $=$ & \textsf{min}$\left(  F_{\triangle}\left(  \beta\right)  \setminus
M\right)  $\\
& $=$ & \textsf{min}$\left(  F_{\triangle\left(  \alpha,\beta\right)  }\left(
\beta\right)  \setminus M\right)  $%
\end{tabular}
\hfill\ 

\qquad\ \qquad\ \ 

If $\delta<\beta,$ then \textsf{min}$\left(  F_{\triangle\left(  \alpha
,\beta\right)  }\left(  \beta\right)  \setminus M\right)  \in C$ by the
crossing models condition of $p.$ On the other hand, if $\delta=\beta$, then
\textsf{min}$\left(  F_{\triangle\left(  \alpha,\beta\right)  }\left(
\beta\right)  \setminus M\right)  =\beta,$ which belongs to $C.\medskip$

$%
\begin{tabular}
[c]{|l|}\hline
\textbf{Case:} $\alpha=\beta$\\\hline
\end{tabular}
\ \medskip$

Here we have that $e_{\beta}\left(  \delta\right)  <d_{p}\leq\triangle<l$. It
follows that $F_{\triangle}\left(  \beta\right)  \cap M\subseteq F_{l}\left(
\beta\right)  \cap M\subseteq\eta\subseteq h\left(  \beta\right)  .$ We have
the following:\newline

\hfill%
\begin{tabular}
[c]{lll}%
$\left[  h\left(  \alpha\right)  \beta\right]  $ & $=$ & \textsf{min}$\left(
F_{\triangle}\left(  \beta\right)  \setminus h\left(  \beta\right)  \right)
$\\
& $=$ & \textsf{min}$\left(  F_{\triangle}\left(  \beta\right)  \setminus
M\right)  $\\
& $=$ & $\delta\in C.$%
\end{tabular}
\hfill\ 

\qquad\ \qquad\ \qquad\ \qquad\ \ 

This completes the proof that $q$ satisfies the completeness condition. We now
verify the crossing models condition. It suffices to show that for any
distinct $\alpha,\beta\in$ \textsf{ht}$_{\mathcal{M}}\left(  p\right)  \cup$
\textsf{ht}$_{\mathcal{M}}\left(  r\right)  $ and any $N\in\mathcal{N}_{p}%
\cup\mathcal{N}_{r}$ with $\delta_{N}<\beta$ and $\alpha\neq\delta_{N},$ we
have that \textsf{min}$\left(  F_{\triangle}\left(  \beta\right)  \setminus
N\right)  \in C,$ where $\triangle=\triangle\left(  \alpha,\beta\right)  $.
The proof proceeds by cases, depending on the membership of $N.\medskip$

$%
\begin{tabular}
[c]{|l|}\hline
\textbf{Case:} $N\in\mathcal{N}_{p}\cap\mathcal{N}_{r}$\\\hline
\end{tabular}
\ \medskip$

Note that if both $\alpha,\beta\in$ \textsf{ht}$\left(  p\right)  $ or both
$\alpha,\beta\in$ \textsf{ht}$\left(  r\right)  ,$ then the result follows
immediately since $p$ and $r$ are conditions in $\mathbb{P}\left(  C\right)
.$ The proof now proceeds by subcases, depending on the relationship between
$\triangle$ and $d_{p}.\medskip$

\textbf{Subcase:} $N\in\mathcal{N}_{p}\cap\mathcal{N}_{r}$ and $\triangle
<d_{p}.$

\qquad\ \qquad\ \ \qquad\ \ \ 

First assume that $\alpha\in$ \textsf{ht}$\left(  p\right)  \setminus$
\textsf{ht}$\left(  r\right)  $ (so $\alpha\notin M$) and $\beta\in$
\textsf{ht}$\left(  r\right)  \setminus$\textsf{ht}$\left(  p\right)  .$ It
follows from $\triangle<d_{p}$ that $h\left(  \alpha\right)  \neq\beta$ and
$\triangle=\triangle\left(  h\left(  \alpha\right)  ,\beta\right)  .$ Since
$N\in\mathcal{N}_{r},$ the crossing models property of $r$ implies that
\textsf{min}$\left(  F_{\triangle}\left(  \beta\right)  \setminus N\right)
\in C.$ We now assume that $\beta\in$ \textsf{ht}$\left(  p\right)  \setminus$
\textsf{ht}$\left(  r\right)  $ (so $\beta\notin M$) and $\alpha\in$
\textsf{ht}$\left(  r\right)  \setminus$\textsf{ht}$\left(  p\right)  .$ Once
again, since $\triangle<d_{p},$ we have that $h\left(  \beta\right)
\neq\alpha$ and $\triangle=\triangle\left(  h\left(  \beta\right)
,\alpha\right)  .$ Since $N\in\mathcal{N}_{p},$ the crossing models property
of $p$ implies that \textsf{min}$\left(  F_{\triangle}\left(  \beta\right)
\setminus N\right)  \in C.\medskip$

\textbf{Subcase:} $N\in\mathcal{N}_{p}\cap\mathcal{N}_{r}$ and $d_{p}%
\leq\triangle.$

\qquad\qquad\ \ \ 

We have that $r_{\alpha}\upharpoonright d_{p}=r_{\beta}\upharpoonright d_{p}.$
Note that $e_{\beta}\left(  \delta_{N}\right)  <d_{p}\leq\triangle.$ It
follows that \textsf{min}$\left(  F_{\triangle}\left(  \beta\right)  \setminus
N\right)  =\delta_{N}\in C.\medskip$

$%
\begin{tabular}
[c]{|l|}\hline
\textbf{Case:} $N\in\mathcal{N}_{r}\setminus\mathcal{N}_{p}$\\\hline
\end{tabular}
\ \medskip$

Note that if $\alpha,\beta\in$ \textsf{ht}$\left(  r\right)  ,$ then the
result follows immediately since $r$ is a condition in $\mathbb{P}\left(
C\right)  .$ The proof now proceeds by subcases, depending on whether $\beta$
is in \textsf{ht}$\left(  r\right)  $.\medskip

\textbf{Subcase: }$N\in\mathcal{N}_{r}\setminus\mathcal{N}_{p}$ and
$\beta\notin$ \textsf{ht}$\left(  r\right)  .$

\qquad\ \qquad\ \qquad\ \qquad\ \ 

Since $\beta\notin$ \textsf{ht}$\left(  r\right)  ,$ it follows that
$\beta\notin M.$ Moreover, $F_{\triangle}\left(  \beta\right)  \cap
M\subseteq\eta<\delta_{N}.$ In this way, \textsf{min}$\left(  F_{\triangle
}\left(  \beta\right)  \setminus N\right)  =$ \textsf{min}$\left(
F_{\triangle}\left(  \beta\right)  \setminus M\right)  .$ If $\beta=\delta,$
then \textsf{min}$\left(  F_{\triangle}\left(  \beta\right)  \setminus
M\right)  =\delta,$ which is in $C.$ Assume that $\beta\neq\delta$ (so
$\delta<\beta$). If $\triangle\geq d_{p},$ since $e_{\beta}\left(
\delta\right)  <d_{p},$ it follows that \textsf{min}$\left(  F_{\triangle
}\left(  \beta\right)  \setminus M\right)  =\delta.$ We now assume that
$\triangle<d_{p}.$ Let $\gamma\in$ \textsf{ht}$_{\mathcal{M}}\left(  p\right)
$ such that $\triangle\left(  \gamma,\beta\right)  =\triangle.$ We conclude
that \textsf{min}$\left(  F_{\triangle}\left(  \beta\right)  \setminus
M\right)  \in C$ by the crossing models property of $p.\medskip$

\textbf{Subcase: }$N\in\mathcal{N}_{r}\setminus\mathcal{N}_{p}$ and $\beta\in$
\textsf{ht}$\left(  r\right)  .$

\qquad\qquad\qquad\ \ \ \ 

Note that $\beta\in M.$ If $\alpha$ is also in $M,$ then $\alpha,\beta\in$
\textsf{ht}$\left(  r\right)  $ and the result follows since $r$ is a
condition$.$ Assume that $\alpha\notin M.$ If $h\left(  \alpha\right)
\neq\beta,$ then $\triangle=\triangle\left(  h\left(  \alpha\right)
,\beta\right)  $ and is smaller than $d_{p}.$ We have that \textsf{min(}%
$F_{\triangle}\left(  \beta\right)  \setminus N)=$ \textsf{min}$(F_{\triangle
\left(  h(\alpha),\beta\right)  }\left(  \beta\right)  \setminus N)$ and it is
in $C$ because of the crossing models property of $r.$ We now assume that
$h\left(  \alpha\right)  =\beta.$ It follows that $e_{\beta}\left(  \delta
_{N}\right)  <d_{p}\leq\triangle.$ In this way, \textsf{min(}$F_{\triangle
}\left(  \beta\right)  \setminus N)=\delta_{N},$ which is in $C.\medskip$

$%
\begin{tabular}
[c]{|l|}\hline
\textbf{Case:} $N\in\mathcal{N}_{p}\setminus\mathcal{N}_{r}$\\\hline
\end{tabular}
\ \medskip$

It must be the case that $N\notin M.$ Since $\delta_{N}<\beta,$ then
$\beta\notin M,$ so $\beta\in$ \textsf{ht}$\left(  p\right)  .$ If $\alpha$ is
also in \textsf{ht}$\left(  p\right)  ,$ the result follows immediately since
$p\in\mathbb{P}\left(  C\right)  .$ Assume that $\alpha\in$ \textsf{ht}%
$\left(  r\right)  \setminus$ \textsf{ht}$\left(  p\right)  .$ If $h\left(
\beta\right)  \neq\alpha,$ then $\triangle<d_{p}$ and $\triangle
=\triangle\left(  \beta,h^{-1}\left(  \alpha\right)  \right)  .$ It follows
that \textsf{min(}$F_{\triangle}\left(  \beta\right)  \setminus N)=$
\textsf{min}$(F_{\triangle\left(  \beta,h^{-1}\left(  \alpha\right)  \right)
}\left(  \beta\right)  \setminus N)$ and it is in $C$ because of the crossing
models property of $p.$ Assume that $h\left(  \beta\right)  =\alpha,$ hence
$d_{p}\leq\triangle,$ so $e_{\beta}\left(  \delta_{N}\right)  <d_{p}%
\leq\triangle.$ We conclude that \textsf{min(}$F_{\triangle}\left(
\beta\right)  \setminus N)=\delta,$ which is in $C.$ We finally finished the proof.
\end{proof}

\qquad\ \qquad\ \qquad\ \ \ 

We conclude the following:

\begin{corollary}
$\mathbb{P}\left(  C\right)  $ is a proper forcing.
\end{corollary}

\begin{proof}
Let $M$ be a big model and $q\in M\cap\mathbb{P}\left(  C\right)  .$ Define
$p=(\mathcal{M}_{q},\mathcal{N}_{q}\cup\left\{  \overline{M}\right\}  ).$ It
follows that $p\leq q$ and it is $(M,\mathbb{P}\left(  C\right)  )$-generic by
Proposition \ref{Prop Can Proper reales}.
\end{proof}

\qquad\ \qquad\ \qquad\ \ \ \ 

The main result of this section now easily follows. This is the same argument
as the one from Lemma 5.4.4 of \cite{Walks}.

\begin{theorem}
\textsf{PFA }implies that every square-bracket operation of the form
$[\cdot\cdot]_{\mathcal{R},e}$ has the Ramsey club property.
\label{PFA para reales}
\end{theorem}

\begin{proof}
Assume \textsf{PFA }and let $C,$ $\mathcal{R}$ and $e$ as before, with
$\mathbb{P}\left(  C\right)  $ the corresponding forcing notion. Given
$G\subseteq\mathbb{P}\left(  C\right)  $ a filter, define the set $W\left(
G\right)  =\bigcup\{\delta_{M}\mid\exists p\in G\left(  M\in\mathcal{M}%
_{p}\right)  \}.$ It follows from the definition of the forcing that $\left\{
\left[  \alpha\beta\right]  \mid\alpha,\beta\in W\left(  G\right)  \right\}
\subseteq C.$ To complete the proof, it suffices to show that there is
$p\in\mathbb{P}\left(  C\right)  $ that forces that $W(\dot{G})$ is
uncountable (where $\dot{G}$ is the canonical name of the generic filter).

\qquad\ \ \ \ \qquad\ \ 

Let $M$ be a big model. Define $p=(\{\overline{M}\},\{\overline{M}\}).$ Since
$p$ is $(M,\mathbb{P}\left(  C\right)  )$-generic and $p\Vdash$%
\textquotedblleft$\delta_{M}\in W(\dot{G})$\textquotedblright, it follows from
Lemma \ref{Lema modelo contension} that $p$ forces that $W(\dot{G})$ is uncountable.
\end{proof}

\section{PFA and square-bracket operations induced by C-sequences
\label{Seccion PFA walks}}

It remains to prove that \textsf{PFA }implies the Ramsey club property for the
square-bracket operations arising from \textsf{C}-sequences. The proof for
this case appears to be the hardest of the three. For the remainder of this
section, fix a club $C\subseteq$ \textsf{LIM}$\left(  \omega_{1}\right)  $ and
$\mathcal{C}=\left\{  C_{\alpha}\mid\alpha\in\omega_{1}\right\}  $ a
\textsf{C}-sequence. For notational convenience, in this section we will write
$\left[  \gamma\xi\right]  $ for $\left[  \gamma\xi\right]  _{\mathcal{C}}$
and $\triangle\left(  \gamma,\xi\right)  $ for $\triangle_{\mathcal{C}}\left(
\gamma,\xi\right)  $ (for any $\gamma,\xi\in\omega_{1}$). We need to prove
some preliminary results.

\begin{definition}
Let $X\subseteq\omega_{1}.$ We define the following:
\label{Def estructura generada}

\begin{enumerate}
\item \textsf{Tr}$\left[  X\right]  =\bigcup\{$\textsf{Tr}$(\alpha,\beta
)\mid\alpha,\beta\in X\}.$

\item $\triangle\left(  X\right)  =\{\triangle\left(  \alpha,\beta\right)
\mid\alpha,\beta\in X\}.$

\item $\lambda\left(  X\right)  =\left\{  \lambda\left(  \alpha,\beta\right)
\mid\alpha,\beta\in X\right\}  .$

\item $S\left(  X\right)  $ is the smallest set that contains $X$ and is
closed under the previous operations.
\end{enumerate}
\end{definition}

\qquad\ \qquad\ \qquad\ 

We now have the following:

\begin{proposition}
If $X\subseteq\omega_{1}$ is finite, then $S\left(  X\right)  $ is also finite.
\end{proposition}

\begin{proof}
Let $X\subseteq\omega_{1}$ be a finite set. Define recursively a sequence of
sets by $X_{0}=X$ and $X_{n+1}=X_{n}\cup$ \textsf{Tr}$\left[  X_{n}\right]
\cup\triangle\left(  X_{n}\right)  \cup\lambda\left(  X_{n}\right)  .$ It is
straightforward to verify that $S\left(  X\right)  =\bigcup\limits_{n\in
\omega}X_{n}$ and that $X_{n}\subseteq X_{m}$ whenever $n\leq m.$ It follows
that $\left\vert S\left(  X\right)  \right\vert \leq\omega$ (this can also be
deduced from the observation that $S\left(  X\right)  \subseteq$
\textsf{max}$\left(  X\right)  +1$). Assume that $S\left(  X\right)  $ is
infinite. Let $S^{\prime}\left(  X\right)  $ denote the set of accumulation
points of $S\left(  X\right)  .$ Since $S\left(  X\right)  $ is countably
infinite, we conclude that $S^{\prime}\left(  X\right)  $ has a maximum, which
we will $\gamma.$ Note that $\gamma$ is a limit ordinal.

\qquad\qquad\qquad\qquad

Denote $Y=S\left(  X\right)  \setminus\gamma.$ Since $\gamma$ is the largest
accumulation point of $S\left(  X\right)  ,$ the set $Y$ must be finite.
Consequently, there exists $n\in\omega$ such that $Y\subseteq X_{n}.$ Now,
choose $\gamma_{m}\in C_{\gamma}$ such that:

\begin{enumerate}
\item $X_{n+1}\cap\gamma\subseteq\gamma_{m}.$

\item $\lambda\left(  \gamma,\beta\right)  <\gamma_{m}$ for all $\beta\in Y$
with $\gamma<\beta.$
\end{enumerate}

\qquad\ \ \qquad\ \ 

We claim that $S\left(  X\right)  \cap\left(  \gamma_{m},\gamma\right)  $ is
empty. This will be a contradiction since $\gamma$ is an accumulation point of
$S\left(  X\right)  .$ We will prove by induction that $X_{l}\cap\left(
\gamma_{m},\gamma\right)  =\emptyset.$ This is trivial in case $l\leq n+1.$
Assume now that $n+1\leq l$ and $X_{l}\cap\left(  \gamma_{m},\gamma\right)
=\emptyset,$ we will show that the same is true for $l+1.$ Since $\lambda$ and
$\triangle$ are regressive functions and $X_{l}\setminus\gamma=Y,$ it follows
that $\lambda\left[  X_{l}\right]  $ and $\triangle\left(  X_{l}\right)  $ are
disjoint with $\left(  \gamma_{m},\gamma\right)  .$ It remains to prove that
if $\beta\in Y$ and $\alpha\in X_{l}\cap\gamma,$ then \textsf{Tr}%
$(\alpha,\beta)\cap\left(  \gamma_{m},\gamma\right)  =\emptyset.$ We proceed
by cases depending wether $\gamma$ is in \textsf{Tr}$(\alpha,\beta)$ or not.

\qquad\qquad\ \ 

First assume that $\gamma\in$ \textsf{Tr}$(\alpha,\beta).$ Since $\alpha
\leq\gamma_{m},$ the walk from $\beta$ to $\alpha,$ must pass through $\gamma$
and then jump to \textsf{min}$\left(  C_{\gamma}\setminus\alpha\right)
\leq\gamma_{m}.$ It follows that \textsf{Tr}$(\alpha,\beta)\cap\left(
\gamma_{m},\gamma\right)  =\emptyset$. Now assume that $\gamma\notin$
\textsf{Tr}$(\alpha,\beta).$ Let \textsf{Tr}$\left(  \alpha,\beta\right)
=\left\{  \beta_{0},...,\beta_{n}\right\}  $ where $\beta_{0}=\beta,$
$\beta_{n}=\alpha$ and $\beta_{i+1}=$ \textsf{stp}$\left(  \alpha,\beta
_{i}\right)  .$ Let $i$ such that $\beta_{i}>\gamma$ but $\gamma<\beta_{i+1}.$
Proposition \ref{Prop juntar trazas} implies that $\beta_{i+1}\leq
\lambda\left(  \gamma,\beta\right)  .$ By our choice of $\gamma_{m},$ we have
that $\lambda\left(  \gamma,\beta\right)  <\gamma_{m}$, hence $\beta
_{i+1}<\gamma_{m}.$ We conclude again that \textsf{Tr}$(\alpha,\beta
)\cap\left(  \gamma_{m},\gamma\right)  =\emptyset$.
\end{proof}

\qquad\ \qquad\ \qquad\ \ \ \ \ \qquad\ \ 

Let $T$ be an Aronszajn tree, $A\subseteq T^{n+1},$ $\gamma\in\omega_{1}$ and
$F=\left(  t_{0},...,t_{n}\right)  \in\left(  T_{\gamma}\right)  ^{n+1}$ (we
do not require that $t_{i}\neq t_{j}$ for $i\neq j$). The game \emph{G}%
$_{F}\left(  T,A\right)  $ is defined as follows:

\begin{center}%
\begin{tabular}
[c]{|l|l|l|l|l|l|l|l|}\hline
$\mathsf{I}$ & $\xi_{0}$ &  & $\xi_{1}$ &  & $...$ & $\xi_{n}$ & \\\hline
$\mathsf{II}$ &  & $s_{0}$ &  & $s_{1}$ &  &  & $s_{n}$\\\hline
\end{tabular}

\end{center}

\qquad\ \qquad\ \ \qquad\ \ 

The game lasts $n+1$ rounds. In round $i,$ Player $\mathsf{I}$ plays $\xi
_{i}\in\omega_{1}$ and Player $\mathsf{II}$ responds with $s_{i}\in T,$
subject to the following rules:

\begin{enumerate}
\item $\xi_{0}<\xi_{1}<...<\xi_{n}.$

\item $\xi_{i}<\left\vert s_{i}\right\vert $ for every $i\leq n.$

\item $t_{i}=s_{i}\upharpoonright\gamma$ for every $i\leq n.$
\end{enumerate}

\qquad\ \ \ \qquad\ \ 

\emph{Player} $\mathsf{II}$ \emph{wins the game }if the resulting
sequence\emph{ }$\left(  s_{0},...,s_{n}\right)  $ is in $A.$ Player
$\mathsf{I}$ is a man and Player $\mathsf{II}$ is a woman. Since this is a
clopen game, it is determined by the\ Gale-Stewart Theorem (see Theorem 20.1
of \cite{Kechris}).

\begin{proposition}
Let $T$ be an Aronszajn tree, $A\subseteq T^{n+1}$ and $\left\{
M_{0},...,M_{n}\right\}  $ an $\in$-chain of countable elementary submodels
such that $T,A\in M_{0}.$ Assume there is $\left(  s_{0},..,s_{n}\right)  \in
A$ such that $\delta_{M_{0}}\leq\left\vert s_{0}\right\vert <\delta_{M_{1}%
}\leq\left\vert s_{1}\right\vert \leq...\leq\delta_{M_{n}}\leq\left\vert
s_{n}\right\vert .$ For every $\gamma\in\omega_{1},$ there is $F\in\left(
T_{\gamma}\right)  ^{n+1}$ such that Player $\mathsf{II}$ has a winning
strategy in \emph{G}$_{F}\left(  T,A\right)  .$
\end{proposition}

\begin{proof}
By elementarity., it suffices to prove the conclusion of the proposition for
$\gamma<\delta_{M_{0}}.$ Let $F=\left(  s_{0}\upharpoonright\gamma
,...,s_{n}\upharpoonright\gamma\right)  .$ We claim that Player $\mathsf{II}$
has a winning strategy in the game \emph{G}$_{F}\left(  T,A\right)  .$ Assume
this is false, so Player $\mathsf{I}$ has a winning strategy in the game. By
elementarity., there is $\sigma\in M_{0}$ a winning strategy for this game.
Consider the run of the game in where Player $\mathsf{I}$ follows $\sigma$ and
Player $\mathsf{II}$ plays $s_{i}$ at round $i.$ It is easy to see that this
is a valid run of the game in which Player $\mathsf{II}$ wins, contradicting
the assumption that $\sigma$ is a winning strategy for Player $\mathsf{I.}$
\end{proof}

\qquad\ \qquad\ \ 

With the proposition, we can prove the following corollary:

\begin{corollary}
Let $T$ be an Aronszajn tree, $A\subseteq T^{n+1}$ and $\left\{
M_{0},...,M_{n}\right\}  $ an $\in$-chain of countable elementary submodels
such that $T,A\in M_{0}.$ Assume there is $\left(  s_{0},..,s_{n}\right)  \in
A$ such that $\delta_{M_{0}}\leq\left\vert s_{0}\right\vert <\delta_{M_{1}%
}\leq\left\vert s_{1}\right\vert \leq...\leq\delta_{M_{n}}\leq\left\vert
s_{n}\right\vert .$ There are $\gamma$ and $F$ with the following properties:
\label{Corol del juego}

\begin{enumerate}
\item $\gamma<\delta_{M_{0}}.$

\item $F=\left(  t_{0},...,t_{n}\right)  \in T^{n+1}$ and $t_{i}\neq
s_{j}\upharpoonright\gamma$ for every $i,j\leq n.$

\item Player $\mathsf{II}$ has a winning strategy in \emph{G}$_{F}\left(
T,A\right)  .$
\end{enumerate}
\end{corollary}

\begin{proof}
By the previous proposition, we know there is $B=\left\{  F_{\alpha}\mid
\alpha\in\omega_{1}\right\}  $ such that for every $\alpha\in\omega_{1},$ we
have that $F_{\alpha}\in\left(  T_{\alpha}\right)  ^{n+1}$ and Player
$\mathsf{II}$ has a winning strategy in \emph{G}$_{F_{\alpha}}\left(
T,A\right)  .$ By elementarity. we may even assume that $B\in M_{0}.$
Proposition \ref{Prop Aronszajn distribuido} implies that there is $\gamma
_{0}\in M_{0}$ such that $B$ is $\left(  n+1\right)  $-distributed in at level
$\gamma_{0}.$ We can now find $\gamma\in M_{0}$ with $\gamma>\gamma_{0}$ such
that the $\gamma_{0}$ projections of $F_{\gamma}$ and $\left(  s_{0}%
,..,s_{n}\right)  $ are disjoint. It follows that $\gamma$ and $F$ are as desired.
\end{proof}

\qquad\ \ \qquad\ \ \ \ 

We will now define the forcing notion required to prove the main result of
this section. It is similar to the one of the previous section and is again
based on the forcing from Lemma 5.4.4 of \cite{Walks}, but this version is the
most complex of the three.

\begin{definition}
Define $\mathbb{P}\left(  C\right)  $ as the set of all $p=(\mathcal{M}%
,\mathcal{N})$ with the following properties:

\begin{enumerate}
\item $\mathcal{N}$ is a finite $\in$-chain of countable elementary submodels
of \textsf{H}$\left(  \omega_{2}\right)  $ and both $C$ and $\mathcal{C}$
belong to all of them.

\item $\mathcal{M\subseteq N}.$

\item \textbf{Completeness condition. }If $M,L\in\mathcal{M}$ are distinct,
then $\left[  \delta_{M}\delta_{L}\right]  \in C.$

\item \textbf{First} \textbf{crossing models condition. }If $M\in\mathcal{M}$
and $N,L\in\mathcal{N}$ are such that $N\in L\in M,$ then \textsf{min}%
$($\textsf{Tr}$(\delta_{N},\delta_{M})\setminus L)\in C.$

\item \textbf{Second crossing models condition. }Let $M,L\in\mathcal{M}$
distinct and $N\in\mathcal{N}$ with $N\in M.$ If $\triangle=\triangle\left(
\delta_{M},\delta_{L}\right)  \in N,$ then \textsf{min}$($\textsf{Tr}%
$(\triangle,\delta_{M})\setminus N)\in C.$
\end{enumerate}
\end{definition}

\qquad\ \qquad\ \qquad\ \ 

Similar to the forcing from the previous section, in the second crossing
models condition, the case $N=L$ is a consequence of the completeness
condition. Therefore, when verifying this condition for a potential element of
$\mathbb{P(}C\mathbb{)},$ we may always assume $N\neq L.$ Furthermore, in this
same condition, no relationship is assumed between $M,N$ and $L$ other than
$N\in M$ and $M\neq L.$ In particular, it is possible that $N\in L\in M,$
$L\in N\in M$ or $N\in M\in L.$

\qquad\ \ \ \qquad\ \ \ 

Let $p=(\mathcal{M}_{p},\mathcal{N}_{p})$ and $q=(\mathcal{M}_{q}%
,\mathcal{N}_{q})$ in $\mathbb{P}\left(  C\right)  .$ Define $p\leq q$ if
$\mathcal{M}_{q}\subseteq\mathcal{M}_{p},$ $\mathcal{N}_{q}\subseteq
\mathcal{N}_{p}$ and $\mathcal{M}_{q}=\mathcal{M}_{p}\cap\mathcal{N}_{q}.$ For
$p,q\in\mathbb{P}\left(  C\right)  ,$ define $q\sqsubseteq p$ if there is
$N\in\mathcal{N}_{p}$ such that $\mathcal{M}_{q}=\mathcal{M}_{p}\cap N$ and
$\mathcal{N}_{q}=\mathcal{N}_{p}\cap N.$ Clearly if $q\sqsubseteq p,$ then
$p\leq q.$

\qquad\ \qquad\ 

We now introduce the invariants\ for the elements of $\mathbb{P(}C\mathbb{)}$.

\begin{definition}
Let $p=(\mathcal{M},\mathcal{N})\in\mathbb{P}\left(  C\right)  .$ Define the
following items:

\begin{enumerate}
\item \textsf{ht}$\left(  p\right)  =\left\{  \delta_{N}\mid N\in
\mathcal{N}\right\}  $ and \textsf{ht}$_{\mathcal{M}}\left(  p\right)
=\left\{  \delta_{M}\mid M\in\mathcal{M}\right\}  $

\item $S\left(  p\right)  =S($\textsf{ht}$\left(  p\right)  )$ (see Definition
\ref{Def estructura generada}).

\item $H\left(  p\right)  =(S\left(  p\right)  ,<,$\textsf{ht}$_{\mathcal{M}%
}(p),$\textsf{ht}$(p),$\textsf{Tr}$\mathsf{,}$\textsf{\thinspace}%
$\triangle,\lambda,C).$
\end{enumerate}
\end{definition}

\qquad\ \qquad\ \ \ 

Let $p,q\in\mathbb{P}\left(  C\right)  $ and $h:S\left(  p\right)
\longrightarrow S\left(  q\right)  $ a bijection. We say that $h$ \emph{is an
isomorphism from }$H\left(  p\right)  $\emph{ to} $H\left(  q\right)  $ if for
every $\alpha,\beta\in S\left(  p\right)  ,$ the following holds:

\begin{enumerate}
\item $\alpha<\beta$ if and only if $h\left(  \alpha\right)  <h\left(
\beta\right)  .$

\item $\alpha\in$ \textsf{ht}$\left(  p\right)  $ if and only if $h\left(
\alpha\right)  \in$ \textsf{ht}$\left(  q\right)  .$

\item $\alpha\in$ \textsf{ht}$_{\mathcal{M}}\left(  p\right)  $ if and only if
$h\left(  \alpha\right)  \in$ \textsf{ht}$_{\mathcal{M}}\left(  q\right)  .$

\item $h[$\textsf{Tr}$(\alpha,\beta)]=$ \textsf{Tr}$(h\left(  \alpha\right)
,h\left(  \beta\right)  ).$

\item $h\left(  \triangle\left(  \alpha,\beta\right)  \right)  =\triangle
\left(  h\left(  \alpha\right)  ,h\left(  \beta\right)  \right)  .$

\item $h\left(  \lambda\left(  \alpha,\beta\right)  \right)  =\lambda\left(
h\left(  \alpha\right)  ,h\left(  \beta\right)  \right)  .$

\item $\alpha\in C$ if and only if $h\left(  \alpha\right)  \in C.$
\end{enumerate}

\qquad\qquad\qquad\qquad\qquad\ \ \ \ 

We say that $H\left(  p\right)  $ and $H\left(  q\right)  $ are isomorphic if
there is an isomorphism between them. As in the previous section, there is at
most one isomorphism from $H\left(  p\right)  $ to $H\left(  q\right)  $ and
there are only countably many isomorphism types.

\begin{definition}
We say that $M$ is a \emph{big model} if the following conditions hold:

\begin{enumerate}
\item There is $\kappa>\omega_{2}$ a large enough regular cardinal such that
$M$ $\ $is a countable elementary submodel of \textsf{H}$\left(
\kappa\right)  .$

\item $C,\mathcal{C},\mathbb{P}\left(  C\right)  $ $\mathbb{\in}$ $M.$
\end{enumerate}
\end{definition}

\qquad\ \qquad\ \ \qquad\ \ \ 

For $M$ a big model, denote $\overline{M}=M\cap$ \textsf{H}$\left(  \omega
_{2}\right)  ,$ which is an elementary submodel of \textsf{H}$\left(
\omega_{2}\right)  .$ We now prove the following

\begin{proposition}
Let $M$ be a big model and $p=(\mathcal{M}_{p},\mathcal{N}_{p})\in
\mathbb{P}\left(  C\right)  .$ If $\overline{M}\in\mathcal{N}_{p},$ then $p$
is $(M,\mathbb{P}\left(  C\right)  )$-generic.
\end{proposition}

\begin{proof}
For convenience, we will denote $T=T\left(  \rho_{0}\right)  .$ To simplify
notation, a subset $W\subseteq\omega_{1}$ will occasionally be identified with
the corresponding set $\{\rho_{0\alpha}\mid\alpha\in W\}.$

\qquad\ \qquad\ \qquad\ \qquad\ \qquad\ \qquad\ 

Let $D\in M$ be an open dense subset of $\mathbb{P}\left(  C\right)  $. By
extending $p$ if necessary, we may assume that $p\in D.$ We need to prove that
$p$ is compatible with an element of $D\cap M.$ Denote $\delta=\delta_{M},$
which is an element of $C.$ Define $X=$ \textsf{ht}$\left(  p\right)
\setminus M$ and let $m$ be the size of $X.$ Note that $\delta$ is the
smallest element of $X.$ Let $p_{M}=(\mathcal{M}_{p}\cap M,$ $\mathcal{N}%
_{p}\cap M).$ It is easy to see that $p_{M}\in\mathbb{P}\left(  C\right)  \cap
M$ and $p_{M}\sqsubseteq p.$

\qquad\ \qquad\ \qquad\ \ \ 

Find $\eta_{0}<\delta$ such that $S\left(  p\right)  \cap M\subseteq\eta_{0}.$
We now define $E$ as the set of all $Y\in\left[  \omega_{1}\right]  ^{m}$ for
which there is $r\in D$ with the following properties:

\begin{enumerate}
\item $p_{M}\sqsubseteq r.$

\item \textsf{ht}$\left(  r\right)  =$ \textsf{ht}$\left(  p_{M}\right)  \cup
Y.$

\item $S\left(  r\right)  \cap\eta_{0}=S\left(  p\right)  \cap\eta_{0}$ (which
is the same as $S\left(  p\right)  \cap M$).

\item $H\left(  p\right)  $ and $H\left(  r\right)  $ are isomorphic and if
$h:S\left(  p\right)  \longrightarrow S\left(  r\right)  $ denote the (unique)
isomorphism, then:

\begin{enumerate}
\item $h$ is the identity mapping on $S\left(  p\right)  \cap M.$

\item $\rho_{0\alpha}\upharpoonright\eta_{0}=\rho_{0h\left(  \alpha\right)
}\upharpoonright\eta_{0}$ for all $\alpha\in X.$
\end{enumerate}
\end{enumerate}

\qquad\ \qquad\ \qquad\ \ \ 

Note that $X\in E.$ Moreover, $E$ is in $M.$ This is because the definition of
$E$ can can be reformulated without reference to objects outside $M.$ For
$Y\in E$ and $r$ as above, we will say that $r$ is a witness of $Y\in E.$

\begin{claim}
There are $\eta_{1}$ and $F=\left(  t_{0},...,t_{m-1}\right)  $ with the
following properties:

\begin{enumerate}
\item $\eta_{1}\in M$ and $\eta_{0}<\eta_{1}<\delta.$

\item $F=\left(  t_{0},...,t_{m-1}\right)  \in\left(  T_{\eta_{1}}\right)
^{m}$ and $t_{i}\neq\rho_{0\alpha}\upharpoonright\eta_{1}$ for every
$\alpha\in X$ and $i<m.$

\item Player $\mathsf{II}$ has a winning strategy in the game \emph{G}%
$_{F}\left(  T,E\right)  .$
\end{enumerate}
\end{claim}

\qquad\ \ \ \qquad\ 

Indeed, the existence of $F$ and $\eta_{1}$ follows by Corollary
\ref{Corol del juego}. We can now find $\eta$ with the following properties:

\begin{enumerate}
\item $\eta_{1}<\eta<\delta.$

\item If $\alpha,\beta\in X$ and $t\in F,$ then \textsf{Tr}$(\triangle
(t,\rho_{0\alpha}),\beta)\cap\delta\subseteq\eta.$
\end{enumerate}

\qquad\ \qquad\ \ 

We now have the following:

\begin{claim}
There is $Y\in E\cap M$ with the following properties:
\label{Claim PFA C suc cosa Y}

\begin{enumerate}
\item $\eta<$ \textsf{min}$\left(  Y\right)  .$

\item Every element of $Y$ extends an element of $F.$

\item If $\gamma\in X$ and $\alpha,\beta\in Y$ are such that $\alpha<\beta,$
then \textsf{Tr}$(\alpha,\gamma)\cap M\subseteq\beta.$
\end{enumerate}
\end{claim}

\qquad\ \ \ \ \ 

Let $\sigma\in M$ be a winning strategy for Player $\mathsf{II}$ in the game
\emph{G}$_{F}\left(  T,E\right)  .$ Consider the following run of the game (in
which Player $\mathsf{II}$ is following her strategy\footnote{Note that in
order for $\sigma$ to be a winning strategy, it is necessary that all the
moves of Player $\mathsf{II}$ are of the form $\rho_{0\varepsilon}$ for
$\varepsilon<\omega_{1}$ (otherwise she inmediately loses the game).} $\sigma
$).\newline

\hfill%
\begin{tabular}
[c]{|l|l|l|l|l|l|}\hline
$\mathsf{I}$ & $\eta$ &  & $\xi_{0}$ &  & $...$\\\hline
$\mathsf{II}$ &  & $\rho_{0\varepsilon_{0}}$ &  & $\rho_{0\varepsilon_{1}}$ &
\\\hline
\end{tabular}
\hfill\ 

\qquad\ \qquad\ \ 

In which Player $\mathsf{I}$ plays as follows:

\begin{enumerate}
\item His first move is $\eta.$

\item After Player $\mathsf{II}$ played her move (say $\rho_{0\varepsilon_{0}%
}$), Player $\mathsf{I}$ chooses $\xi_{0}\in M$ that is larger than $\eta,$
$\varepsilon_{0}$ and $\bigcup\limits_{\gamma\in X}$\textsf{Tr}$(\varepsilon
_{0},\gamma)\cap M.$

\item After Player $\mathsf{II}$ played her move (say $\rho_{0\varepsilon_{1}%
}$), Player $\mathsf{I}$ chooses $\xi_{1}\in M$ that is larger than
$\varepsilon_{1}$ and $\bigcup\limits_{\gamma\in X}$\textsf{Tr}$(\varepsilon
_{1},\gamma)\cap M.$

\item[$\vdots$] $\vdots\hfill\vdots\hfill\vdots$
\end{enumerate}

\qquad\ \ \qquad

Since $\sigma$ is a winning strategy for Player $\mathsf{II,}$ it follows that
$Y=\left(  \varepsilon_{0},...,\varepsilon_{m-1}\right)  $ \ is in $E$ and
clearly has the desired properties. This finishes the proof of the claim.

\qquad\qquad\qquad

Fix $Y$ as in the claim. By elementarity., there exists $r\in D\cap M$ that is
a witnesses of $Y\in E.$ Let $h:S\left(  p\right)  \longrightarrow S\left(
r\right)  $ be the unique isomorphism from $H\left(  p\right)  $ to $H\left(
r\right)  .$ Recall that $h$ is the identity on $M.$ We claim that $p$ and $r$
are compatible. It suffices to show that $q=(\mathcal{M}_{p}\cup
\mathcal{M}_{r},\mathcal{N}_{p}\cup\mathcal{N}_{r})$ belongs to $\mathbb{P(}%
C\mathbb{)}.$ Clearly $\mathcal{N}_{p}\cup\mathcal{N}_{r}$ is a finite $\in
$-chain and contains $\mathcal{M}_{p}\cup\mathcal{M}_{r}.$

\qquad\ \qquad\ \ \ 

To verify that $q$ satisfies the completeness condition, it suffices to show
that for any $\alpha,\beta\in$ \textsf{ht}$_{\mathcal{M}}\left(  p\right)
\setminus M,$ we have that $\left[  h\left(  \alpha\right)  \beta\right]  \in
C.$ Denote $\triangle=\triangle\left(  h\left(  \alpha\right)  ,\beta\right)
.$ Recall that $h\left(  \alpha\right)  \in M$ while $\beta\notin M$, so
$h\left(  \alpha\right)  <\beta.$ It follows that $\left[  h\left(
\alpha\right)  \beta\right]  =$ \textsf{min}$($\textsf{Tr}$(\triangle
,\beta)\setminus h\left(  \alpha\right)  $). We proceed by cases, depending on
whether $\rho_{0\alpha}\upharpoonright\delta$ and $\rho_{0\beta}%
\upharpoonright\delta$ are equal.$\medskip$

$%
\begin{tabular}
[c]{|l|}\hline
\textbf{Case:} $\rho_{0\alpha}\upharpoonright\delta\neq\rho_{0\beta
}\upharpoonright\delta$ (equivalently, $\triangle\left(  \alpha,\beta\right)
<\delta$)\\\hline
\end{tabular}
\ \medskip$

In this case, we have that $\triangle=\triangle\left(  \alpha,\beta\right)  $
and \textsf{Tr}$(\triangle,\beta)\cap M\subseteq\eta_{0}\subseteq h\left(
\alpha\right)  .$ It follows that \textsf{Tr}$(\triangle,\beta)\setminus
h\left(  \alpha\right)  =$ \textsf{Tr}$(\triangle,\beta)\setminus\delta$ and
therefore \newline$\left[  h\left(  \alpha\right)  \beta\right]  =$
\textsf{min}$($\textsf{Tr}$(\triangle,\beta)\setminus\delta)$. The proof now
proceeds by subcases, depending on whether $\beta$ is equal to $\delta
$.\medskip

\textbf{Subcase: }$\rho_{0\alpha}\upharpoonright\delta\neq\rho_{0\beta
}\upharpoonright\delta$ and $\beta=\delta.$

\ \ \ \ \ 

We have that \textsf{Tr}$(\triangle,\beta)\setminus\delta=\left\{
\delta\right\}  ,$ so $\left[  h\left(  \alpha\right)  \beta\right]
=\delta\in C.\medskip$

\textbf{Subcase: }$\rho_{0\alpha}\upharpoonright\delta\neq\rho_{0\beta
}\upharpoonright\delta$ and $\beta\neq\delta.$

\qquad\qquad\qquad\ \ \ \ 

We know that $\left[  h\left(  \alpha\right)  \beta\right]  =$ \textsf{min}%
$($\textsf{Tr}$(\triangle,\beta)\setminus\delta),$ which is in $C$ by the
second crossing models condition of $p.\medskip$

$%
\begin{tabular}
[c]{|l|}\hline
\textbf{Case:} $\rho_{0\alpha}\upharpoonright\delta=\rho_{0\beta
}\upharpoonright\delta$ (equivalently, $\delta\leq\triangle\left(
\alpha,\beta\right)  $)\\\hline
\end{tabular}
\ \medskip$

The proof now proceeds by subcases, depending on the relationship of $\alpha$
and $\beta$ with $\delta$.\medskip

\textbf{Subcase:}\ \ \ $\rho_{0\alpha}\upharpoonright\delta=\rho_{0\beta
}\upharpoonright\delta$ and $\alpha=\delta$ or $\beta=\delta.$

\qquad\qquad\qquad\ \ 

Lemma \ref{Lema limite anticadena} implies that if $\rho_{0\alpha
}\upharpoonright\delta=\rho_{0\beta}\upharpoonright\delta$ and $\alpha=\delta$
or $\beta=\delta,$ then $\alpha=\beta=\delta.$ We have that \textsf{Tr}%
$(\triangle\left(  h\left(  \delta\right)  ,\delta\right)  ,\delta)\cap
\delta\subseteq\eta\subseteq h\left(  \delta\right)  ,$ so $\left[  h\left(
\delta\right)  \delta\right]  =\delta\in C.\medskip$

\textbf{Subcase:}\ \ \ $\rho_{0\alpha}\upharpoonright\delta=\rho_{0\beta
}\upharpoonright\delta$ and $\alpha,\beta\neq\delta.$

\qquad\ \qquad\ \qquad\ 

Recall that $\lambda\left(  \delta,\beta\right)  <\eta_{0}<\triangle.$
Proposition \ref{Prop juntar trazas} implies that $\delta\in$ \textsf{Tr}%
$(\triangle,\beta).$ Since \textsf{Tr}$(\triangle,\beta)\cap M\subseteq
\eta\subseteq h\left(  \alpha\right)  ,$ it follows that $\left[  h\left(
\alpha\right)  \beta\right]  =\delta\in C.$

\qquad\ \qquad\ \qquad\ \ \qquad\ \qquad\ \ 

This completes the proof that $q$ satisfies the completeness condition. We now
verify the first crossing models condition. It suffices to show that for every
$\alpha<\beta<\gamma$ with $\gamma\in$ \textsf{ht}$_{\mathcal{M}}\left(
p\right)  \setminus M$ and $\alpha,\beta\in$ \textsf{ht}$\left(  p\right)
\cup$ \textsf{ht}$\left(  r\right)  ,$ we have that \textsf{min}$($%
\textsf{Tr}$(\alpha,\gamma)\setminus\beta)\in C.$ The proof proceeds by
cases.$\medskip$

$%
\begin{tabular}
[c]{|l|}\hline
\textbf{Case:} \textbf{ }$\alpha\in$ \textsf{ht}$\left(  p\right)  $\\\hline
\end{tabular}
\ \medskip$

If $\beta$ is also in \textsf{ht}$\left(  p\right)  ,$ then there is nothing
to do since $p\in\mathbb{P}\left(  C\right)  .$ Assume that $\beta\in$
\textsf{ht}$\left(  r\right)  \setminus$ \textsf{ht}$\left(  p\right)  .$ Note
that \textsf{Tr}$(\alpha,\gamma)\cap M\subseteq\eta_{0}\subseteq\beta.$ It
follows that \newline\textsf{min}$($\textsf{Tr}$(\alpha,\gamma)\setminus
\beta)=$ \textsf{min}$($\textsf{Tr}$(\alpha,\gamma)\setminus\delta)$, which is
in $C$ by the first crossing models condition of $p.\medskip$

$%
\begin{tabular}
[c]{|l|}\hline
\textbf{Case:} $\alpha\in$ \textsf{ht}$\left(  r\right)  \setminus$
\textsf{ht}$\left(  p\right)  $ and $\beta\in$ \textsf{ht}$\left(  p\right)
$\\\hline
\end{tabular}
\ \medskip$

Note that $\delta\leq\beta.$ Since $\lambda\left(  \delta,\gamma\right)
<\eta_{0}<\alpha,$ Proposition \ref{Prop juntar trazas} implies that
\textsf{Tr}$(\delta,\gamma)$ is an initial segment of \textsf{Tr}%
$(\alpha,\gamma).$ It follows that \newline\textsf{min}$($\textsf{Tr}%
$(\alpha,\gamma)\setminus\beta)=$ \textsf{min}$($\textsf{Tr}$(\delta
,\gamma)\setminus\delta)$, which is in $C$ by the first crossing models
condition of $p.\medskip$

$%
\begin{tabular}
[c]{|l|}\hline
\textbf{Case:} $\alpha\in$ \textsf{ht}$\left(  r\right)  \setminus$
\textsf{ht}$\left(  p\right)  $ and $\beta\in$ \textsf{ht}$\left(  r\right)
$\\\hline
\end{tabular}
\ \medskip$

Since $\lambda\left(  \delta,\gamma\right)  <\eta_{0}<\alpha,$ Proposition
\ref{Prop juntar trazas} implies that $\delta\in$ \textsf{Tr}$(\alpha
,\gamma).$ Furthermore, since \textsf{Tr}$(\alpha,\gamma)\cap M\subseteq\beta$
(see Claim \ref{Claim PFA C suc cosa Y}), it follows that \newline%
\textsf{min}$($\textsf{Tr}$(\alpha,\gamma)\setminus\beta)=\delta\in C.$

\qquad\qquad\qquad\qquad

This completes the proof that $q$ satisfies the first crossing models
condition. We now verify the second crossing models condition. It suffices to
show that for every $N\in$ $\mathcal{N}_{p}\cup\mathcal{N}_{r}$ and
$\alpha,\beta\in$ \textsf{ht}$_{\mathcal{M}}\left(  p\right)  \cup$
\textsf{ht}$_{\mathcal{M}}\left(  r\right)  $ are such that $\triangle
=\triangle\left(  \alpha,\beta\right)  \in N$ and $\delta_{N}<\beta,$
$\alpha\neq\delta_{N},$ we have that \textsf{min}$($\textsf{Tr}$(\triangle
,\beta)\setminus N)\in C.$ As the reader probably guessed, the proof proceeds
by cases.$\medskip$

$%
\begin{tabular}
[c]{|l|}\hline
\textbf{Case:} $N\in\mathcal{N}_{p}\cap\mathcal{N}_{r}$ (equivalently,
$\mathcal{N\in N}_{p}\cap M$)\\\hline
\end{tabular}
\ \medskip$

If either $\alpha,\beta\in$ \textsf{ht}$\left(  p\right)  $ or $\alpha
,\beta\in$ \textsf{ht}$\left(  r\right)  ,$ then there is nothing to do since
both $p$ and $r$ are conditions, so we assume this is not the case.\medskip

\textbf{Subcase:} $N\in\mathcal{N}_{p}\cap\mathcal{N}_{r},$ $\alpha\in$
\textsf{ht}$\left(  p\right)  \setminus$ \textsf{ht}$\left(  r\right)  $ and
$\beta\in$ \textsf{ht}$\left(  r\right)  \setminus$ \textsf{ht}$\left(
p\right)  .$

\qquad\ \qquad\ \qquad\ \ 

Note that $\alpha\notin M.$ Moreover, since $\triangle\in N,$ it follows that
$\triangle<\eta_{0},$ which implies that $\triangle=\triangle\left(  h\left(
\alpha\right)  ,\beta\right)  .$ In this way, \textsf{min}$($\textsf{Tr}%
$(\triangle,\beta)\setminus N)=$ \textsf{min}$($\textsf{Tr}$(\triangle\left(
h\left(  \alpha\right)  ,\beta\right)  ,\beta)\setminus N),$ which is in $C$
by the second crossing models condition of $r.\medskip$

\textbf{Subcase:} $N\in\mathcal{N}_{p}\cap\mathcal{N}_{r},$ $\alpha\in$
\textsf{ht}$\left(  r\right)  \setminus$ \textsf{ht}$\left(  p\right)  $ and
$\beta\in$ \textsf{ht}$\left(  p\right)  \setminus$ \textsf{ht}$\left(
r\right)  .$

\qquad\qquad\ \ \qquad\ \ 

Note that $\beta\notin M.$ Since $\triangle<\eta_{0},$ it follows that
$\triangle=\triangle(h^{-1}\left(  \alpha\right)  ,\beta).$ In this way,
\newline\textsf{min}$($\textsf{Tr}$(\triangle,\beta)\setminus N)=$
\textsf{min}$($\textsf{Tr}$(\triangle\left(  h^{-1}\left(  \alpha\right)
,\beta\right)  ,\beta)\setminus N),$ which is in $C$ by the second crossing
models condition of $p.\medskip$

$%
\begin{tabular}
[c]{|l|}\hline
\textbf{Case:} $N\in\mathcal{N}_{r}\setminus\mathcal{N}_{p}$\\\hline
\end{tabular}
\ \medskip$

If both $\alpha,\beta$ belong to \textsf{ht}$\left(  r\right)  ,$ then there
is nothing to do, so we assume this is not the case.\medskip

\textbf{Subcase: }$N\in\mathcal{N}_{r}\setminus\mathcal{N}_{p}$, $\alpha\in$
\textsf{ht}$\left(  r\right)  $ and $\beta\notin$ \textsf{ht}$\left(
r\right)  .$

\qquad\ \ 

Note that $\beta\notin M.$ Let $\overline{\alpha}=h^{-1}\left(  \alpha\right)
.$ First consider the case where $\rho_{0\beta}\upharpoonright\delta\neq
\rho_{0\overline{\alpha}}\upharpoonright\delta.$ It follows that
$\triangle=\triangle(\overline{\alpha},\beta)$ and \textsf{Tr}$(\triangle
,\beta)\cap M\subseteq\eta_{0}\subseteq\eta.$ Moreover, \textsf{min}%
$($\textsf{Tr}$(\triangle,\beta)\setminus N)=$ \textsf{min}$($\textsf{Tr}%
$(\triangle\left(  \overline{\alpha},\beta\right)  ,\beta)\setminus M).$ If
$\beta=\delta,$ then this minimum is $\delta,$ which is in $C.$ If
$\delta<\beta,$ the conclusion follows by the second crossing models of $p.$
We now assume that\ $\rho_{0\beta}\upharpoonright\delta=\rho_{0\overline
{\alpha}}\upharpoonright\delta.$ We have that $\lambda\left(  \delta
,\beta\right)  <\eta_{0}<\triangle.$ Proposition \ref{Prop juntar trazas}
implies that $\delta\in$ \textsf{Tr}$(\triangle,\beta).$ Moreover,
\newline\textsf{Tr}$(\triangle,\beta)\cap M\subseteq\eta\subseteq N.$ It
follows that \textsf{min}$($\textsf{Tr}$(\triangle,\beta)\setminus
N)=\delta\in C.\medskip$

\textbf{Subcase: }$N\in\mathcal{N}_{r}\setminus\mathcal{N}_{p}$, $\alpha
\notin$ \textsf{ht}$\left(  r\right)  $ and $\beta$ $\in$ \textsf{ht}$\left(
r\right)  .$

\qquad\ \ \ \qquad\ \qquad\ \ 

Note that $\alpha\notin M.$ Let $\overline{\beta}=h^{-1}\left(  \beta\right)
.$ First consider the case where $\rho_{0\alpha}\upharpoonright\delta\neq
\rho_{0\overline{\beta}}\upharpoonright\delta.$ It follows that $\triangle
=\triangle(h\left(  \alpha\right)  ,\beta).$ Moreover, \textsf{min}%
$($\textsf{Tr}$(\triangle,\beta)\setminus N)=$ \textsf{min}$($\textsf{Tr}%
$(\triangle\left(  h\left(  \alpha\right)  ,\beta\right)  ,\beta)\setminus
N),$ which is in $C$ by the second crossing models condition of $r.$ We now
assume that\ $\rho_{0\alpha}\upharpoonright\delta=\rho_{0\overline{\beta}%
}\upharpoonright\delta.$ Since $\lambda(\delta,\overline{\beta})\in M$ and $h$
is an isomorphism that is the identity in $M,$ we have the following:\newline

\hfill%
\begin{tabular}
[c]{lll}%
$\lambda(\delta,\overline{\beta})$ & $=$ & $h(\lambda(\delta,\overline{\beta
}))$\\
& $=$ & $\lambda(h\left(  \delta\right)  ,h(\overline{\beta}))$\\
& $=$ & $\lambda(h\left(  \delta\right)  ,\beta)$%
\end{tabular}
\hfill\ 

\qquad\ \qquad\ \qquad\ \qquad\ \ \ \ 

Since $\lambda(\delta,\overline{\beta})<\eta_{0}<\triangle,$ it follows that
$\lambda(h\left(  \delta\right)  ,\beta)<\triangle.$ Proposition
\ref{Prop juntar trazas} implies that \textsf{Tr}$(h\left(  \delta\right)
,\beta)$ is an initial segment of \textsf{Tr}$(\triangle,\beta).$ Since
$h\left(  \delta\right)  \leq\delta_{N},$ we get that \textsf{min}%
$($\textsf{Tr}$(\triangle,\beta)\setminus N)=$ \textsf{min}$($\textsf{Tr}%
$(h\left(  \delta\right)  ,\beta)\setminus N),$ which is in $C$ by the first
crossing models condition of $p.\medskip$

\textbf{Subcase: }$N\in\mathcal{N}_{r}\setminus\mathcal{N}_{p}$ and
$\alpha,\beta\notin$ \textsf{ht}$\left(  r\right)  .$

\qquad\ \qquad\ \ 

Since $\triangle\in N,$ it follows that $\triangle\in M,$ which implies that
$\triangle<\eta_{0}.$ Moreover, \textsf{Tr}$(\triangle,\beta)\cap
M\subseteq\eta_{0}\in N.$ It follows that \textsf{min}$($\textsf{Tr}%
$(\triangle,\beta)\setminus N)=$ \textsf{min}$($\textsf{Tr}$(\triangle
,\beta)\setminus M),$ which is in $C$ by the second crossing models condition
of $p.\medskip$

$%
\begin{tabular}
[c]{|l|}\hline
\textbf{Case:} $N\in\mathcal{N}_{p}\setminus\mathcal{N}_{r}$\\\hline
\end{tabular}
\ \medskip$

Since $\delta_{N}<\beta,$ it follows that $\beta\in$ \textsf{ht}$\left(
p\right)  \setminus$ \textsf{ht}$\left(  r\right)  .$ If $\alpha$ is also in
\textsf{ht}$\left(  p\right)  ,$ then there is nothing to do, so assume that
$\alpha\in$ \textsf{ht}$\left(  r\right)  \setminus$ \textsf{ht}$\left(
p\right)  .$ Let $\overline{\alpha}=h^{-1}\left(  \alpha\right)  .$ If
$\rho_{0\overline{\alpha}}\upharpoonright\delta\neq\rho_{0\beta}%
\upharpoonright\delta,$ then $\triangle=\triangle(\overline{\alpha},\beta),$
so \textsf{min}$($\textsf{Tr}$(\triangle,\beta)\setminus N)$ is equal to
\textsf{min}$($\textsf{Tr}$(\triangle(\overline{\alpha},\beta),\beta)\setminus
N)$ and it is in $C$ by the second crossing models condition of $p.$ Assume
now that $\rho_{0\overline{\alpha}}\upharpoonright\delta=\rho_{0\beta
}\upharpoonright\delta.$ It follows that $\lambda(\delta_{N},\beta
)<\triangle.$ Proposition \ref{Prop juntar trazas} implies that $\delta_{N}%
\in$ \textsf{Tr}$(\triangle,\beta),$ so \textsf{min}$($\textsf{Tr}%
$(\triangle,\beta)\setminus N)=\delta_{N}\in C.$ At last we finished the proof.
\end{proof}

\qquad\ \qquad\ \qquad\ \qquad\ \qquad\ 

Arguing as in the previous section, we obtain the following result:

\begin{theorem}
\textsf{PFA }implies that every square-bracket operation of the form
$[\cdot\cdot]_{\mathcal{C}}$ has the Ramsey club property.
\end{theorem}

\qquad\ \ \ 

With this result and Theorems \ref{PFA para reales} and
\ref{Teorema PFA Aronszajn}, we conclude:

\begin{corollary}
\textsf{PFA }implies that every square-bracket operation has the Ramsey club property.
\end{corollary}

\section{Open questions}

We believe there remains much to be discovered on the topics explored in this
paper. In this final section, we present several open questions that we have
been unable to solve, hoping they may guide future research on the topic.

\qquad\ \ \ 

We proved that \textsf{MA }is consistent with the failure of the Ramsey club
property for the square-bracket operations induced by \textsf{C}-sequences. We
want to know whether this is also true for the other types of square-bracket operations.

\begin{problem}
Is the statement \textquotedblleft No square-bracket operation has the Ramsey
club property\textquotedblright\ consistent with \textsf{MA?}
\end{problem}

\qquad\ \qquad\ \ \qquad

As noted earlier, a positive resolution of the problem above would follow if
we could prove Proposition \ref{Prop ccc no arregla} for the other
square-bracket operations. We conjecture this is the case, but were unable to
prove it.

\qquad\qquad\qquad\qquad\qquad\qquad

In Proposition \ref{Prop SC destruye a V} for the case of the square-bracket
operation induced by a set of reals and a nice sequence, we required the
additional hypothesis that the tree generated by the nice sequence is
Aronszajn. It is unclear if this assumption is necessary.

\begin{problem}
Is the statement \textquotedblleft No square-bracket operation induced by a
set of reals and a nice sequence has the Ramsey club
property\textquotedblright\ consistent with $\mathsf{\Diamond?}$
\end{problem}

\qquad\ \ \qquad\qquad\qquad

We conjecture that Proposition \ref{Prop SC destruye a V} holds for all
square-bracket operations. A proof of this would solve the previous question.

\qquad\qquad\ \qquad\ \ 

Probably the most interesting problem on this topic is the following: \qquad\ \ \ \ 

\begin{problem}
Is \textsf{CH }consistent with the statement \textquotedblleft every
square-bracket operation has the Ramsey club property\textquotedblright? What
about other guessing principles like $\Diamond,\Diamond^{\ast},\Diamond^{+}$
or $\Diamond^{\#}$?
\end{problem}

\qquad\ \ \ \qquad\qquad\ \ \ 

It is possible that the method developed by Asper\'{o} and Mota in
\cite{FewNewReals} provide a path to a positive answer for the first part of
question. It would also be interesting to know if the \emph{P-ideal Dichotomy}
(\textsf{PID) }introduced by the second author in
\cite{ADichotomyforPIdealsofCountableSets}\textsf{ }(see also
\cite{PartitionPropertiesCompatiblewithCH} and \cite{NotesonForcingAxioms})
has any influence on this topic.

\begin{problem}
Does \textsf{PID }imply that every square-bracket operation has the Ramsey
club property?
\end{problem}

\qquad\ \qquad\ \ \ \qquad\ \ \ 

We would like to know if there are models where there are two square-bracket
operations that behave different in respect to the Ramsey club property.

\begin{problem}
Is it consistent that there exists two \textsf{C}-sequences $\mathcal{C}$ and
$\mathcal{D}$ such that $[\cdot\cdot]_{\mathcal{C}}$ has the Ramsey club
property but $[\cdot\cdot]_{\mathcal{D}}$ does not? (A similar question
applies to the other classes of square-bracket operations).
\end{problem}

\qquad\qquad\qquad

It would also be interesting to compare the behavior between different classes
of square-bracket operations. For example:

\begin{problem}
Is it consistent that every square-bracket operation induced by a
\textsf{C}-sequence has the Ramsey club property, while no square-bracket
operation induced by a sequence of reals and a nice sequence has such property?
\end{problem}

\qquad\ \qquad\ \ 

Naturally, we are interested in all variants of the previous question.

\qquad\ \qquad\ \ \ \ \ \qquad\ \ \ 

The influence of the underlying combinatorial structure of the square-bracket
operation is not completely understood. For example, we could ask the following:

\begin{problem}
Let $\mathcal{R=}$ $\left\{  r_{\alpha}\mid\alpha\in\omega_{1}\right\}
\subseteq2^{\omega}$ and $e=\left\langle e_{\beta}\mid\beta\in\omega
_{1}\right\rangle $ a nice sequence. How does the combinatorial properties of
$\mathcal{R}$ reflect on the properties of $[\cdot\cdot]_{\mathcal{R},e}$? For
example, is there anything interesting to be said in case $\mathcal{R}$ is a
Luzin or a Sierpi\'{n}ski set?
\end{problem}

Fix a square-bracket operation $[\cdot\cdot]$. For a subset$\ A\subseteq
\omega_{1},$ define $C\left(  A\right)  =\{\left[  \alpha\beta\right]
\mid\alpha,\beta\in A\}.$ It is natural to consider the following problem:

\begin{problem}
What can we say about the following statement: \textquotedblleft For every
$A,B\in\left[  \omega_{1}\right]  ^{\omega_{1}}$ there is $D\in\left[
\omega_{1}\right]  ^{\omega_{1}}$ such that $C\left(  D\right)  \subseteq
C\left(  A\right)  \cap C\left(  B\right)  $\textquotedblright.
\end{problem}

\qquad\ \qquad\ \ 

Of course, such statement follows from \textsf{PFA}, but we do not know much
more than that. Since we use the Proper Forcing Axiom to prove that all
square-bracket operations have the Ramsey club property, it is natural to ask
the following questions:

\begin{problem}
Is the statement \textquotedblleft all square-bracket operations have the
Ramsey club property\textquotedblright\ consistent with an arbitrarily large continuum?
\end{problem}

\begin{problem}
Is it possible to force that all square-bracket operations have the Ramsey
club property without any Large Cardinal hypothesis?
\end{problem}

\def\cprime{$'$}

Osvaldo Guzm\'{a}n

Centro de Ciencias Matem\'{a}ticas, UNAM.

oguzman@matmor.unam.mx

\qquad\qquad\qquad\ \ 

Stevo Todorcevic

Department of Mathematics, University of Toronto, Canada

stevo@math.toronto.edu

\qquad\ \qquad\ \ \ \qquad\ \ \ 

Institut de Math\'{e}matiques de Jussieu, CNRS, Paris, France

stevo.todorcevic@imj-prg.fr

\qquad\qquad\ \qquad\ \ \ 

Mthematical Institute of the Serbian Academy of Sciences and Arts, Belgrade, Serbia

stevo.todorcevic@sanu.ac.rs


\begin{thebibliography}{10}

\bibitem{AbrahamHandbook}
Uri Abraham.
\newblock Proper forcing.
\newblock In {\em Handbook of set theory. {V}ols. 1, 2, 3}, pages 333--394.
  Springer, Dordrecht, 2010.

\bibitem{ForcingClubs}
Uri Abraham and Saharon Shelah.
\newblock Forcing closed unbounded sets.
\newblock {\em J. Symbolic Logic}, 48(3):643--657, 1983.

\bibitem{PartitionPropertiesCompatiblewithCH}
Uri Abraham and Stevo Todor\v{c}evi\'{c}.
\newblock Partition properties of {$\omega_1$} compatible with {CH}.
\newblock {\em Fund. Math.}, 152(2):165--181, 1997.

\bibitem{AddingBaumgartnerClubs}
David Asper\'{o}.
\newblock Adding many {B}aumgartner clubs.
\newblock {\em Arch. Math. Logic}, 56(7-8):797--810, 2017.

\bibitem{AsperoMota}
David Asper\'{o} and Miguel~Angel Mota.
\newblock Forcing consequences of {PFA} together with the continuum large.
\newblock {\em Trans. Amer. Math. Soc.}, 367(9):6103--6129, 2015.

\bibitem{AGeneralizationofMA}
David Asper\'{o} and Miguel~Angel Mota.
\newblock A generalization of {M}artin's axiom.
\newblock {\em Israel J. Math.}, 210(1):193--231, 2015.

\bibitem{SeparatingClubPrinciples}
David Asper\'{o} and Miguel~Angel Mota.
\newblock Separating club-guessing principles in the presence of fat forcing
  axioms.
\newblock {\em Ann. Pure Appl. Logic}, 167(3):284--308, 2016.

\bibitem{FewNewReals}
David Asper\'{o} and Miguel~Angel Mota.
\newblock Few new reals.
\newblock {\em J. Math. Log.}, 24(2):Paper No. 2350009, 35, 2024.

\bibitem{AddingaClub}
J.~E. Baumgartner, L.~A. Harrington, and E.~M. Kleinberg.
\newblock Adding a closed unbounded set.
\newblock {\em J. Symbolic Logic}, 41(2):481--482, 1976.

\bibitem{Allw1DenseSetsofRealscanbeIsomorphic}
James~E. Baumgartner.
\newblock All {$\aleph _{1}$}-dense sets of reals can be isomorphic.
\newblock {\em Fund. Math.}, 79(2):101--106, 1973.

\bibitem{IteratedForcing}
James~E. Baumgartner.
\newblock Iterated forcing.
\newblock In {\em Surveys in set theory}, volume~87 of {\em London Math. Soc.
  Lecture Note Ser.}, pages 1--59. Cambridge Univ. Press, Cambridge, 1983.

\bibitem{ApplicationsofPFA}
James~E. Baumgartner.
\newblock Applications of the proper forcing axiom.
\newblock In {\em Handbook of set-theoretic topology}, pages 913--959.
  North-Holland, Amsterdam, 1984.

\bibitem{BaumgartnerTesis}
James~Earl Baumgartner.
\newblock {\em Results and independence proofs in combinatorial set theory}.
\newblock ProQuest LLC, Ann Arbor, MI, 1970.
\newblock Thesis (Ph.D.)--University of California, Berkeley.

\bibitem{TopicsinSetTheory}
M.~Bekkali.
\newblock {\em Topics in set theory}, volume 1476 of {\em Lecture Notes in
  Mathematics}.
\newblock Springer-Verlag, Berlin, 1991.
\newblock Lebesgue measurability, large cardinals, forcing axioms,
  rho-functions, Notes on lectures by Stevo Todor\v{c}evi\'{c}.

\bibitem{FirstOmegaAlephs}
Jeffrey Bergfalk.
\newblock The first omega alephs: from simplices to trees of trees to higher
  walks.
\newblock {\em Adv. Math.}, 393:Paper No. 108083, 74, 2021.

\bibitem{AlanSubmodelos}
Alan Dow.
\newblock An introduction to applications of elementary submodels to topology.
\newblock {\em Topology Proc.}, 13(1):17--72, 1988.

\bibitem{IteratedForcingandCH}
Todd Eisworth, Justin~Tatch Moore, and David Milovich.
\newblock Iterated forcing and the continuum hypothesis.
\newblock In {\em Appalachian set theory 2006--2012}, volume 406 of {\em London
  Math. Soc. Lecture Note Ser.}, pages 207--244. Cambridge Univ. Press,
  Cambridge, 2013.

\bibitem{TasteofGoldstern}
Martin Goldstern.
\newblock A taste of proper forcing.
\newblock In {\em Set theory ({C}ura\c{c}ao, 1995; {B}arcelona, 1996)}, pages
  71--82. Kluwer Acad. Publ., Dordrecht, 1998.

\bibitem{Rendezvous}
Michael Hrusak.
\newblock {\em Rendezvous with madness}.
\newblock ProQuest LLC, Ann Arbor, MI, 1999.
\newblock Thesis (Ph.D.)--York University (Canada).

\bibitem{HrusakEstrellitas}
Michael Hru\v{s}\'{a}k and Carlos Mart\'{\i}nez~Ranero.
\newblock Some remarks on non-special coherent {A}ronszajn trees.
\newblock {\em Acta Univ. Carolin. Math. Phys.}, 46(2):33--40, 2005.

\bibitem{MultipleForcing}
T.~Jech.
\newblock {\em Multiple forcing}, volume~88 of {\em Cambridge Tracts in
  Mathematics}.
\newblock Cambridge University Press, Cambridge, 1986.

\bibitem{Kechris}
Alexander~S. Kechris.
\newblock {\em Classical descriptive set theory}, volume 156 of {\em Graduate
  Texts in Mathematics}.
\newblock Springer-Verlag, New York, 1995.

\bibitem{oldKunen}
Kenneth Kunen.
\newblock {\em Set theory}, volume 102 of {\em Studies in Logic and the
  Foundations of Mathematics}.
\newblock North-Holland Publishing Co., Amsterdam-New York, 1980.
\newblock An introduction to independence proofs.

\bibitem{Kunen}
Kenneth Kunen.
\newblock {\em Set theory}, volume~34 of {\em Studies in Logic (London)}.
\newblock College Publications, London, 2011.

\bibitem{KnasterandFriendsI}
Chris Lambie-Hanson and Assaf Rinot.
\newblock Knaster and friends {I}: closed colorings and precalibers.
\newblock {\em Algebra Universalis}, 79(4):Paper No. 90, 39, 2018.

\bibitem{KomjathInaccesible}
Hossein Lamei~Ramandi and Stevo Todorcevic.
\newblock Can you take {K}omjath's inaccessible away?
\newblock {\em Ann. Pure Appl. Logic}, 175(7):Paper No. 103452, 24, 2024.

\bibitem{forcingwithoutcombinatorics}
Alan~H. Mekler.
\newblock C.c.c. forcing without combinatorics.
\newblock {\em J. Symbolic Logic}, 49(3):830--832, 1984.

\bibitem{GroupsfromWalks}
Stepan Milovsevi\'c and Stevo Todorcevic.
\newblock Groups from walks on countable ordinals.
\newblock {\em Topology Appl.}, 373:Paper No. 109480, 13, 2025.

\bibitem{FiveElementBasis}
Justin~Tatch Moore.
\newblock A five element basis for the uncountable linear orders.
\newblock {\em Ann. of Math. (2)}, 163(2):669--688, 2006.

\bibitem{Lspace}
Justin~Tatch Moore.
\newblock A solution to the {$L$} space problem.
\newblock {\em J. Amer. Math. Soc.}, 19(3):717--736, 2006.

\bibitem{LSpacewithseparablesquare}
Justin~Tatch Moore.
\newblock An {L} space with a {$d$}-separable square.
\newblock {\em Topology Appl.}, 155(4):304--307, 2008.

\bibitem{LindeloffGroup}
Yinhe Peng and Liuzhen Wu.
\newblock A {L}indel\"of group with non-{L}indel\"of square.
\newblock {\em Adv. Math.}, 325:215--242, 2018.

\bibitem{LVectorSpaces}
Yinhe Peng and Liuzhen Wu.
\newblock L vector spaces and {L} fields.
\newblock {\em Sci. China Math.}, 67(10):2195--2216, 2024.

\bibitem{CombinatorialPropertyofpfunctions}
D.~Raghavan and S.~Todorcevic.
\newblock A combinatorial property of rho-functions.
\newblock {\em Acta Math. Hungar.}, 167(1):355--363, 2022.

\bibitem{RectangularSquareBracket}
Assaf Rinot and Stevo Todorcevic.
\newblock Rectangular square-bracket operation for successor of regular
  cardinals.
\newblock {\em Fund. Math.}, 220(2):119--128, 2013.

\bibitem{ComplicatedColoringsRevisited}
Assaf Rinot and Jing Zhang.
\newblock Complicated colorings, revisited.
\newblock {\em Ann. Pure Appl. Logic}, 174(4):Paper No. 103243, 11, 2023.

\bibitem{WalksandselectiveUltrafilters}
S.~Todorcevic.
\newblock Walks on countable ordinals and selective ultrafilters.
\newblock {\em Bull. Cl. Sci. Math. Nat. Sci. Math.}, (35):65--78, 2010.

\bibitem{PartitioningPairs}
Stevo Todorcevic.
\newblock Partitioning pairs of countable ordinals.
\newblock {\em Acta Math.}, 159(3-4):261--294, 1987.

\bibitem{PartitionProblems}
Stevo Todor{\v{c}}evi{\'c}.
\newblock {\em Partition problems in topology}, volume~84 of {\em Contemporary
  Mathematics}.
\newblock American Mathematical Society, Providence, RI, 1989.

\bibitem{Walks}
Stevo Todorcevic.
\newblock {\em Walks on ordinals and their characteristics}, volume 263 of {\em
  Progress in Mathematics}.
\newblock Birkh\"{a}user Verlag, Basel, 2007.

\bibitem{NotesonForcingAxioms}
Stevo Todorcevic.
\newblock {\em Notes on forcing axioms}, volume~26 of {\em Lecture Notes
  Series. Institute for Mathematical Sciences. National University of
  Singapore}.
\newblock World Scientific Publishing Co. Pte. Ltd., Hackensack, NJ, 2014.

\bibitem{PIDandTukey}
Stevo Todorcevic.
\newblock P-ideal dichotomy and {T}ukey order.
\newblock {\em Fund. Math.}, 267(2):181--194, 2024.

\bibitem{StevoHandbook}
S.~Todor\v{c}evi\'{c}.
\newblock Trees and linearly ordered sets.
\newblock In {\em Handbook of set-theoretic topology}, pages 235--293.
  North-Holland, Amsterdam, 1984.

\bibitem{5CofinalTypes}
Stevo Todor\v{c}evi\'{c}.
\newblock Directed sets and cofinal types.
\newblock {\em Trans. Amer. Math. Soc.}, 290(2):711--723, 1985.

\bibitem{ADichotomyforPIdealsofCountableSets}
Stevo Todor\v{c}evi\'{c}.
\newblock A dichotomy for {P}-ideals of countable sets.
\newblock {\em Fund. Math.}, 166(3):251--267, 2000.

\bibitem{PmaxBook}
W.~Hugh Woodin.
\newblock {\em The axiom of determinacy, forcing axioms, and the nonstationary
  ideal}, volume~1 of {\em De Gruyter Series in Logic and its Applications}.
\newblock Walter de Gruyter GmbH \& Co. KG, Berlin, revised edition, 2010.

\bibitem{RectangleRefiningProperty}
Teruyuki Yorioka.
\newblock Some weak fragments of {M}artin's axiom related to the rectangle
  refining property.
\newblock {\em Arch. Math. Logic}, 47(1):79--90, 2008.

\bibitem{CharacterizationofClubForcing}
Jindrich Zapletal.
\newblock Characterization of the club forcing.
\newblock In {\em Papers on general topology and applications ({G}orham, {ME},
  1995)}, volume 806 of {\em Ann. New York Acad. Sci.}, pages 476--484. New
  York Acad. Sci., New York, 1996.

\bibitem{CharacterizationofDefinableForcings}
Jind\v{r}ich Zapletal.
\newblock A classification of definable forcings on {$\omega_1$}.
\newblock {\em Fund. Math.}, 153(2):141--144, 1997.

\end{thebibliography}
\end{document}